\documentclass[11pt]{amsart}

\usepackage{amsfonts,amsthm,amsmath,enumerate,amssymb,latexsym,color,tcolorbox,tikz,mathrsfs,bm}
\usepackage{subfig}
\usetikzlibrary{patterns}
\usepackage[text={6in,9in},centering]{geometry}
\usepackage[colorlinks,
            linkcolor=blue,
            anchorcolor=red,
            citecolor=red]{hyperref}
\usepackage{graphicx}
\usepackage{epstopdf}
\usepackage{color}

\graphicspath{{figures/}}

\theoremstyle{plain}
\newtheorem{theorem}{Theorem}[section]
\newtheorem{lemma}[theorem]{Lemma}
\newtheorem{corollary}[theorem]{Corollary}
\newtheorem{proposition}[theorem]{Proposition}

\theoremstyle{definition}
\newtheorem{definition}[theorem]{Definition}
\newtheorem{example}[theorem]{Example}

\theoremstyle{remark}
\newtheorem{remark}[theorem]{Remark}

\numberwithin{equation}{section}

\DeclareMathOperator{\dist}{dist}
\DeclareMathOperator{\diam}{diam}

\def\R{\mathbb R}
\def\Z{\mathbb Z}
\def\D{\mathcal D}
\def\E{\mathcal E}

\def\i{\bm i}
\def\j{\bm j}
\def\k{\bm k}
\def\bi{\mathbf i}
\def\bj{\mathbf j}
\def\bp{\mathbf p}
\def\bq{\mathbf q}
\def\bw{\mathbf w}
\def\bu{\mathbf u}

\def\vp{\varphi}

\makeatletter
\def\l@subsection{\@tocline{2}{0pt}{2.5pc}{5pc}{}}
\def\l@subsubsection{\@tocline{2}{0pt}{5pc}{7.5pc}{}}
\makeatother

\begin{document}

\title[Connectedness and local cut points of GSCs]{Connectedness and local cut points of \\ generalized Sierpi\'nski carpets}

\author{Xin-Rong Dai}
\address{School of Mathematics, Sun Yat-sen University, Guangzhou, China}
\email{daixr@mail.sysu.edu.cn}

\author{Jun Luo}
\address{School of Mathematics, Sun Yat-sen University, Guangzhou, China}
\email{luojun3@mail.sysu.edu.cn}

\author{Huo-Jun Ruan}
\address{School of Mathematical Sciences, Zhejiang University, Hangzhou, China}
\email{ruanhj@zju.edu.cn}

\author{Yang Wang}
\address{Department of Mathematics, The Hong Kong University of Science and Technology, Clear Water Bay, Kowloon, Hong Kong}
\email{yangwang@ust.hk}

\author{Jian-Ci Xiao}
\address{School of Mathematical Sciences, Zhejiang University, Hangzhou, China}
\email{jcxshaw24@zju.edu.cn}



\thanks{Corresponding author: Huo-Jun Ruan}

\maketitle


\begin{abstract}
    We investigate a  homeomorphism problem on a class of self-similar sets called generalized Sierpi\'nski carpets (or shortly GSCs). It follows from two well-known results by Hata and Whyburn that a connected GSC is homeomorphic to the standard Sierpi\'nski carpet if and only if it has no local cut points.
    On the one hand, we show that to determine whether a given GSC is connected, it suffices to iterate the initial pattern twice.
    On the other hand, we obtain two criteria: (1)  for a connected GSC to have cut points, (2) for a connected GSC with no cut points to have local cut points. With these two criteria, we characterize all GSCs that are homeomorphic to the standard Sierpi\'nski carpet.

    Our results on cut points and local cut points hold for Bara\'nski carpets, too. Moreover, we extend the connectedness result to Bara\'nski sponges. Thus, we also characterize when a Bara\'nski carpet is homeomorphic to the standard GSC.
\end{abstract}

\tableofcontents

\section{Introduction}

An iterated function system, or shortly an IFS, is a family $\Phi=\{\varphi_1,\ldots,\varphi_m\}$ of contractions on the Euclidean space $\R^d$. A well-known theorem of Hutchinson~\cite{Hut81} tells us that there is a unique non-empty compact set $K\subset\R^d$ such that $K=\bigcup_{i=1}^m \varphi_i(K)$. The set $K$ is usually called the \emph{attractor} of the IFS $\Phi$. When the IFS consists of contracting similarity (resp. affine) maps, the attractor is called a \emph{self-similar} (resp. \emph{self-affine}) set.

There is a growing body of literature that studies the topology of attractors of given IFSs, especially self-similar or self-affine sets, in the last three decades. In~\cite{Hat85}, Hata explored topological properties of attractors of general IFSs including connectedness, path connectedness, local connectedness, end points, etc. Luo, Rao and Tan~\cite{LRT02} studied the interior and boundary of planar self-similar sets generated by an IFS consisting of injective contractions and satisfying the open set condition. Hare and Sidorov~\cite{HS15} provided a detailed analysis on when a class of self-affine sets have non-empty interiors. For further work, please refer to~\cite{BK91, Kig01, LL17, Luo19}. There are also a number of researches on basic topological properties of self-similar or self-affine tiles such as~\cite{DLN18, DL11, LL07, NT05, TZ20}.

In general, the topology of a given attractor can be quite complicated. A natural and perhaps one of the simplest question is: \emph{When is the attractor connected}? In \cite{Hat85}, Hata proved the following criterion: The connectedness of the attractor and the associated Hata graph are equivalent.

\begin{definition}[\cite{Hat85}]\label{def:hata}
    Let $\Phi=\{\varphi_1,\ldots,\varphi_m\}$ be an IFS on $\R^d$ and let $K$ be its attractor. The \emph{Hata graph} associated with $\Phi$ is defined as follows. The vertex set is the index set $\Lambda:=\{1,\ldots,m\}$, and there is an edge joining $i,j\in\Lambda$ ($i\neq j$) if and only if $\varphi_i(K)\cap\varphi_j(K)\neq\varnothing$.
\end{definition}

Another one of the most fundamental question concerning the topology of fractal sets is: \emph{when are two given fractal sets homeomorphic?} Generally this is a challenging problem and there are few existing results. A pioneer theorem was provided by Whyburn in 1958. Recall that given a connected space $X$, a point $x\in X$ is called a \emph{cut point} of $X$ if $X\setminus\{x\}$ is no longer connected, and is called a \emph{local cut point} of $X$ if $x$ is the cut point of some connected neighborhood of itself.

\begin{theorem}[\cite{Why58}]\label{thm:Whyburn}
    A metric space is homeomorphic to the standard Sierpi\'nski carpet if and only if it is a planar continuum of topological dimension 1 that is locally connected and has no local cut points.
\end{theorem}
Here a set $X$ is called \emph{planar} if it is homeomorphic to a subset of $\R^2$.

In this paper, we will focus on one of the most classical classes of self-similar sets called  \emph{generalized Sierpi\'nski carpets}. Let $N\geq 2$ and let $\D\subset\{0,1,\ldots,N-1\}^2$ be a non-empty digit set with $1<|\D|<N^2$ (to avoid trivial cases), where $|\D|$ denotes the cardinality of $\D$. For each $i\in\{0,1,\ldots,N-1\}^2$, define a similarity map $\vp_i$ by
\[
    \vp_i(x)=\frac{1}{N}(x+i), \quad x\in\R^2.
\]
We call the self-similar attractor $F=F(N,\D)$ of the IFS $\{\vp_i:i\in\D\}$ a \emph{generalized Sierpi\'nski carpet} (or shortly a GSC). In some recent papers such as \cite {LLR13, Roi10, RW17}, $F(N,\D)$ is also called a \emph{fractal square}. Figure~\ref{fig:class} depicts the standard Sierpi\'nski carpet where $N=3$ and $\D=\{0,1,2\}^2\setminus\{(1,1)\}$.
It is well-known that  GSCs can also be generated in a geometric way by first dividing the unit square $[0, 1]^2$ into an $N\times N$ grid, selecting a subset of squares formed by the grid (usually called the \emph{initial pattern}) and then repeatedly substituting the initial pattern on each of the selected square. For convenience, we write $Q_0=[0,1]^2$ and recursively define
\begin{equation}\label{eq:qn}
    Q_n = \bigcup_{i\in\D} \vp_i(Q_{n-1}), \quad n\in\Z^+.
\end{equation}
In particular, $Q_1$ is just the initial pattern. Note that $\{Q_n\}_{n=0}^\infty$ forms a decreasing sequence of compact sets and $F=\bigcap_{n=0}^\infty Q_n$.
\begin{figure}[htbp]
    \centering
    \includegraphics[width=4.3cm]{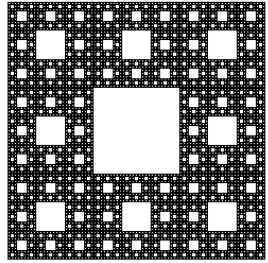}
    \caption{The standard Sierpi\'nski carpet}
    \label{fig:class}
\end{figure}

There are many works related with topological properties of GSCs in various fields of fractals, including analysis on fractals \cite{BB89,BBKT10,KuZh92}, quasisymmetric equivalence \cite{BM13, BM20},  Lipschitz equivalence \cite{LMR20,RW17,Xi20,XiXiong10} and the behaviour of their connected components \cite{DT20,LLR13,XiXiong10,Xiao21}. In particular, Lau et al. provided a characterization together with a checkable algorithm on totally disconnected GSCs in~\cite{LLR13}.

The main goal of this paper attempts to determine when a given GSC is homeomorphic to the standard Sierpi\'nski carpet. Note that a connected GSC is always a planar continuum (i.e., connected and compact) of topological dimension $1$, and a result in Hata~\cite{Hat85} tells us that it is locally connected (one can also see~\cite[Proposition 1.6.4]{Kig01} for a proof). Thus a GSC is homeomorphic to the standard Sierpi\'nski carpet if and only if it is connected and has no local cut points. For the connectedness, we find an effective criterion based on the Hata's criterion. For the existence of local cut points, we first characterize the existence of cut points of connected GSCs, and then turn to those who have no cut points but local cut points.


We remark that it might be of independent interest to study the existence of cut points of connected GSCs. For example, cut points play an important role in the bi-Lipschitz classification of GSCs in a recent work~\cite{RW17} of Ruan and Wang. One can also refer to~\cite{ADTW09, DTW13, LLST16} for relevant studies on cut points of fractal sets.

Before stating our results, let us introduce several notations. As a matter of convenience, we regard the digit set $\D$ as the index set of the IFS $\{\vp_i:i\in\D\}$ instead of enumerating it by $\vp_1,\ldots,\vp_{|\D|}$, where we denote by $|A|$ the cardinality of a set $A$. Under this setting, we list the following commonly used notations.
\begin{enumerate}
    \item For $k\in\Z^+$, $\D^k:=\{\i=i_1\cdots i_k: i_1,\ldots,i_k\in\D\}$. Let $\D^0=\{\vartheta\}$. We call $\vartheta$ the empty word. For $k\geq 0$ and $\i\in \D^k$, we call $\i$ a word of length $|\i|:=k$.
    \item $\D^*:=\bigcup_{k=1}^\infty\D^k$ and $\D^\infty:=\{i_1i_2\cdots: i_j\in\D \text{ for all } j\in\Z^+\}$.
    \item For $1\leq k\leq n$ and $\i=i_1\cdots i_n\in\D^n$, $\i|_k:=i_1\cdots i_k$ stands for its prefix of length $k$. For $\bi\in\D^\infty$ and $k\geq 1$, $\bi|_k$ is similarly defined.
    \item For $\i,\j\in\D^*$, we write $\i\prec\j$ if $\i$ is a prefix of $\j$.
    \item For $\i\in\D^*$ and $k\in\Z^+$, $\i\D^k:=\{\i\j: \j\in\D^k\}$.
    \item For $k\geq 1$ and $\i=i_1\cdots i_k\in\D^k$, $\vp_{\i}:=\vp_{i_1}\circ\cdots \circ\vp_{i_k}$. Define $\vp_{\vartheta}$ to be the identity map. Denote by  $\vp_{\j}^k$ the $k$-fold composition of $\vp_{\j}$ for all $\j\in \D^*$.
\end{enumerate}

It turns out that the existence of cut points of GSCs is closely related to structures of a sequence of associated ``Hata graphs''.

\begin{definition}
    Let $F=F(N,\D)$ be a GSC. For $k\geq 1$, the $k$-th Hata graph $\Gamma_k=\Gamma_k(N,\D)$ of $F$ is defined as follows. The vertex set is $\D^k$, and there is an edge joining $\i,\j\in\D^k$ ($\i\neq\j$) if and only if $\vp_{\i}(F)\cap \vp_{\j}(F)\neq\varnothing$.
\end{definition}

The Hata graph sequence has been used in previous literature. Please see \cite{Kig01,LLST16,Str06,TZ20} for examples. It is clear that the Hata graph defined in Definition~\ref{def:hata} is just $\Gamma_1$. Moreover, since $F=\bigcup_{\i\in \D^k} \vp_{\i}(F)$, if $F$ is connected then $\Gamma_k$ is connected for all $k$.


Given a graph $G=(V,E)$ and a vertex $v\in V$, we denote by $G-\{v\}$ the subgraph of $G$ such that its vertex set is $V\setminus \{v\}$, and its edge set is the subset of $E$ by deleting all edges incident with $v$. We call $v$ a \emph{cut vertex} of $G$ if the subgraph $G-\{v\}$ is disconnected.

\begin{definition}\label{def:chi}
    Let $G=(V,E)$ be a connected graph.  Given a cut vertex $v\in V$ of $G$, let $G_1(v),\ldots,G_m(v)$ be all connected components of $G-\{v\}$ with $|G_1(v)|\geq |G_2(v)|\geq \cdots \geq |G_m(v)|$, where $|G_i(v)|$ is the number of vertices in $G_i(v)$, $1\leq i\leq m$. We define
    \[
        \chi(G) = \max\{|G_2(v)|:\, \textrm{$v$ is a cut vertex of $G$}\}
    \]
    if $G$ has cut vertices, and $\chi(G)=0$ if $G$ has no cut vertices.
\end{definition}

The following two theorems are main results of this paper.

\begin{theorem}\label{thm:main}
    A connected GSC $F=F(N,\D)$ has cut points if and only if $\chi(\Gamma_k)\geq|\D|^{k-1}-1$ for all $k\geq 2$.
\end{theorem}

\begin{theorem}\label{thm:local}
    Let $F=F(N,\D)$ be a connected GSC with no cut points. Then $F$ contains local cut points if and only if there are disjoint subsets $I,J$ of $\D^n$ for $n=1$ or $n=2$ such that the following conditions hold:
    \begin{enumerate}
        \item $ \big( \bigcup_{\i\in I}\vp_{\i}(F) \big)\cap \big( \bigcup_{\i\in J}\vp_{\i}(F) \big)$ is a singleton, denoted by $\{x\}$;
        \item $I\cup J=\{\i\in\D^n: x\in\vp_{\i}(F)\}$.
    \end{enumerate}
\end{theorem}

We remark that Theorem~\ref{thm:main} can be improved as follows: A connected GSC $F=F(N,\D)$ has cut points if and only if $\chi(\Gamma_k)\geq|\D|^{k-1}-1$ for all $2\leq k\leq 3^8+4$. However, the proof is very technical so we decide to present it in another paper~\cite{RWXpre}.

\begin{example}\label{ex:nocutpt-1}
    Let $F=F(6,\D)$ be the GSC as in Figure~\ref{fig:nocutpt-1}. By using the algorithms in Section~2 or Section~6, one can easily draw the associated Hata graphs and see that $\chi(\Gamma_1)=10$ and $\chi(\Gamma_2)=\chi(\Gamma_3)=35$. So $\chi(\Gamma_3)<25^2-1=|\D|^2-1$ and hence it has no cut points.
 \end{example}

\begin{figure}[htbp]
    \centering
    \includegraphics[width=4.5cm]{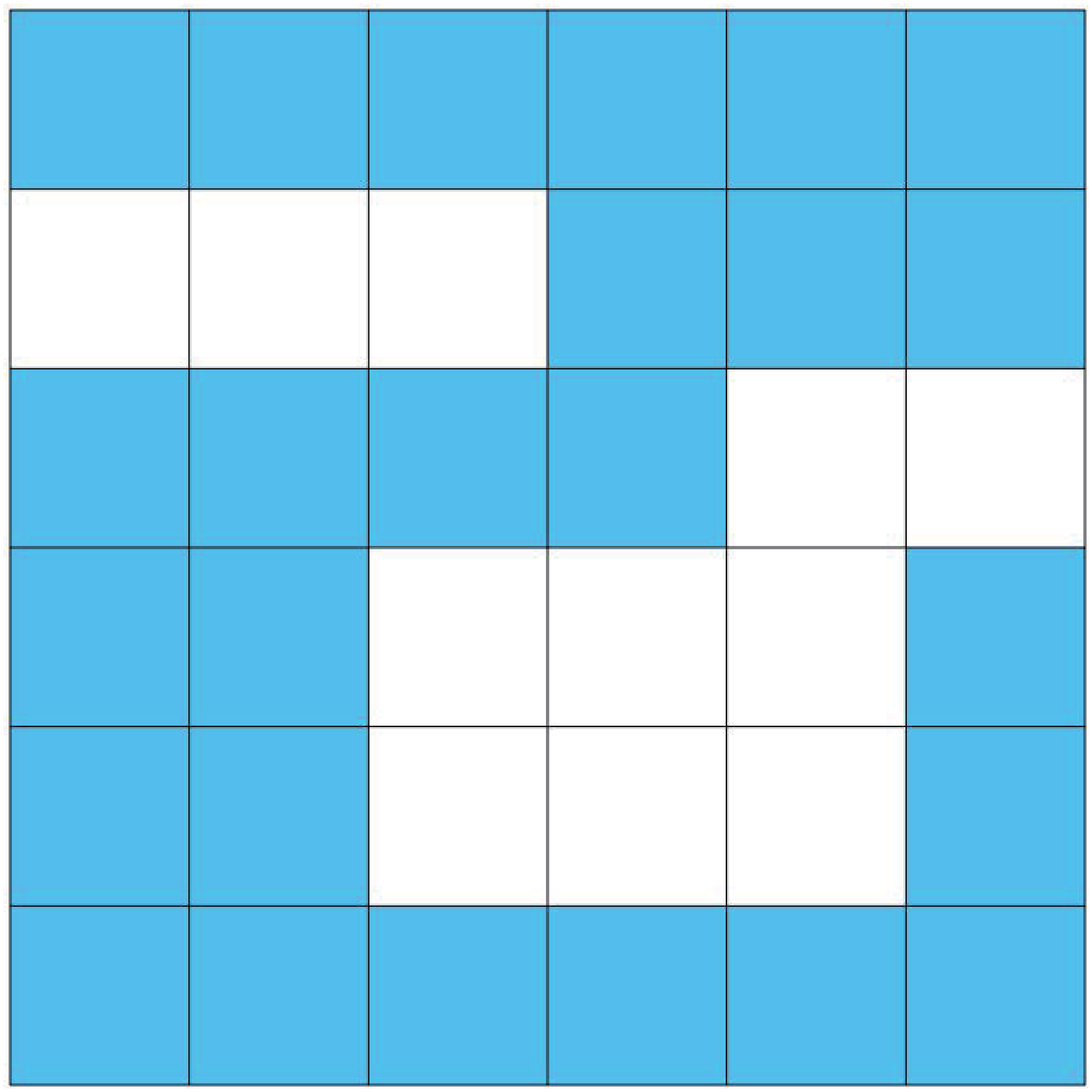} \quad
    \includegraphics[width=4.5cm]{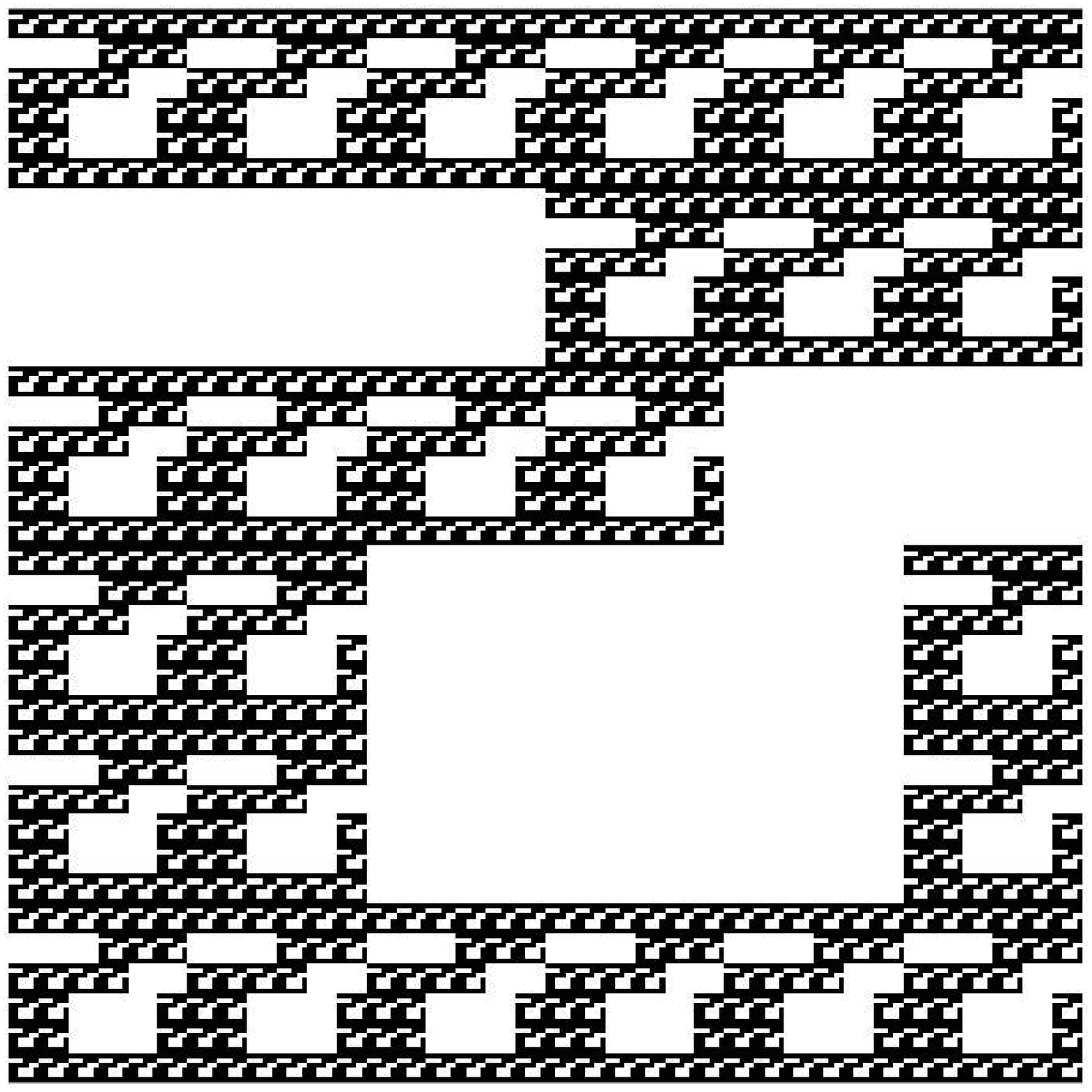} \quad
    \includegraphics[width=4.5cm]{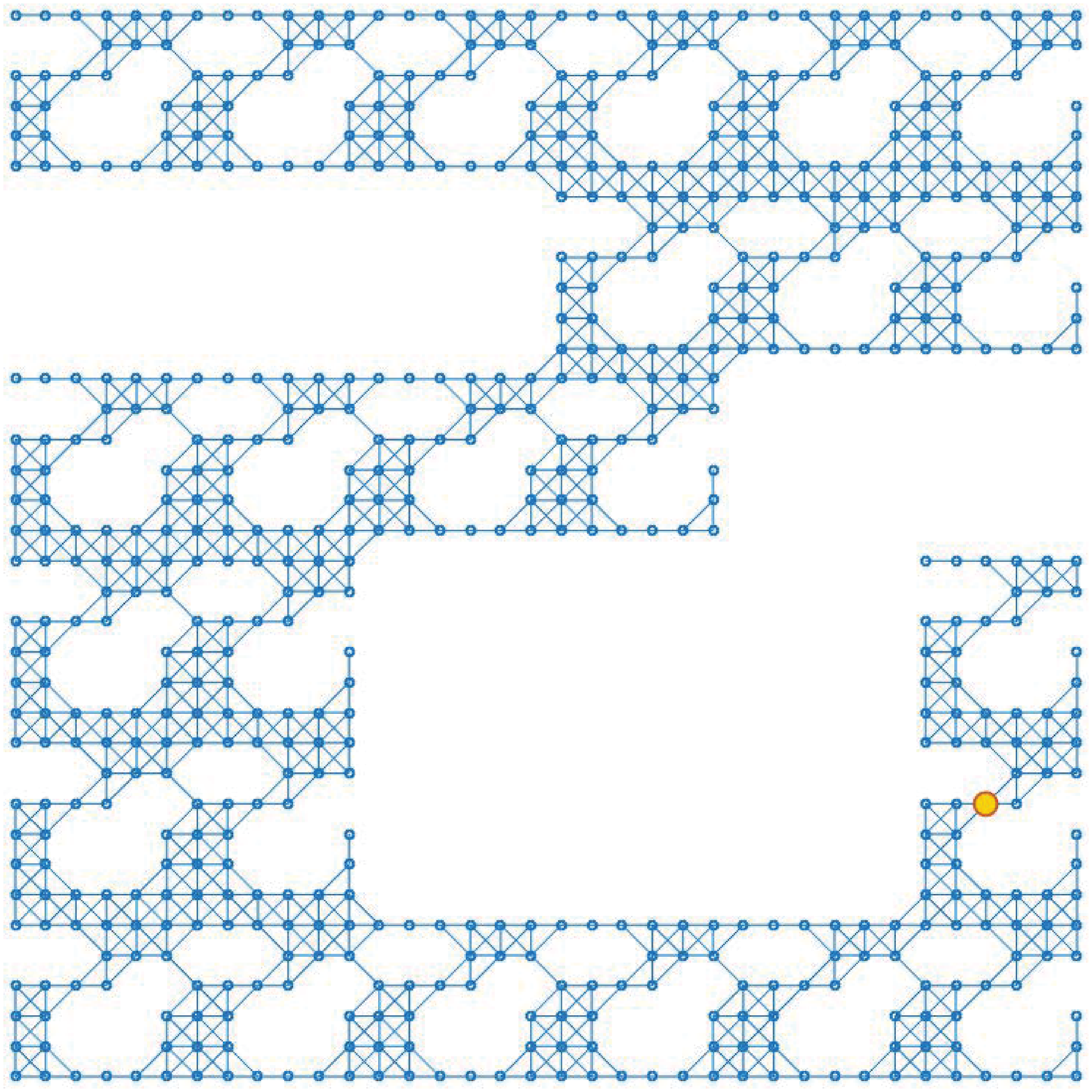}
    \caption{A GSC with no cut point. From left to right: the initial pattern, the GSC and the 2-nd Hata graph.}
    \label{fig:nocutpt-1}
\end{figure}

\begin{example}\label{ex:fragile}
    Let $F=F(4,\D)$ be the GSC as in Figure~\ref{fig:gamma1+2}. Write $i=(3,0)$ and $\omega_{k}=(3,0)\underbrace{(0,3)\cdots(0,3)}_{(k-1) \text{ terms}}$. It is not hard to see that
    \begin{equation*}
        \chi(\Gamma_k) = |i\D^{k-1}\setminus\{\omega_{k}\}| = |\D|^{k-1}-1, \quad k\geq 2,
    \end{equation*}
    and $(\frac{3}{4},\frac{1}{4})$, which is the unique element in the singleton $\bigcap_{k=1}^\infty \vp_{\omega_k}([0,1]^2)$, is a cut point of $F$.  For $N\geq 5$, such examples can be constructed similarly. In particular, the lower bound in Theorem~\ref{thm:main} is sharp.
 \end{example}

\begin{figure}[htbp]
    \centering
    \subfloat[The GSC]
    {
        \begin{minipage}[t]{120pt}
        \centering
        \includegraphics[height=3.8cm]{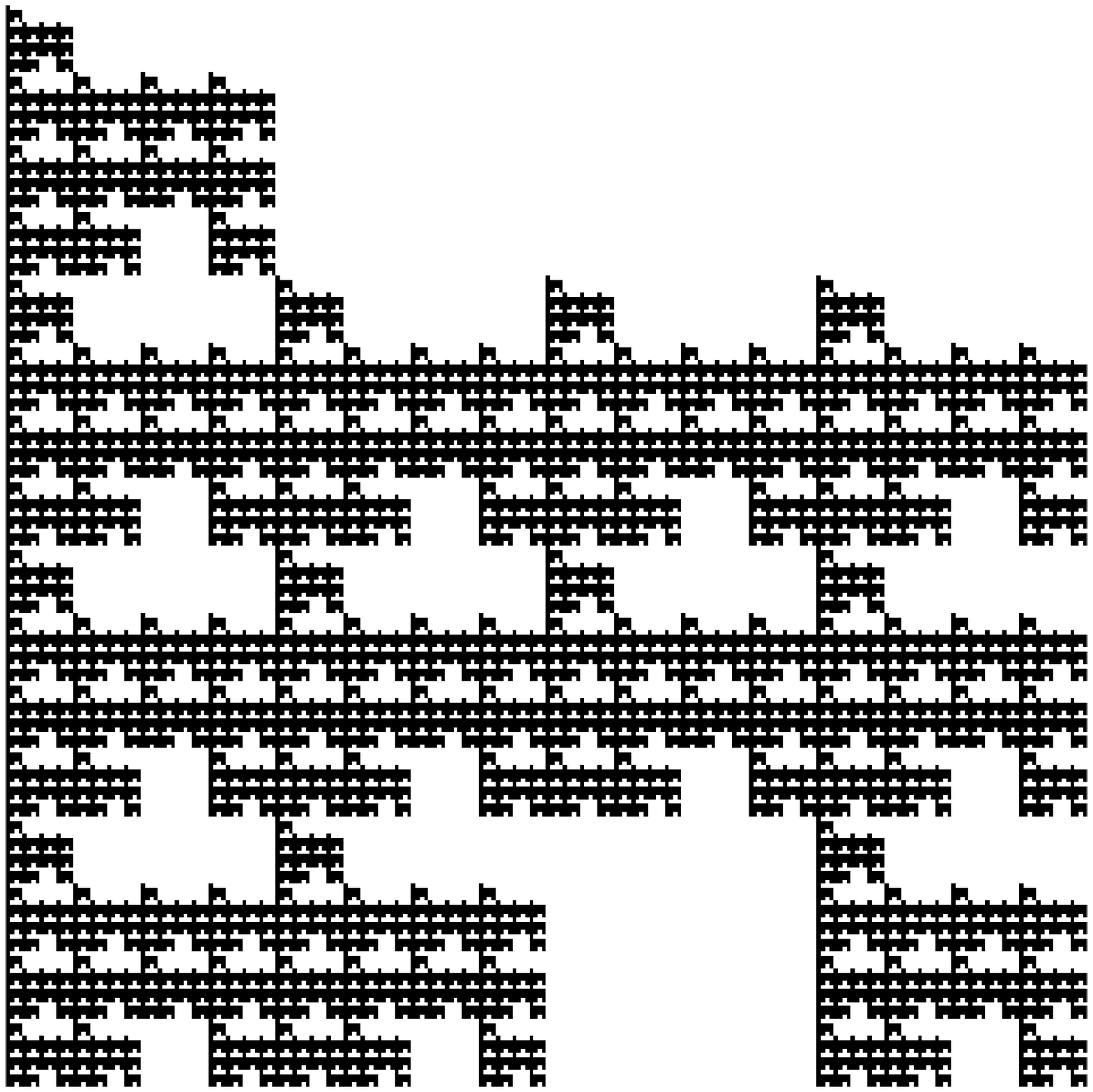}
      \end{minipage}
    }
    \subfloat[$1$-st Hata graph]
    {
        \begin{minipage}[t]{120pt}
        \centering
        \includegraphics[height=3.8cm]{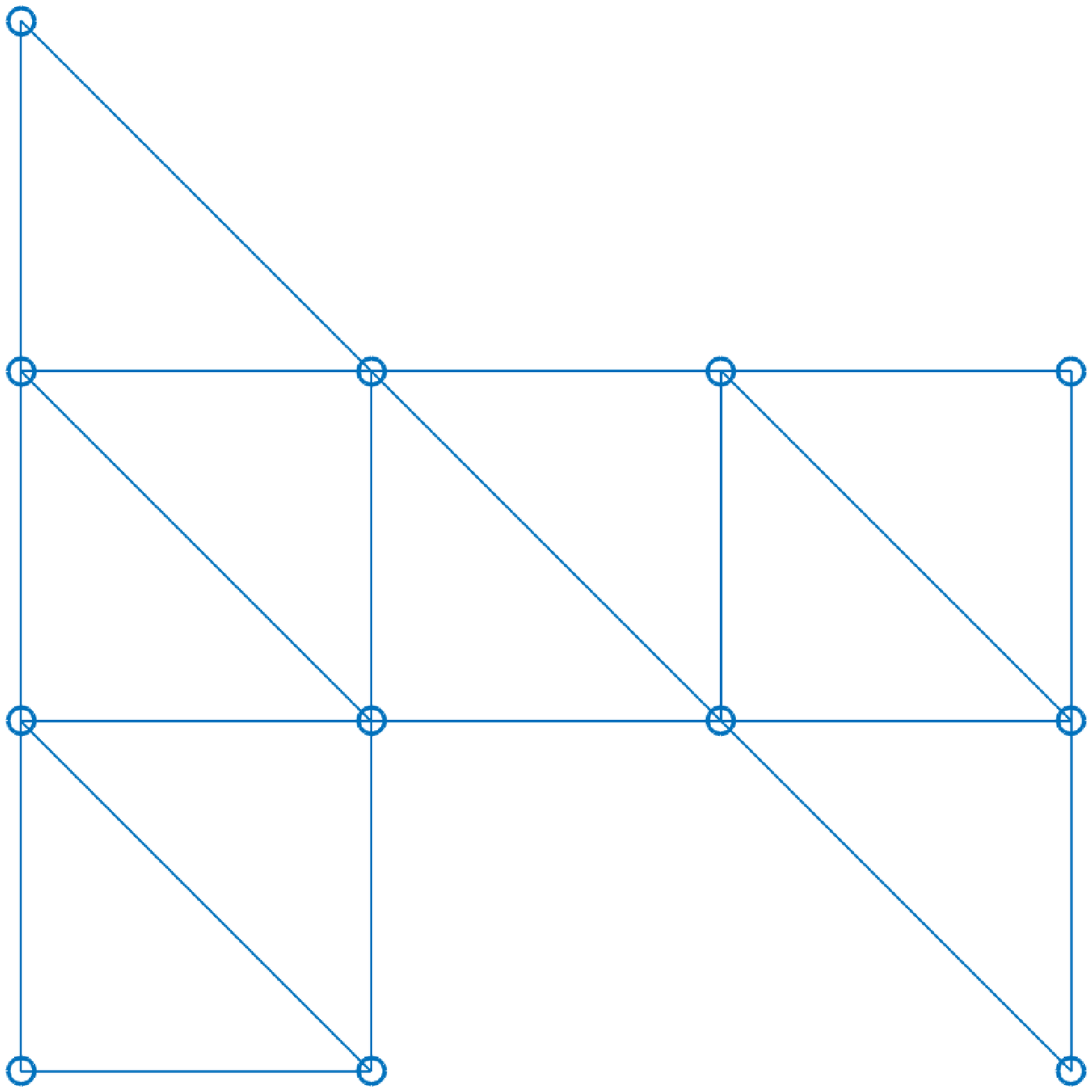}
        \end{minipage}
    }
    \subfloat[$2$-nd Hata graph]
    {
        \begin{minipage}[t]{120pt}
        \centering
        \includegraphics[height=3.8cm]{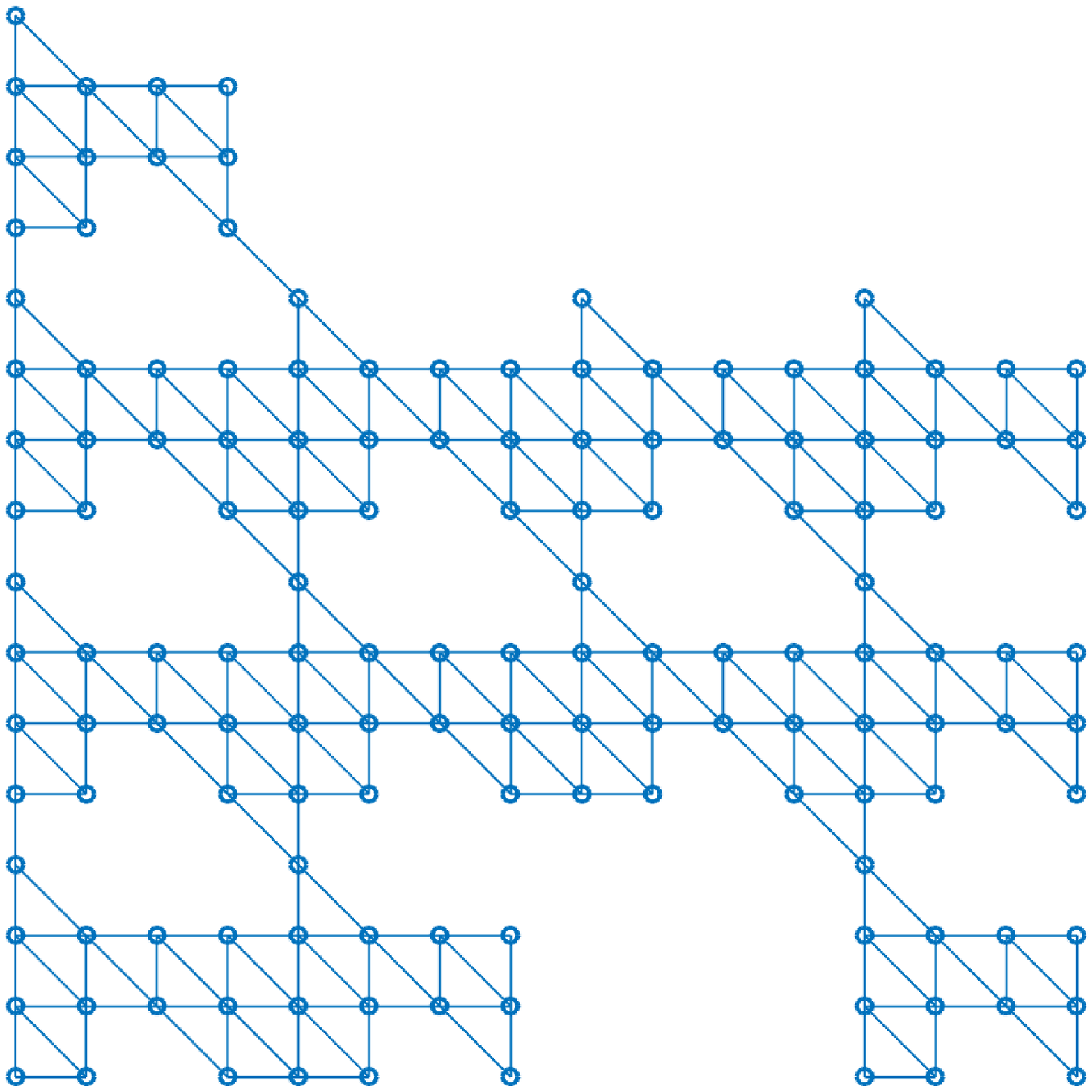}
        \end{minipage}
    }
    \caption{A GSC with $\chi(\Gamma_k)= |\D|^{k-1}-1$ for all $k\geq 2$}
    \label{fig:gamma1+2}
\end{figure}

As a direct consequence of Theorem~\ref{thm:main}, we have the following easily checked sufficient condition on connected GSCs with no cut points.


\begin{corollary}\label{cor:cutpt}
    A connected GSC $F(N,\D)$ has no cut points if $\chi(\Gamma_2)< |\D|-1$.
\end{corollary}

For finitely ramified fractals, Bandt and Retta showed that if $\Gamma_2$ has no cut vertices, then the corresponding fractal has no cut points (see \cite[Proposition~2.1]{BR92}). Our corollary is a more general result in the GSC setting. It is also noteworthy that there does exist some connected GSC (e.g., the one in Figure~\ref{fig:gamma1+2}) such that $\Gamma_1$ has no cut vertices while $F(N,\D)$ has cut points.



Now let us explain the rough idea to prove Theorem~\ref{thm:main}. To determine the existence of cut points, we divide the collection of all connected GSCs into the following two types and treat them in different ways.

\begin{definition}
    A connected GSC $F=F(N,\D)$ is called \emph{fragile} if there is a decomposition of $\D$, say $\D=\D_1\cup\D_2$ with $\D_1\cap\D_2=\varnothing$, such that the intersection
    \begin{equation}\label{eq:frag-decom}
        \Big( \bigcup_{i\in\D_1}\vp_i(F) \Big) \cap \Big( \bigcup_{i\in\D_2}\vp_i(F) \Big)
    \end{equation}
    is a singleton. The GSC $F$ is called \emph{non-fragile} if it is not fragile.
\end{definition}

For example, let $F=F(4,\D)$ be the GSC as in Figure~\ref{fig:gamma1+2}. Writing $\D_1=\{(3,0)\}$ and $\D_2=\D\setminus\D_1$,
\[
    \Big( \bigcup_{i\in\D_1}\vp_i(F) \Big) \cap \Big( \bigcup_{i\in\D_2}\vp_i(F) \Big) = \{(3/4,1/4)\}
\]
and hence $F$ is fragile. It is easy to see that a fragile GSC always has cut points. More precisely, the unique point in the singleton in \eqref{eq:frag-decom} should be a cut point (see Section 3). On the other hand, the lower bound in Theorem~\ref{thm:main} is not hard to show for fragile GSCs (see Proposition~\ref{prop:fragilelower}).


Things become much more challenging in the non-fragile case. Fortunately, by a delicate analysis, we have the following theorem.

\begin{theorem}\label{thm:nonfragile}
    A non-fragile connected GSC $F=F(N,\D)$ has cut points if and only if $\chi(\Gamma_k)\geq|\D|^{k-1}$ for all $k\geq 2$. Moreover, if there exists $k_0\geq 2$ such that $\chi(\Gamma_k)\geq|\D|^{k-1}$ for all $k\geq k_0$, then $F$ has cut points.
\end{theorem}

Combining the results in fragile and non-fragile cases, we can prove Theorem~\ref{thm:main}.

\begin{proof}[Proof of Theorem~\ref{thm:main}]
     The ``only if\,'' part follows from Proposition~\ref{prop:fragilelower} and Theorem~\ref{thm:nonfragile} directly. For the ``if\,'' part, since a fragile connected GSC has cut points, we can assume by Proposition~\ref{prop:frag-suff} that $F$ is non-fragile and $\chi(\Gamma_k)\geq|\D|^{k-1}$ for all $k\geq 3$. Then we can see by Theorem~\ref{thm:nonfragile} that $F$ has cut points.
\end{proof}

For the connectedness problem, we prove the following interesting result.  Recall that Bara\'nski sponges serve as a self-affine and higher dimensional generalization of GSCs. For a rigorous definition, please see Section 6.

\begin{theorem}\label{thm:main2}
    A Bara\'nski sponge $K$ in $\R^d$ is connected if and only if $Q_{d+1}$ is connected.
\end{theorem}


Notice that both of the criteria on the connectedness and the existence of cut points are based on the associated Hata graph sequence. In this paper, we also provide an effective method to draw the Hata graph sequence of GSCs and Bara\'nski sponges (please see Section 2 and Theorem~\ref{thm:khataofbaranski}).

Let us remark here that in an earlier paper~\cite{CS10}, Cristea and Steinsky studied the connectedness problem of GSCs under a more general setting: The pattern of selected squares is allowed to change during the iteration process. They constructed a sequence of graphs and showed that the limit set is connected if and only if every graph in this sequence is connected. However, they did not tell us how many iterations we need to determine the connectedness of the carpet when there is only one pattern.


Finally, we investigate the cardinality of the digit set of a given connected GSC. A better lower bound of $\chi(\Gamma_k)$ is also presented for non-fragile connected GSCs with cut points.

\begin{theorem}\label{thm:gap}
    Let $N\geq 2$ and let $F=F(N,\D)$ be a connected GSC with $|\D|>N$. Then $|\D|\geq 2N-1$. Conversely, for every integer $k\geq 2N-1$ or $k=N$, there is a digit set $\D$ with $|\D|=k$ such that $F=F(N,\D)$ is connected.
\end{theorem}

\begin{proposition}\label{prop:chi-betterBd}
  Let $F=F(N,\D)$ be a non-fragile connected GSC with cut points. Then $\chi(\Gamma_k)\geq |\D|^{k-1}+|\D|^{k-3}$ for $k\geq 3$.
\end{proposition}

This paper is organized as follows. In Section 2, we give a characterization of connected GSCs. In Section 3, we prove the sharp lower bound estimate of $\chi(\Gamma_k)$ for the fragile case. Furthermore, we obtain several sufficient conditions for a connected GSC to be fragile. Sections 4 is devoted to treating non-fragile case and proving Theorem~\ref{thm:nonfragile}. In Section~5, we turn to local cut points and prove Theorem~\ref{thm:local}. In Section~6, we explain that our results hold for Bara\'nski carpets. In particular, we discuss the connectedness of Bara\'nski sponges and prove Theorem~\ref{thm:main2}. We also present an effective method to draw the associated Hata graph sequence. The proofs of Theorem~\ref{thm:gap} and Proposition~\ref{prop:chi-betterBd} are given in Section 7.

\section{Connectedness of GSCs}

In this section, we deal with the connectedness problem of GSCs. The following result is a special case of the aforementioned Hata's criterion.

\begin{lemma}[\cite{Hat85}]\label{thm:hata}
    A GSC $F=F(N,\D)$ is connected if and only if for every pair of $i,i'\in\D$, there exists $\{i_k\}_{k=1}^{n}\subset\D$ such that $i_1=i$, $i_n=i'$ and $\vp_{i_k}(F)\cap \vp_{i_{k+1}}(F)\neq\varnothing$ for all $1\leq k\leq n-1$.
\end{lemma}

In the following, we shall take a step further by characterizing when $\vp_i(F)\cap \vp_{i'}(F)\neq\varnothing$ and provide a simpler criterion with the aid of that. Write $i=(a,b)$ and $i'=(a',b')$. Since $F\subset [0,1]^2$, it is easy to see that
\begin{align*}
    \vp_i(F)\cap \vp_{i'}(F)\neq\varnothing &\Longrightarrow \vp_i([0,1]^2)\cap \vp_{i'}([0,1]^2)\neq\varnothing \\
    &\Longleftrightarrow ([a,a+1]\times[b,b+1])\cap ([a',a'+1]\times[b',b'+1])\neq\varnothing \\
    &\Longleftrightarrow |a-a'|\leq 1 \text{ and } |b-b'|\leq 1.
\end{align*}
Hence it suffices to consider the following four cases (we may of course assume that $i\neq i'$). Recall the notation $Q_n$ from \eqref{eq:qn}.

\textbf{Case 1}. $i-i'=\pm(1,1)$. In this case (please see Figure~\ref{fig:case1+2}(A) for an illustration), we claim that the following four statements are mutually equivalent:
\begin{align*}
        &\text{(1) } \vp_i(F)\cap \vp_{i'}(F)\neq\varnothing; &\text{(2) } &\vp_i(Q_1)\cap \vp_{i'}(Q_1)\neq\varnothing; \\
        &\text{(3) } (0,0),(N-1,N-1)\in\D; &\text{(4) } &(N-1,N-1)\in\D-\D.
\end{align*}
Without loss of generality, assume that $i-i'=(1,1)$. Firstly, it follows immediately from $F\subset Q_1$ that (1) $\Longrightarrow$ (2). For (2) $\Longrightarrow$ (3), note that
\[
    \vp_i([0,1]^2)\cap \vp_{i'}([0,1]^2) = \{\vp_i((0,0))\}=\{\vp_{i'}((1,1))\}.
\]
Since $Q_1\subset [0,1]^2$, (2) implies that $(0,0),(1,1)\in Q_1$ and hence $(0,0),(N-1,N-1)\in\D$. For (3) $\Longrightarrow$ (1), note that $(0,0),(N-1,N-1)\in\D$ implies for all $k\geq 1$,
\[
    \{(0,0)\}=\bigcap_{k=1}^\infty \vp_{(0,0)}^k([0,1]^2) \subset F, \quad \{(1,1)\}=\bigcap_{k=1}^\infty \vp_{(N-1,N-1)}^k([0,1]^2) \subset F.
\]
Letting $x=\vp_i((0,0))=\vp_{i'}((1,1))$, we see that $x\in \vp_i(F)\cap \vp_{i'}(F)$ and hence (1) holds. Finally, it is clear that (3) $\Longleftrightarrow$ (4).

\begin{figure}[htbp]
    \centering
    \subfloat[Case 1]
    {
        \begin{minipage}[t]{120pt}
        \centering
        \begin{tikzpicture}[scale=0.85]
            \draw[thick] (-1.2,-1.2) rectangle (0,0);
            \draw[thick] (0,0) rectangle (1.2,1.2);
            \draw[thick] (-0.3,-0.3) rectangle (0,0);
            \draw[thick] (0,0) rectangle (0.3,0.3);
        \end{tikzpicture}
      \end{minipage}
    }
    \subfloat[Case 2]
    {
        \begin{minipage}[t]{120pt}
        \centering
        \begin{tikzpicture}[scale=0.9]
            \draw[thick] (-1.2,0) rectangle (0,1.2);
            \draw[thick] (0,-1.2) rectangle (1.2,0);
            \draw[thick] (-0.3,0) rectangle (0,0.3);
            \draw[thick] (0,-0.3) rectangle (0.3,0);
        \end{tikzpicture}
        \end{minipage}
    }
    \caption{Case 1 and Case 2}
    \label{fig:case1+2}
\end{figure}
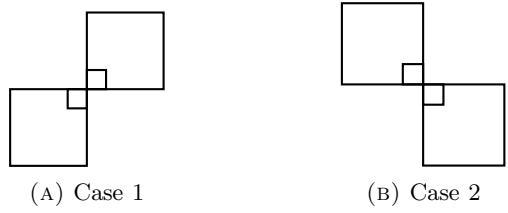

\textbf{Case 2}. $i-i'=\pm(1,-1)$. Similarly as above, we have the following mutually equivalent statements:
\begin{align*}
    &\text{(1) } \vp_i(F)\cap \vp_{i'}(F)\neq\varnothing; &\text{(2) } &\vp_i(Q_1)\cap \vp_{i'}(Q_1)\neq\varnothing; \\
    &\text{(3) } (0,N-1),(N-1,0)\in\D; &\text{(4) } &(N-1,-(N-1))\in\D-\D.
\end{align*}
Please see Figure~\ref{fig:case1+2}(B) for an illustration.

\begin{figure}[htbp]
    \centering
    \subfloat[Subcase 3.1]
    {
        \begin{minipage}[t]{135pt}
        \centering
        \begin{tikzpicture}[scale=0.8]
            \draw[thick] (-2.4,0) rectangle (0,2.4);
            \draw[thick] (0,0) rectangle (2.4,2.4);
            \draw[thick] (-0.6,1.2) rectangle (0,1.8);
            \draw[thick] (0,1.2) rectangle (0.6,1.8);
        \end{tikzpicture}
        \end{minipage}
    }
    \subfloat[Subcase 3.2.1]
    {
        \begin{minipage}[t]{135pt}
        \centering
        \begin{tikzpicture}[scale=0.8]
            \draw[thick] (-2.4,0) rectangle (0,2.4);
            \draw[thick] (0,0) rectangle (2.4,2.4);
            \draw[thick] (-0.6,0.6) rectangle (0,1.2);
            \draw[thick] (0,1.2) rectangle (0.6,1.8);
            \draw[thick] (-0.15,1.05) rectangle (0,1.2);
            \draw[thick] (0,1.2) rectangle (0.15,1.35);
        \end{tikzpicture}
        \end{minipage}
    }
    \subfloat[Subcase 3.2.2]
    {
        \begin{minipage}[t]{135pt}
        \centering
        \begin{tikzpicture}[scale=0.8]
            \draw[thick] (-2.4,0) rectangle (0,2.4);
            \draw[thick] (0,0) rectangle (2.4,2.4);
            \draw[thick] (-0.6,1.2) rectangle (0,1.8);
            \draw[thick] (0,0.6) rectangle (0.6,1.2);
            \draw[thick] (-0.15,1.2) rectangle (0,1.35);
            \draw[thick] (0,1.05) rectangle (0.15,1.2);
        \end{tikzpicture}
        \end{minipage}
    }
    \caption{Case 3}
    \label{fig:case3}
\end{figure}

\textbf{Case 3}. $i-i'=\pm(1,0)$. In this case, $\vp_i(Q_n)\cap \vp_{i'}(Q_n)$ is a scaled copy of some translation of $Q_n\cap (Q_n+(1,0))$ for all $n$. In particular, $\vp_i([0,1]^2)\cap \vp_{i'}([0,1]^2)$ is a line segment (which is the common side of the two squares). Moreover, $\vp_i(Q_1)\cap \vp_{i'}(Q_1)\neq\varnothing$ if and only if at least one of the following happens (Please see Figure~\ref{fig:case3} for  illustration):
\begin{enumerate}
    \item There is some $0\leq k\leq N-1$ such that $(N-1,k),(0,k)\in\D$.
    \item There is some $0\leq k\leq N-1$ such that $(N-1,k),(0,k+1)\in\D$.
    \item There is some $0\leq k\leq N-1$ such that $(N-1,k+1),(0,k)\in\D$.
\end{enumerate}

\textbf{Subcase 3.1}. (1) holds. In this subcase, note that
\begin{align*}
    \bigcap_{n=1}^\infty \vp_{(N-1,k)}^n([0,1]^2) &= \Big\{ \Big( 1,\sum_{n=1}^\infty  \frac{k}{N^n} \Big) \Big\} \\
    &= (1,0)+\Big\{ \Big( 0,\sum_{n=1}^\infty  \frac{k}{N^n} \Big) \Big\} = (1,0)+ \bigcap_{n=1}^\infty \vp_{(0,k)}^n([0,1]^2),
\end{align*}
implying that $F \cap (F+(1,0))$ contains $(1,\sum_{n=1}^\infty  \frac{k}{N^n})$. So $\vp_i(F)\cap \vp_{i'}(F)\neq\varnothing$ since it is a scaled copy of $F \cap (F+(1,0))$. Note that (1)  holds if and only if  $(N-1,0)\in\D-\D$. Thus
\[
    \text{(1)} \Longleftrightarrow (N-1,0)\in\D-\D \Longrightarrow \vp_i(F)\cap\vp_{i'}(F)\neq\varnothing.
\]

\textbf{Subcase 3.2}. (1) fails but one of (2) and (3) holds. In this subcase, if $Q_2\cap (Q_2+(1,0))\neq\varnothing$ (equivalently, $\vp_i(Q_2)\cap \vp_{i'}(Q_2)\neq\varnothing$), then either (2) holds and $\{(0,0),(N-1,N-1)\}\subset\D$  (see Figure~\ref{fig:case3}(B)), or (3) holds and $\{(0,N-1),(N-1,0)\}\subset\D$ (see Figure~\ref{fig:case3}(C)).

\textbf{Subcase 3.2.1} (2) holds and $\{(0,0),(N-1,N-1)\}\subset\D$. Equivalently, $\{(N-1,-1),(N-1,N-1)\}\subset\D-\D$. Then
\begin{align*}
    &\bigcap_{n=1}^\infty \vp_{(N-1,k)}\circ \vp_{(N-1,N-1)}^n([0,1]^2) \\=& \Big\{ \Big( 1,\frac{k+1}{N} \Big) \Big\}
    = (1,0)+\Big\{ \Big( 0,\frac{k+1}{N} \Big) \Big\}
    = (1,0)+ \bigcap_{n=1}^\infty \vp_{(0,k+1)}\circ \vp_{(0,0)}^n([0,1]^2),
\end{align*}
implying that $F\cap (F+(1,0))\neq\varnothing$ (equivalently, $\vp_i(F)\cap \vp_{i'}(F)\neq\varnothing$).


\textbf{Subcase 3.2.2}. (3) holds and $\{(0,N-1),(N-1,0)\}\subset\D$. Equivalently, $\{(N-1,1),(N-1,-(N-1))\}\subset\D-\D$. We can show as in Subcase 3.2.1 that $\vp_i(F)\cap \vp_{i'}(F)\neq\varnothing$.


Combining above discussions, we see that in Case 3, $\vp_i(F)\cap \vp_{i'}(F)\neq\varnothing\Longleftrightarrow \vp_i(Q_2)\cap \vp_{i'}(Q_2)\neq\varnothing\Longleftrightarrow$ at least one of the following sets is contained in $\D-\D$:
\[
    \{(N-1,0)\}, \{(N-1,-1),(N-1,N-1)\}, \{(N-1,1),(N-1,-(N-1))\}.
\]

\textbf{Case 4}. $i-i'=\pm(0,1)$. Similarly as in Case 3, we have in this case that $\vp_i(F)\cap \vp_{i'}(F)\neq\varnothing\Longleftrightarrow \vp_i(Q_2)\cap \vp_{i'}(Q_2)\neq\varnothing\Longleftrightarrow$ at least one of the following sets is contained in $\D-\D$:
\[
    \{(0,N-1)\}, \{(-1,N-1),(N-1,N-1)\}, \{(1,N-1),(N-1,-(N-1))\}.
\]

The above analysis enables us to draw the associated Hata graph (by computer) as follows. Recall that the vertex set consists of elements of $\D$. Moreover, we add an edge joining $i,i'\in\D$  if and only if one of the followings holds:
\begin{enumerate}
    \item $i-i'=\pm(1,1)$ and $(N-1,N-1)\in\D-\D$;
    \item $i-i'=\pm(1,-1)$ and $(N-1,-(N-1))\in\D-\D$;
    \item $i-i'=\pm(1,0)$ and at least one of the following sets is contained in $\D-\D$:
    \[
        \{(N-1,0)\}, \{(N-1,-1),(N-1,N-1)\}, \{(N-1,1),(N-1,-(N-1))\};
    \]
    \item $i-i'=\pm(0,1)$ and at least one of the following sets is contained in $\D-\D$:
    \[
        \{(0,N-1)\}, \{(-1,N-1),(N-1,N-1)\}, \{(1,N-1),(N-1,-(N-1))\}.
    \]
\end{enumerate}

Moreover, it suffices to iterate the initial pattern twice to check the connectedness.

\begin{theorem}\label{Thm:2-2}
    Let $F=F(N,\D)$ be a GSC. Then $F$ is connected $\Longleftrightarrow$ $Q_3$ is connected.
\end{theorem}
\begin{proof}
    The ``$\Longrightarrow$'' part follows from monotonicity. For the ``$\Longleftarrow$'' part, suppose $Q_3$ is connected. Notice that $Q_3=\bigcup_{i\in\D}\vp_i(Q_2)$ is a finite union of compact sets. Thus for each pair of $i,i'\in\D$, we can find a sequence $\{i_k\}_{k=1}^{m}\subset\D$ such that $i_1=i$, $i_m=i'$ and $\vp_{i_k}(Q_2)\cap \vp_{i_{k+1}}(Q_2)\neq\varnothing$ for all $1\leq k\leq m-1$. Recall from the previous discussion that
    \[
        \vp_{i_k}(Q_2)\cap \vp_{i_{k+1}}(Q_2)\neq\varnothing \Longleftrightarrow \vp_{i_k}(F)\cap \vp_{i_{k+1}}(F)\neq\varnothing.
    \]
    It then follows from Lemma~\ref{thm:hata} that $F$ is connected.
\end{proof}

\begin{example}
    In Figure~\ref{fig:exa1}, we present a disconnected GSC $F=F(N,\D)$ where $N=3$ and $\D=\{(0,0),(1,0),(1,1),(1,2),(2,1)\}$. It is easy to see that $Q_2$ is connected. So one cannot hope to determine the connectedness of all GSCs by iterating the initial pattern only once.
    \begin{figure}[htbp]
        \centering
        \includegraphics[width=4cm]{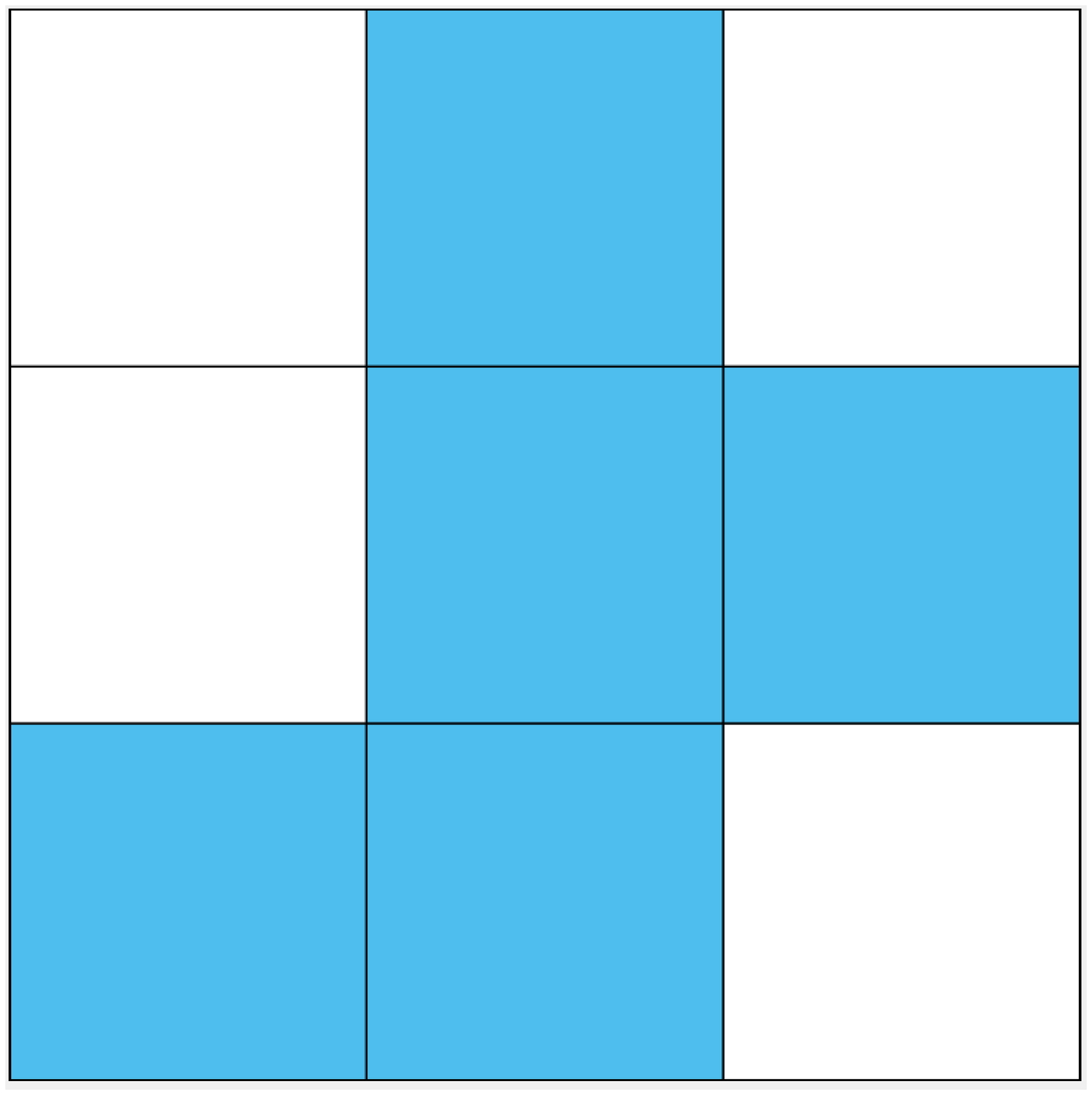} \quad
        \includegraphics[width=4cm]{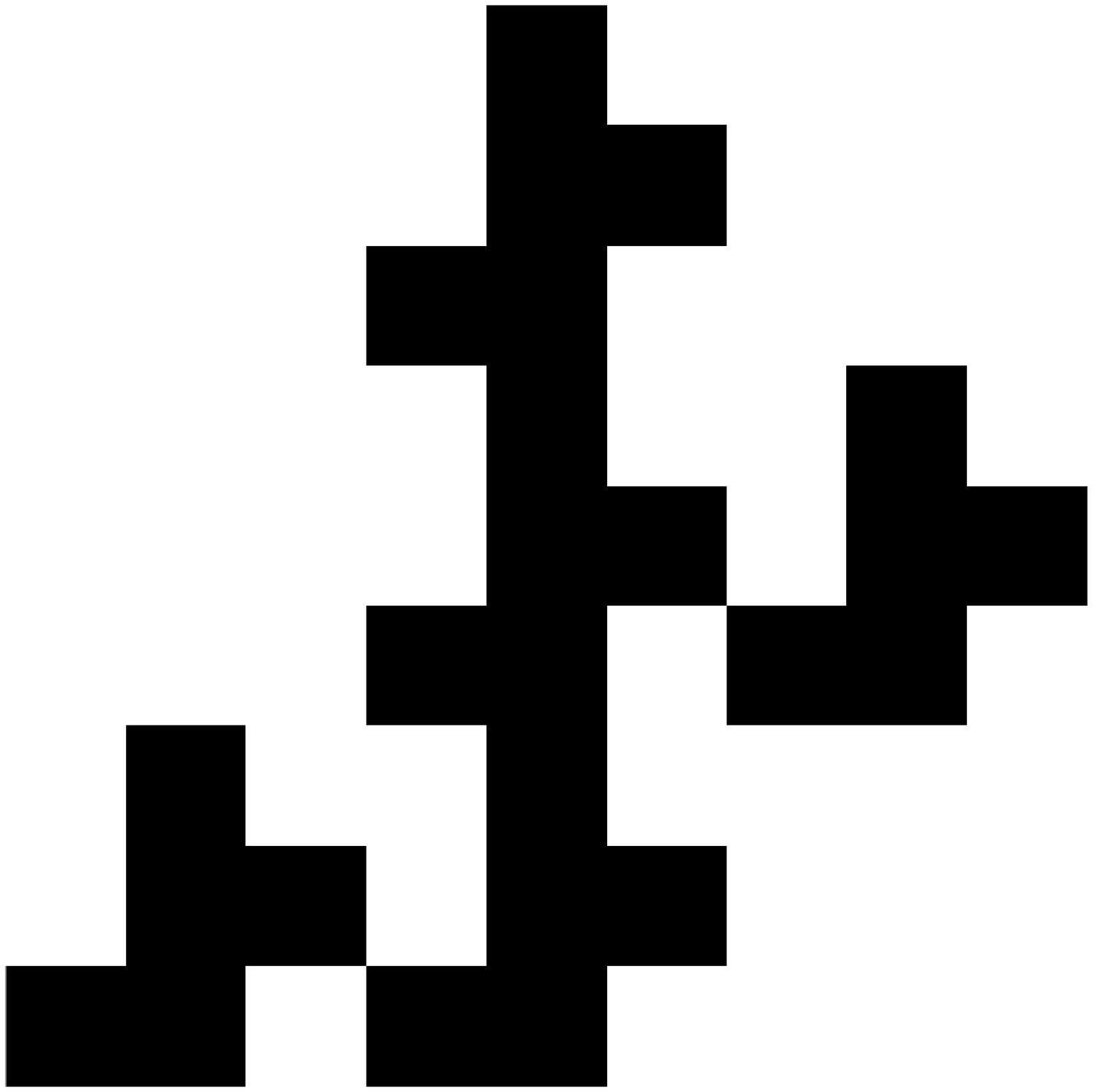} \quad
        \includegraphics[width=4cm]{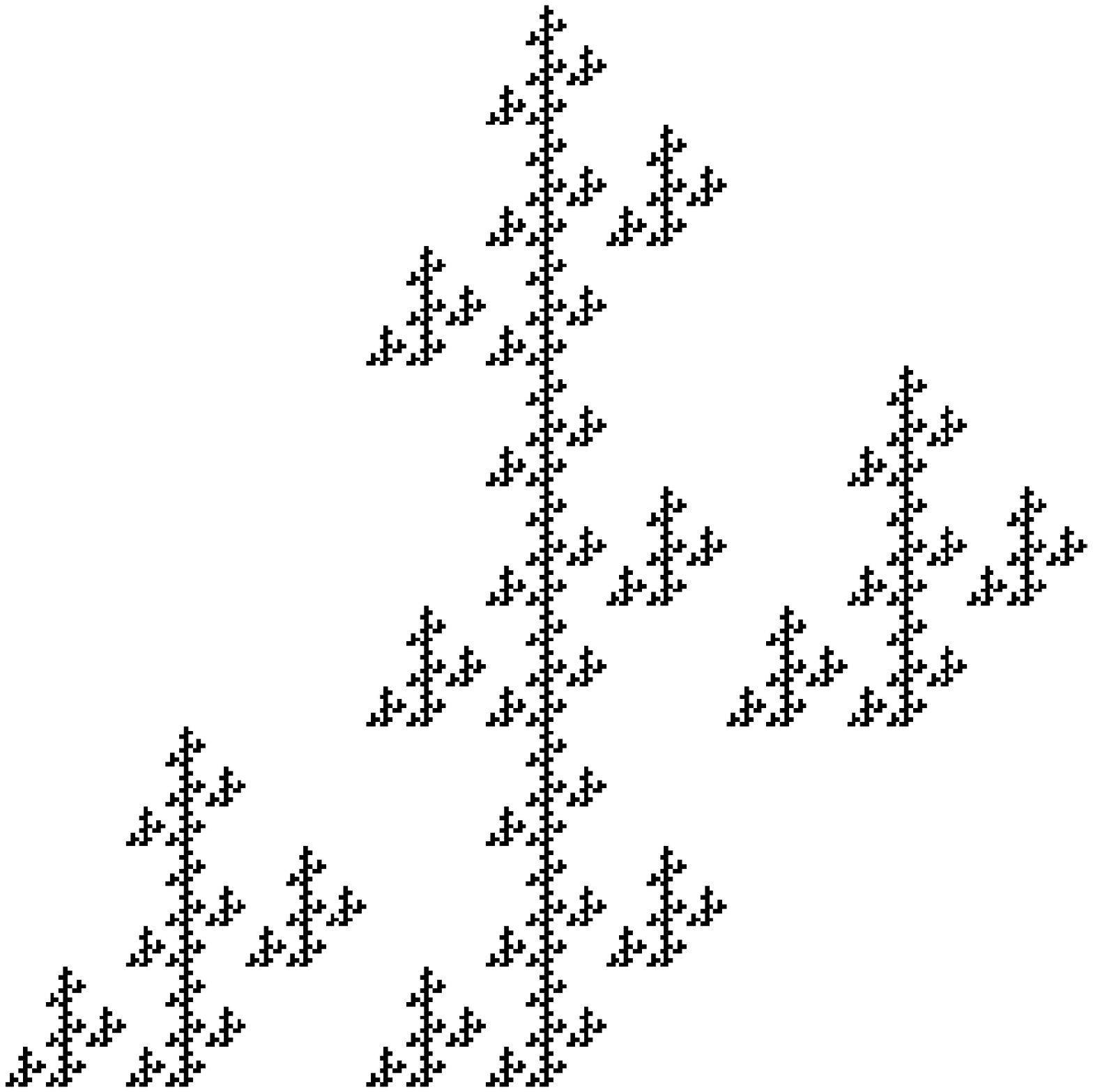}
        \caption{A disconnected GSC. From left to right: the initial pattern, $Q_2$ and the GSC}
        \label{fig:exa1}
    \end{figure}
\end{example}

\section{Cut points of fragile connected GSCs}

Let $F=F(N,\D)$ be a connected GSC and let $\{\Gamma_n\}_{n=1}^\infty$ be the associated Hata graph sequence. First, we show that fragile connected GSCs always have cut points.


\begin{lemma}\label{lem:local1}
    Let $X=A\cup B$ be a connected topological space, where $A,B$ are both closed sets containing at least two points. If $A\cap B=\{x\}$ then $x$ is a cut point of $X$.
\end{lemma}
\begin{proof}
    For simplicity, we write $Y=X\setminus\{x\}$, $U=A\setminus\{x\}$ and $V=B\setminus\{x\}$. Then $U,V$ are disjoint non-empty sets and $Y=U\cup V$. Since $A$ is closed, $U=A\cap Y$ is a closed subset of the topological space $Y$. Similarly, $V$ is also a closed subset of $Y$. Since both of them are non-empty and $U\cap V=\varnothing$, $Y=U\cup V$ is disconnected, i.e., $x$ is a cut point of $X$.
\end{proof}

Note that if $F$ is fragile then we can decompose $\D$ as $\D=\D_1\cup\D_2$ such that
\begin{equation}\label{eq:fragilelower1}
   \Big( \bigcup_{i\in\D_1}\vp_i(F) \Big) \cap \Big( \bigcup_{i\in\D_2}\vp_i(F) \Big) = \{x\}
\end{equation}
for some $x\in F$. It follows immediately from the above lemma that $x$ is a cut point of $F$. So a fragile GSC has at least one cut point.

The following observation is useful to determine the fragility of GSCs.

\begin{proposition}\label{prop:evenfragile}
    Suppose that there is some $m\geq 1$ such that $\D^m$ can be decomposed as $\D^m=I\cup J$ with $I\cap J=\varnothing$ and
    \[
        \Big( \bigcup_{\i\in I}\vp_{\i}(F) \Big) \cap \Big( \bigcup_{\i\in J}\vp_{\i}(F) \Big) = \{x\}
    \]
    for some $x\in F$. Then $F$ is fragile.
\end{proposition}
\begin{proof}
    Let $m_*$ be the smallest positive integer such that $\D^{m_*}$ has a decomposition as above. It suffices to show that $m_*=1$. Suppose $m_*>1$. We first claim that there is some $\omega\in\D^{m_*-1}$ such that $\omega\D\cap I\neq\varnothing$ and $\omega\D\cap J\neq\varnothing$. Otherwise, for every $\omega\in\D^{m_*-1}$, we have either $\omega\D\subset I$ or $\omega\D\subset J$. Then letting $I'=\{\omega\in\D^{m_*-1}:\omega\D\subset I\}$ and $J'=\{\omega\in\D^{m_*-1}:\omega\D\subset J\}$, we see that $I'\cap J'= \varnothing$ and
    \[
        \Big( \bigcup_{\i\in I'}\vp_{\i}(F) \Big) \cap \Big( \bigcup_{\i\in J'}\vp_{\i}(F) \Big) = \Big( \bigcup_{\i\in I}\vp_{\i}(F) \Big) \cap \Big( \bigcup_{\i\in J}\vp_{\i}(F) \Big)=\{x\}.
    \]
    So $\D^{m_*-1}$ has such a decomposition and this contradicts the minimality of $m_*$.

    So there is a word $\omega\in\D^{m_*-1}$ as claimed above. Letting $\D_1=\{i\in\D: \omega i\in I\}$ and $\D_2=\{i\in\D: \omega i\in J\}$, we see that both of $\D_1,\D_2$ are non-empty and $\D_1\cup\D_2=\D$. Since $\vp_\omega(F)$ is connected,
    \[
        \varnothing\neq \Big( \bigcup_{i\in \D_1}\vp_{\omega i}(F) \Big) \cap \Big( \bigcup_{i\in \D_2}\vp_{\omega i}(F) \Big) \subset \Big( \bigcup_{\i\in I}\vp_{\i}(F) \Big) \cap \Big( \bigcup_{\i\in J}\vp_{\i}(F) \Big) = \{x\}.
    \]
    Hence $\big( \bigcup_{i\in\D_1}\vp_{\omega i}(F) \big) \cap \big( \bigcup_{i\in\D_2}\vp_{\omega i}(F) \big) = \{x\}$, which implies that
    \[
        \Big( \bigcup_{i\in\D_1}\vp_{i}(F) \Big) \cap \Big( \bigcup_{i\in\D_2}\vp_{i}(F) \Big) = \{\vp_{\omega}^{-1}(x)\}.
    \]
    So $F$ is fragile.
\end{proof}

Next we prove the lower estimate of $\chi(\Gamma_k)$  for fragile GSCs.
\begin{proposition}\label{prop:fragilelower}
    Let $F=F(N,\D)$ be a fragile connected GSC. Then $\chi(\Gamma_k)\geq |\D|^{k-1}-1$ for all $k\geq 2$.
\end{proposition}
\begin{proof}
    Since $F$ is fragile, we can find some $x\in F$ and a decomposition of $\D$, namely $\D=\D_1\cup\D_2$, such that \eqref{eq:fragilelower1} holds. Fix any $k\geq 2$.
    Letting $\D_\ell\D^{k-1}=\{i\j: i\in\D_\ell,\j\in\D^{k-1}\}$, $\ell=1,2$, we have
    \begin{equation}\label{eq:fragile-levelk}
        \Big( \bigcup_{\j\in\D_1\D^{k-1}}\vp_{\j}(F) \Big) \cap \Big( \bigcup_{\j\in\D_2\D^{k-1}}\vp_{\j}(F) \Big) = \{x\}.
    \end{equation}
    Write $\Omega_k(x)=\{\i\in\D^k: x\in \vp_{\i}(F)\}$.  By the above identity, it is clear that $|\D_\ell\D^{k-1}\cap\Omega_{k}(x)|\geq 1$ for $\ell=1,2$ and for every $k\geq 2$.

    We claim that one of $|\D_\ell\D^{k-1}\cap\Omega_{k}(x)|$, $\ell=1,2$, equals $1$. Otherwise, there are $\i_*,\j_*\in\D_1\D^{k-1}$ and $\i^*,\j^*\in\D_2\D^{k-1}$ with $x\in\bigcap_{t\in\{\i_*,\j_*,\i^*,\j^*\}}\vp_t(F)$. Without loss of generality, assume that these cells locate as in Figure~\ref{fig:4common}, where $x$ is the common vertex of the four corresponding level-$k$ cells. Thus $\{(0,0),(1,0),(0,1),(1,1)\}\subset F$. Then it is easy to see that $\vp_{\i_*}(F)\cap\vp_{\i^*}(F)$ contains more than one point, which contradicts~\eqref{eq:fragile-levelk} and hence proves the claim.


    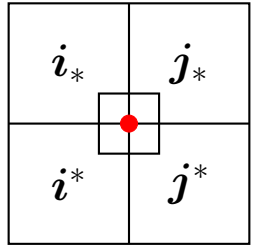
\begin{figure}[htbp]
        \centering
        \begin{tikzpicture}[scale=0.8]
            \draw[thick] (-2,-2) rectangle (2,2);
            \draw[thick] (-2,0) to (2,0);
            \draw[thick] (0,-2) to (0,2);
            \draw[thick] (-0.5,-0.5) rectangle (0.5,0.5);
            \node[font=\fontsize{18}{1}\selectfont] at (-1,1) {$\i_*$};
            \node[font=\fontsize{18}{1}\selectfont] at (1,1) {$\j_*$};
            \node[font=\fontsize{18}{1}\selectfont] at (-1,-1) {$\i^*$};
            \node[font=\fontsize{18}{1}\selectfont] at (1,-1) {$\j^*$};
            \filldraw[draw=red,fill=red] (0,0) circle (4pt);
        \end{tikzpicture}
        \caption{The center point is the common vertex of the four level-$k$ cells.  Here and afterwards, we will typically use words like $i,\i$, etc. to illustrate the corresponding squares or cells.}
        \label{fig:4common}
    \end{figure}

    By our claim, we may assume that $|\D_1\D^{k-1}\cap\Omega_{k}(x)|=1$, say $\D_1\D^{k-1}\cap\Omega_{k}(x)=\{\omega_{k}\}$. Noting that
    \[
        x\notin \bigcup_{\j\in\D_1\D^{k-1}\setminus\Omega_{k}(x)}\vp_{\j}(F) = \bigcup_{\j\in\D_1\D^{k-1}\setminus\{\omega_{k}\}}\vp_{\j}(F),
    \]
    we have by~\eqref{eq:fragilelower1} that
    \[
        \Big( \bigcup_{\j\in\D_1\D^{k-1}\setminus\{\omega_{k}\}}\vp_{\j}(F) \Big)\cap \Big( \bigcup_{\j\in\D_2\D^{k-1}}\vp_{\j}(F) \Big) \subset \Big( \bigcup_{i\in\D_1}\vp_{i}(F)\setminus\{x\} \Big) \cap \Big( \bigcup_{i\in\D_2}\vp_i(F) \Big) = \varnothing.
    \]
    Since $\bigcup_{\j\in\D^{k}\setminus\{\omega_{k}\}}\vp_{\j}(F)$ is the union of the above two sets (which are both non-empty and compact), it is disconnected. So $\omega_{k}$ is a cut vertex of $\Gamma_{k}$ and
    \[
        \chi(\Gamma_k) \geq \min\{|\D_1\D^{k-1}\setminus\{\omega_{k}\}|, |\D_2\D^{k-1}|\} \geq |\D|^{k-1}-1.
    \]
\end{proof}

In order to find an effective way to determine whether a given GSC is fragile, we build a labelling system on the edge set of $\Gamma_1$.


\begin{definition}[Labelled $1$-st Hata graph]\label{def:label}
    Given any edge in $\Gamma_1$, it is labelled ``$x$'' (some point in $F$) if and only if its two endvertices, say $i$ and $j$, satisfy that $\vp_i(F)\cap \vp_j(F)=\{x\}$.
\end{definition}

\begin{example}
    Figure~\ref{fig:fra} depicts a GSC with its initial pattern and the labelled $1$-st Hata graph (in which $a=(\frac{1}{2},\frac{2}{3}), b=(\frac{1}{6},\frac{1}{3}), c=(\frac{1}{2},\frac{1}{3})$ and $d=(\frac{5}{6},\frac{1}{3})$). Note that $F$ is fragile since
    \[
        \vp_{(1,2)}(F) \cap \Big( \bigcup_{i\in\D\setminus\{(1,2)\}}\vp_i(F) \Big) = \{a\}.
    \]

    \begin{figure}[htbp]
        \centering
        \includegraphics[width=4cm]{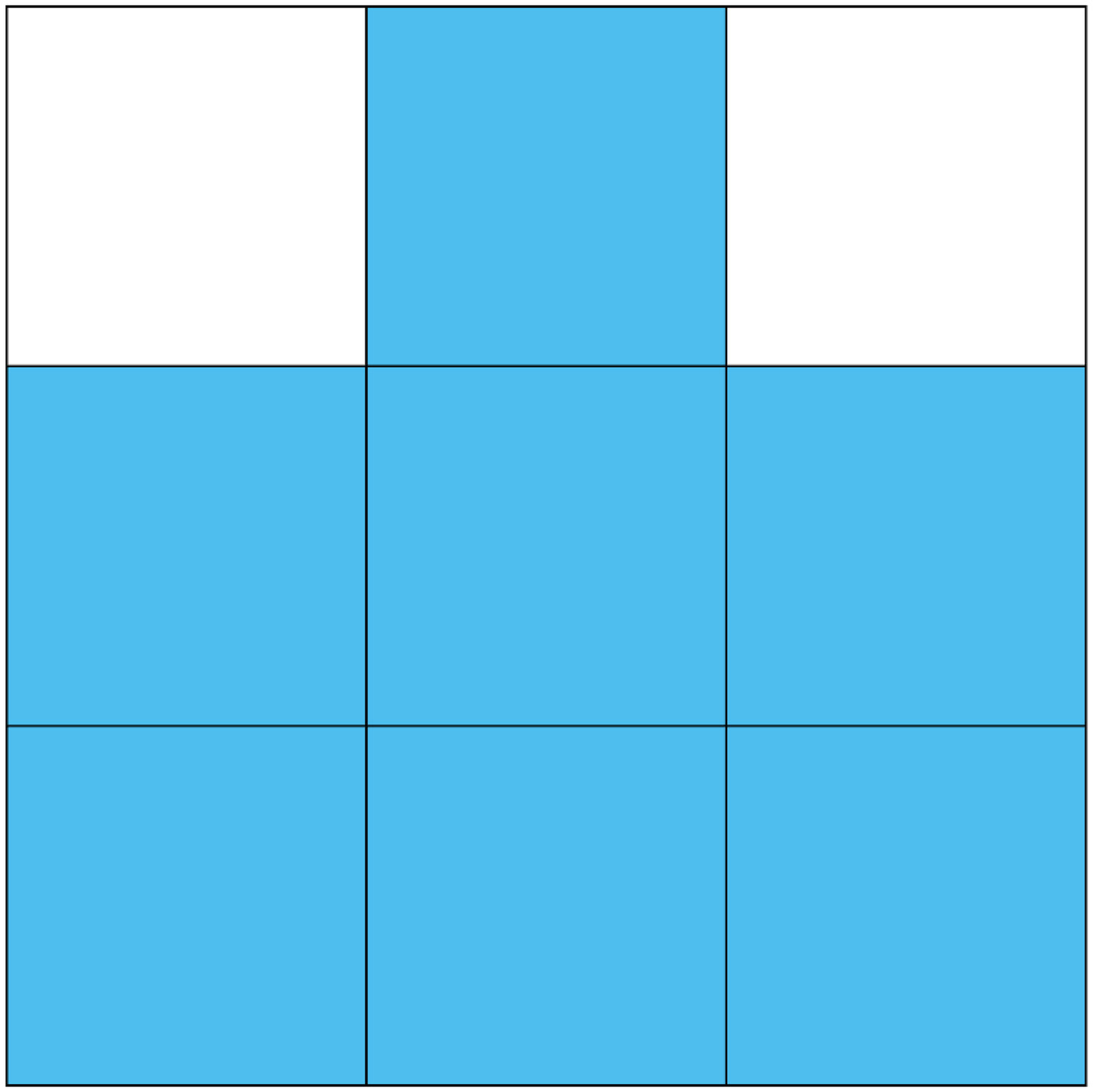} \quad
        \includegraphics[width=4cm]{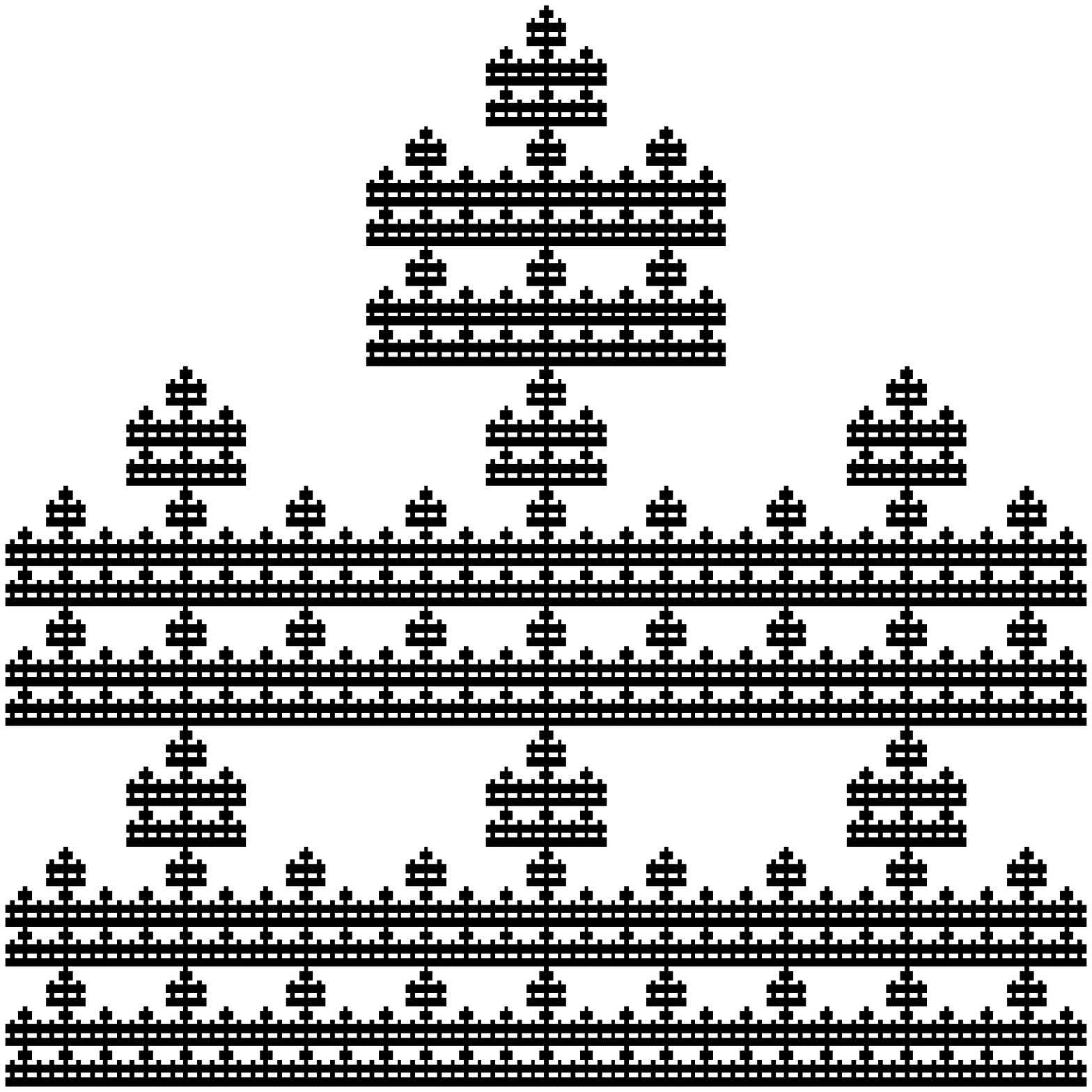} \quad
        \includegraphics[width=4cm]{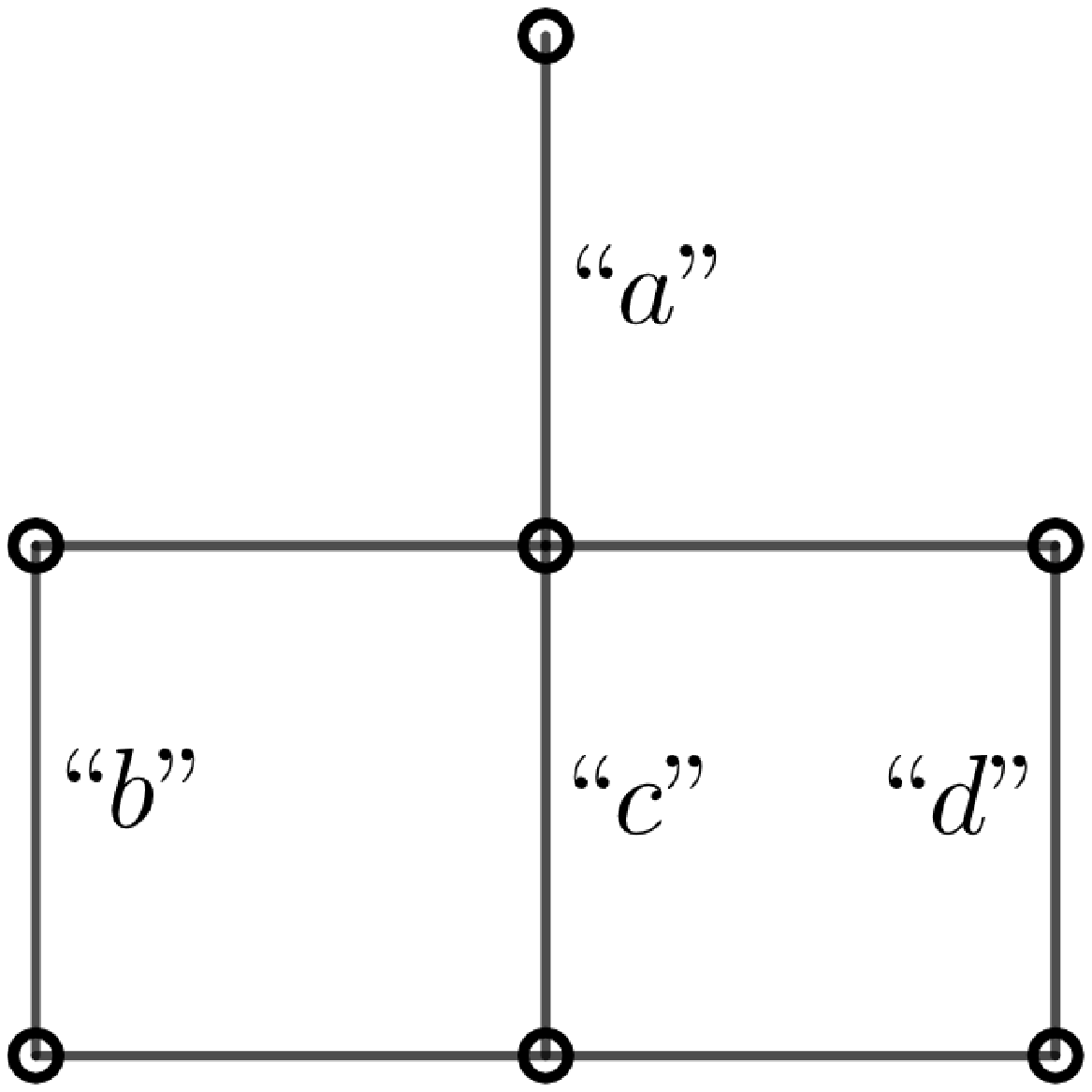}
        \caption{A GSC and the associated labelled $1$-st Hata graph}
        \label{fig:fra}
    \end{figure}
\end{example}

\begin{theorem}\label{thm:fragile}
    A connected GSC $F$ is fragile if and only if there is some $x\in F$ such that if we delete all the edges labelled by ``$x$'' in the associated labelled $1$-st Hata graph then the remaining subgraph is no longer connected.
\end{theorem}
\begin{proof}
    If $F$ is fragile then we have~\eqref{eq:fragilelower1} for some $x\in F$. In particular, the deletion of all edges labelled by ``$x$'' in the associated labelled $1$-st Hata graph destroies its connectedness.

    Conversely, suppose that there is such an $x\in F$. As the assumption suggests, we can decompose $\D$ into two parts, again denoted by $\D_1$ and $\D_2$, such that every pair of vertices $i\in\D_1,j\in\D_2$ belong to different connected components of the labelled $\Gamma_1$ after deleting all edges labelled ``$x$''. Then it is easy to see that \eqref{eq:fragilelower1} holds.    So $F$ is fragile.
\end{proof}

It might be helpful to provide a characterization on when $\vp_i(F)\cap \vp_j(F)$ is a singleton.
Let $A=\{(0,0),(N-1,N-1)\}$ and $B=\{(0,N-1),(N-1,0)\}$. Clearly, if $\vp_i(F)\cap \vp_j(F)\neq\varnothing$ then the two corresponding level-$1$ squares $\vp_i([0,1]^2)$ and $\vp_j([0,1]^2)$ are necessarily adjacent, which leaves us the following four cases to consider.

\textbf{Case 1}: $i-j=(\pm 1,0)$. Without loss of generality, assume that $i-j=(1,0)$. Let
\begin{equation*}
    \begin{gathered}
        I=\{0\leq a\leq N-1: (0,a),(N-1,a)\in\D\}, \\
        J=\{0\leq a\leq N-1: (0,a),(N-1,a-1)\in\D\}, \\
        J'=\{0\leq a\leq N-1: (0,a),(N-1,a+1)\in\D\}.
    \end{gathered}
\end{equation*}
The intersection $\vp_i(F)\cap \vp_j(F)$ is a singleton if and only if one of the following happens:
\begin{enumerate}
    \item We have: (i) $|I|=1$; (ii) if $J\neq\varnothing$ then $A\nsubseteq\D$; (iii) if $J'\neq\varnothing$ then $B\nsubseteq\D$. In this case, the singleton is $\bigcap_{n=1}^\infty\vp_i\circ \vp_{(0,a)}^n([0,1]^2)$ (assuming that $I=\{a\}$).
    \item We have: (i) $|I|=\varnothing$; (ii) $|J|=1$ and $A\subset\D$; (iii) if $J'\neq\varnothing$ then $B\nsubseteq\D$. In this case, the singleton is $\bigcap_{n=1}^\infty\vp_i\circ \vp_{(0,a)}\circ \vp_{(0,0)}^n([0,1]^2)$ (assuming that $J=\{a\}$).
    \item We have: (i) $|I|=\varnothing$; (ii) if $J\neq\varnothing$ then $A\nsubseteq\D$; (iii) $|J'|=1$ and $B\subset\D$. In this case, the singleton is $\bigcap_{n=1}^\infty\vp_i\circ \vp_{(0,a)}\circ \vp_{(0,N-1)}^n([0,1]^2)$ (assuming that $J'=\{a\}$).
\end{enumerate}

\textbf{Case 2}: $i-j=(0,\pm 1)$. The discussion is similar to Case 1 so we omit the details.


\textbf{Case 3}: $i-j=\pm(1,1)$. Without loss of generality, assume that $i-j=(1,1)$. In this case, $\vp_i(F)\cap \vp_j(F)$ is a singleton if and only if $(0,0),(N-1,N-1)\in\D$ and the singleton is $\bigcap_{n=1}^\infty\vp_i\circ \vp_{(0,0)}^n([0,1]^2)$.

\textbf{Case 4}: $i-j=\pm(1,-1)$. The discussion is similar to Case 3 so we omit the details.


\medskip

In the sequel of this paper, $i\in \D$ is called a \emph{corner digit} if $i\in \{(0,0),(0,N-1),(N-1,0),(N-1,N-1)\}$.
From the above arguments, we can obtain the following proposition.

\begin{proposition}\label{prop:4-3}
    Let $i,j\in\D$ be two distinct words. Then $\vp_i(F)\cap\vp_j(F)$ is a singleton if and only if there is exactly one pair of $i',j'\in\D$ such that $\vp_{ii'}(F)\cap\vp_{jj'}(F)\neq\varnothing$. Moreover, if it happens and $i'$ is a corner digit, then so is $j'$.
\end{proposition}
\begin{proof}
    Although the ``only if\,'' part can be directly obtained from the above arguments, we will explain this so that the proof is more readable. Write $\vp_i(F)\cap\vp_j(F)=\{x\}$. Since $\vp_i([0,1]^2)\cap\vp_j([0,1]^2)\neq\varnothing$, it is easy to see that
    \[
        i-j\in\{\pm(1,0),\pm(0,1),\pm(1,1),\pm(1,-1)\}.
    \]

    In the case that $i-j\in\{\pm(1,1),\pm(1,-1)\}$, we may assume  without loss of generality that $i-j=(1,1)$. Since $\vp_i(F)\cap\vp_j(F)\neq\varnothing$, $i'=(0,0)$ and $j'=(N-1,N-1)$ constitute the only one pair such that $\vp_{ii'}(F)\cap\vp_{jj'}(F)\neq\varnothing$.

    In the case that $i-j\in\{\pm(1,0),\pm(0,1)\}$, we may assume without loss of generality that $i-j=(1,0)$. Let $I, J, J'$ be same as defined in the above Case 1. Since $\vp_i(F)\cap \vp_j(F)$ is a singleton and $\vp_{ii'}\cap \vp_{jj'}(F)\not=\varnothing$, from the arguments in the Case 1, if $|I|=1$, then $i'=(0,a)$ and $j'=(N-1,a)$ where we let $I=\{a\}$;  if $|J|=1$ and $\{(0,0),(N-1,N-1)\}\subset \D$, then $i'=(0,0)$ and $j'=(N-1,N-1)$; if $|J'|=1$ and $\{(N-1,0),(0,N-1)\}\subset \D$, then $i'=(0,N-1)$ and $j'=(N-1,0)$.

    The ``Moreover'' part follows directly from the above discussion.

    Now we prove the ``if\,'' part.  Note that if $\vp_{ii'}([0,1]^2)\cap\vp_{jj'}([0,1]^2)$ is a singleton then there is nothing to prove. If not, these two squares must be either left-right or up-down adjacent and without loss of generality, we may assume the former. Then $\vp_{i}([0,1]^2)$ and $\vp_{j}([0,1]^2)$ are also left-right adjacent.
    Thus we may assume without loss of generality that $i-j=(1,0)$. Since $\vp_{ii'}([0,1]^2)$ and $\vp_{jj'}([0,1]^2)$  are left-right adjacent, there exists $a\in \{0,1,\ldots,N-1\}$ such that $i'=(a,0)$ and $j'=(a,N-1)$. Then it is not hard to see by the self-similarity that
    \[
        \vp_{i}(F)\cap\vp_j(F) = \lim_{n\to\infty}\vp_i\circ\vp_{(a,0)}^n([0,1]^2),
    \]
    which is a singleton.
\end{proof}

Using similar arguments in the proof of the above proposition, we can obtain the following two results that will be used later.

\begin{lemma}\label{lem:3-7}
  Let $i,j\in \D$ be two distinct words and $i'$ be a corner digit. If $\vp_{jj'}(F)$ is the only one level-$2$ cell in $\vp_j(F)$ which intersects $\vp_{ii'}(F)$, then $j'$ is a corner digit.
\end{lemma}
\begin{proof}
    Applying similar arguments as in the proof of the ``only if\,'' part of the above proposition, it suffices to consider the case that $i-j\in \{\pm(1,0),\pm(0,1)\}$. Without loss of generality, assume that $i-j=(1,0)$ and $i'=(0,0)$.

    Since $\vp_{ii'}(F)\cap \vp_{jj'}(F)\not=\varnothing$, we have either $j' =(N-1,0)$ or $j'=(N-1,1)$. If $j'=(N-1,1)$ then $(N-1,0)\in \D$. However, this implies that $\vp_j((1,0))=\vp_i((0,0))\in \vp_{j}(\vp_{(N-1,0)}(F))\cap \vp_{ii'}(F)$ so that $\vp_{j}(\vp_{(N-1,0)}(F))$ is another level-$2$ cell in $\vp_j(F)$ meeting $\vp_{ii'}(F)$. This is a contradiction. Thus, we must have $j'=(N-1,0)$ so it is a corner digit. This completes the proof.
\end{proof}


\begin{lemma}\label{lem:fourvertex}
    Let $\alpha$ be a corner digit  and let $\i,\j\in\D^*$ with $|\i|=|\j|$. If $\vp_{\j}(F)\cap\vp_{\i\alpha}(F)\neq\varnothing$, then $\vp_{\i}(\frac{\alpha}{N-1})$, which is a vertex of the square $\vp_{\i}([0,1]^2)$, is an element of $\vp_{\i}(F)\cap\vp_{\j}(F)$.
\end{lemma}
\begin{proof}
    Without loss of generality, assume that $\alpha=(0,0)$. Since $F\subset[0,1]^2$, $\vp_{\j}([0,1]^2)\cap\vp_{\i\alpha}([0,1]^2)\neq\varnothing$. Please see Figure~\ref{fig:relativeposi} for all possible relative locations of these two squares. We will discuss them from left to right as follows.

    For the first case, note that if $(N-1,0)\notin\D$, then  $(N-1,1)\in \D$, and
    \begin{align*}
        \varnothing &\neq \vp_{\j}([0,1]^2)\cap\vp_{\i\alpha}([0,1]^2) \\
        &= (\vp_{\j}\circ\vp_{(N-1,1)}([0,1]^2)) \cap \vp_{\i\alpha}([0,1]^2) = \{\vp_{\j}\circ\vp_{(N-1,1)}((1,0))\},
    \end{align*}
    which implies that $(1,0)\in F$. But this is equivalent to $(N-1,0)\in\D$ and we arrive at a contradiction. So we always have $(N-1,0)\in\D$. Since $\vp_{\i}([0,1]^2)$ and $\vp_{\j}([0,1]^2)$ are left-right adjacent, it is clear that
    \[
        \lim_{n\to\infty} \vp_{\i}\circ\vp_{(0,0)}^n([0,1]^2) = \{\vp_{\i}((0,0))\} = \{\vp_{\j}((1,0))\} = \lim_{n\to\infty} \vp_{\j}\circ\vp_{(N-1,0)}^n([0,1]^2).
    \]
    Thus $\vp_{\i}((0,0))$ is a common element of $\vp_{\i}(F)$ and $\vp_{\j}(F)$.

    The second case can be analogously discussed. For the last one, just note that
    \[
        \vp_{\j}(F) \cap \vp_{\i\alpha}(F) \subset \vp_{\j}([0,1]^2)\cap\vp_{\i}([0,1]^2) = \{\vp_{\i}((0,0))\}.
    \]
    So if the intersection on the left hand side is non-empty, then $\vp_{\i}((0,0))$ must be an element of it. This completes the proof.
\end{proof}

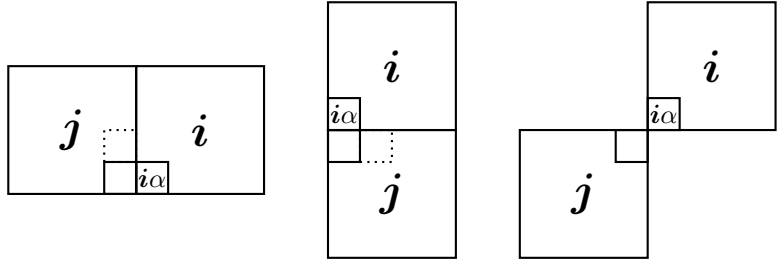
\begin{figure}[htbp]
    \centering
    \begin{tikzpicture}[scale=0.85]
        \draw[thick] (-5,-1) rectangle (-3,1);
        \draw[thick] (-3,-1) rectangle (-1,1);
        \draw[thick] (-3.5,-1) rectangle (-3,-0.5);
        \draw[thick,dotted] (-3.5,-0.5) rectangle (-3,0);
        \draw[thick] (-3,-1) rectangle (-2.5,-0.5);
        \node[font=\fontsize{18}{1}\selectfont] at(-2,0) {$\i$};
        \node[font=\fontsize{18}{1}\selectfont] at(-4,0) {$\j$};
        \node[font=\fontsize{10}{1}\selectfont] at(-2.75,-0.75) {$\i\alpha$};
        \draw[thick] (0,0) rectangle (2,2);
        \draw[thick] (0,-2) rectangle (2,0);
        \draw[thick] (0,0) rectangle (0.5,0.5);
        \draw[thick] (0,-0.5) rectangle (0.5,0);
        \draw[thick,dotted] (0.5,-0.5) rectangle (1,0);
        \node[font=\fontsize{18}{1}\selectfont] at(1,1) {$\i$};
        \node[font=\fontsize{18}{1}\selectfont] at(1,-1) {$\j$};
        \node[font=\fontsize{10}{1}\selectfont] at(0.25,0.25) {$\i\alpha$};
        \draw[thick] (3,-2) rectangle (5,0);
        \draw[thick] (5,0) rectangle (7,2);
        \draw[thick] (4.5,-0.5) rectangle (5,0);
        \draw[thick] (5,0) rectangle (5.5,0.5);
        \node[font=\fontsize{18}{1}\selectfont] at(6,1) {$\i$};
        \node[font=\fontsize{18}{1}\selectfont] at(4,-1) {$\j$};
        \node[font=\fontsize{10}{1}\selectfont] at(5.25,0.25) {$\i\alpha$};
    \end{tikzpicture}
    \caption{Relative positions of $\vp_{\i}([0,1]^2)$ and $\vp_{\j}([0,1]^2)$}
    \label{fig:relativeposi}
\end{figure}

From Proposition~\ref{prop:4-3} and Lemma~\ref{lem:fourvertex}, we can obtain other two sufficient conditions for GSCs to be fragile.

\begin{proposition}\label{prop:1and3}
    Let $i\in\D$. If there is only one level-$3$ cell in $\vp_i(F)$ which intersects other level-$1$ cells, then $F$ is fragile.
\end{proposition}
\begin{proof}
    Let $w=w_1w_2$ be such that $\vp_{iw}(F)$ is the level-$3$ cell as in the statement and let $\mathcal{I}=\{j\in\D\setminus\{i\}:\vp_j(F)\cap\vp_i(F)\neq\varnothing\}$. We first show that $\vp_j(F)\cap\vp_i(F)$ is a singleton for every $j\in\mathcal{I}$ and then show that these singletons are identical.
    As a consequence, $\vp_i(F) \cap \Big( \bigcup_{j\in\D\setminus\{i\}}\vp_j(F) \Big)$ is a singleton,    implying that $F$ is fragile.

    Fix any $j\in\mathcal{I}$. To show that $\vp_i(F)\cap\vp_j(F)$ is a singleton, it suffices to show by Proposition~\ref{prop:4-3} that there is only one level-$2$ cell in $\vp_j(F)$ which intersects $\vp_i(F)$. Suppose on the contrary that one can find distinct $j',j''\in\D$ such that both of $\vp_{jj'}(F)$ and $\vp_{jj''}(F)$ intersect $\vp_i(F)$ (and hence $\vp_{iw}(F)$). Rotating or reflecting if necessary, Figure~\ref{fig:1and3} illustrates the only possibility. In this case, it is easy to see that $w_1\neq w_2$. Note that $\vp_{iw_1}([0,1]^2),\vp_{jj'}([0,1]^2)$ are left-right adjacent as are $\vp_i([0,1]^2),\vp_j([0,1]^2)$. By the self-similarity, $\vp_{iw_1}(\vp_{w_1}(F))$ intersects $\vp_{jj'}(F)$. So there are at least two level-$3$ cells in $\vp_{i}(F)$ meeting $\vp_j(F)$. This is a contradiction and our first goal is achieved.

    If $|\mathcal{I}|=1$ then our second goal is automatically achieved. Suppose $|\mathcal{I}|\geq 2$. Then $\vp_{iw}([0,1]^2)$ locates at one of the corners of $\vp_{i}([0,1]^2)$, i.e., $w_1$ is a corner digit. Fix any $j\in \mathcal{I}$. From Lemma~\ref{lem:fourvertex},
    \[
       \vp_i\big(\frac{w_1}{N-1}\big)\in \vp_{iw_1}(F)\cap \vp_j(F)\subset \vp_{i}(F)\cap \vp_j(F).
    \]
    Since $\vp_i(F)\cap \vp_j(F)$ is a singleton, we have $\vp_i(F)\cap \vp_j(F)=\{\vp_i(\frac{w_1}{N-1})\}$. Since $j$ is arbitrarily chosen in $\mathcal{I}$, this completes the proof.
\end{proof}

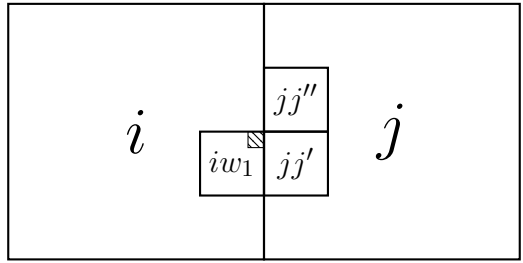
\begin{figure}[htbp]
    \centering
    \begin{tikzpicture}[scale=0.85]
        \draw[thick] (-4,0) rectangle (0,4);
        \draw[thick] (0,0) rectangle (4,4);
        \draw[thick] (-1,1) rectangle (0,2);
        \draw[thick] (0,1) rectangle (1,2);
        \draw[thick] (0,2) rectangle (1,3);
        \draw[pattern=north west lines] (-0.25,1.75) rectangle (0,2);
        \node[font=\fontsize{25}{1}\selectfont] at(-2,2) {$i$};
        \node[font=\fontsize{25}{1}\selectfont] at(2,2) {$j$};
        \node[font=\fontsize{13}{1}\selectfont] at(-0.5,1.5) {$iw_1$};
        \node[font=\fontsize{13}{1}\selectfont] at(0.5,1.5) {$jj'$};
        \node[font=\fontsize{13}{1}\selectfont] at(0.5,2.5) {$jj''$};
    \end{tikzpicture}
    \caption{The shaded square is $\vp_{iw}([0,1]^2)$.}
    \label{fig:1and3}
\end{figure}

\begin{proposition}\label{prop:frag-suff}
    If there exists $k\geq 3$ such that $\chi(\Gamma_k)=|\D|^{k-1}-1$ or $|\D|^{k-1}$, then $F$ is fragile.
\end{proposition}
\begin{proof}
    By the definition of $\chi(\Gamma_k)$, there exists a cut vertex $\i=i_1\cdots i_{k}$ of $\Gamma_{k}$ such that $\Gamma_{k}-\{\i\}$ contains a connected component of size $|\D|^{k-1}-1$ or $|\D|^{k-1}$. From the connectedness of $F$, for every digit $j\in\D\setminus\{i_1\}$ and every connected component $C$ of $\Gamma_{k}-\{\i\}$, either $j\D^{k-1}\subset C$ or none of elements in $j\D^{k-1}$ belongs to $C$.

    In the case that $\chi(\Gamma_k)=|\D|^{k-1}-1$, we know from $|\D^{k-1}|=|\D|^{k-1}$ and the above argument that  $i_1\D^{k-1}\setminus\{\i\}$ is the vertex set of some connected component of  $\Gamma_{k}-\{\i\}$. So
    $\varphi_{\i}(F)$ is the only level-$k$ cell in $\vp_{i_1}(F)$ which intersects  $\bigcup_{j\in\D\setminus\{i_1\}}\varphi_j(F)$. In particular, $\varphi_{\i|_3}(F)$ is the only level-$3$ cell in $\varphi_{i_1}(F)$ which intersects $\bigcup_{j\in\D\setminus\{i_1\}}\varphi_j(F)$. By Proposition~\ref{prop:1and3}, $F$ is fragile.

    In the case that $\chi(\Gamma_k)=|\D|^{k-1}$, there exists some $j_*\in\D\setminus\{i_1\}$ such that $j_*\D^{k}$ is the vertex set of some connected component of $\Gamma_{k}-\{\i\}$. So
    $\varphi_{\i}(F)$ is the only level-$k$ cell in $\bigcup_{j\in\D\setminus\{j_*\}}\varphi_j(F)$ which intersects $\varphi_{j_*}(F)$. In particular, $\varphi_{\i|_3}(F)$ is the only level-$3$ cell in $\varphi_{i_1}(F)$ which intersects $\varphi_{j_*}(F)$. From the first part of the proof of Proposition~\ref{prop:1and3},  $\varphi_{i_1}(F)\cap\varphi_{j_*}(F)$ is a singleton. Then $F$ is fragile since
    \[
        \varphi_{j_*}(F) \cap \Big( \bigcup_{j\in\D\setminus\{j_*\}}\varphi_j(F) \Big) = \varphi_{j_*}(F) \cap \varphi_{i_1}(F).
    \]
\end{proof}

\begin{example}
    The above proposition does not hold for $k=2$. Please see Figure~\ref{fig:equalD} for a GSC with $\chi_{2}(N,\D)=|\D|$. It is easy to see that the GSC is non-fragile. Moreover, we can check that $\chi_3(N,\D)=|\D|<|\D|^2-1$. Thus, by Theorem~\ref{thm:main}, the GSC has no cut point. Similarly, in Figure~\ref{fig:lessD}, we construct a non-fragile GSC with $\chi_{2}(N,\D)=|\D|-1$.

    \begin{figure}[htbp]
    \centering
    \subfloat[Initial pattern]
    {
        \begin{minipage}[t]{120pt}
            \centering
            \includegraphics[width=4cm]{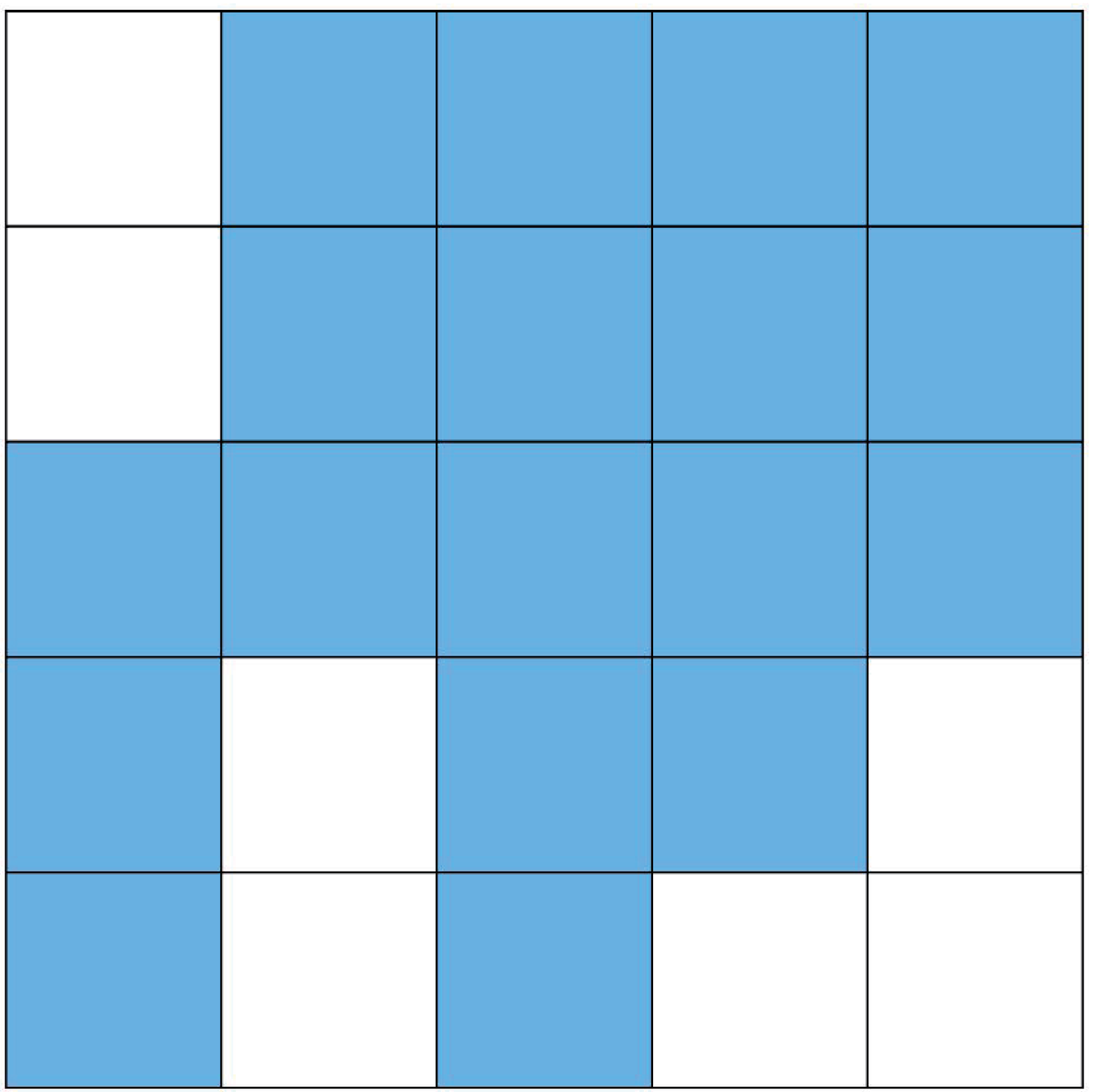}
        \end{minipage}
    }
    \subfloat[The GSC]
    {
        \begin{minipage}[t]{120pt}
            \centering
            \includegraphics[width=4cm]{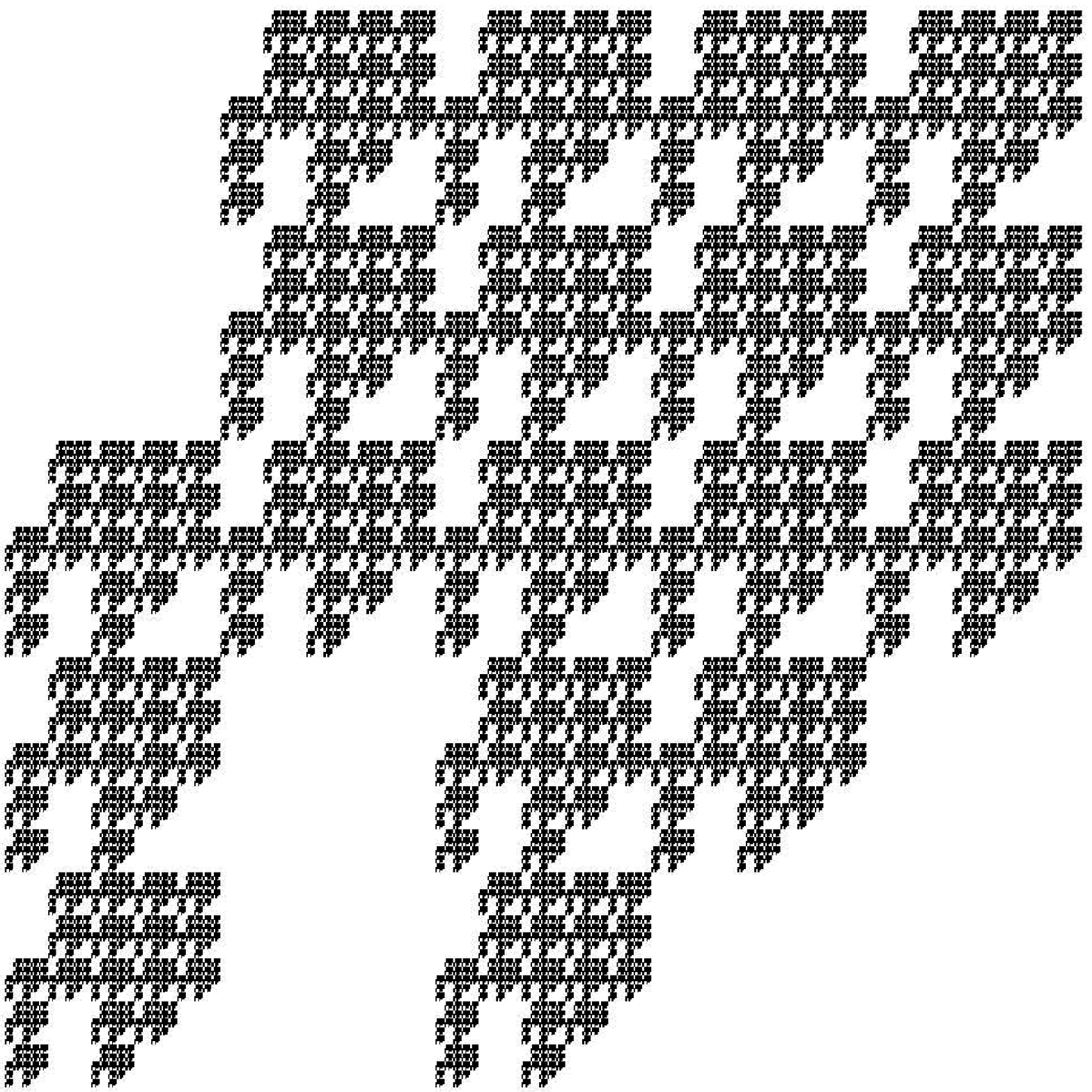}
        \end{minipage}
    }
    \subfloat[$2$-nd Hata graph]
    {
        \begin{minipage}[t]{120pt}
            \centering
            \includegraphics[width=4cm]{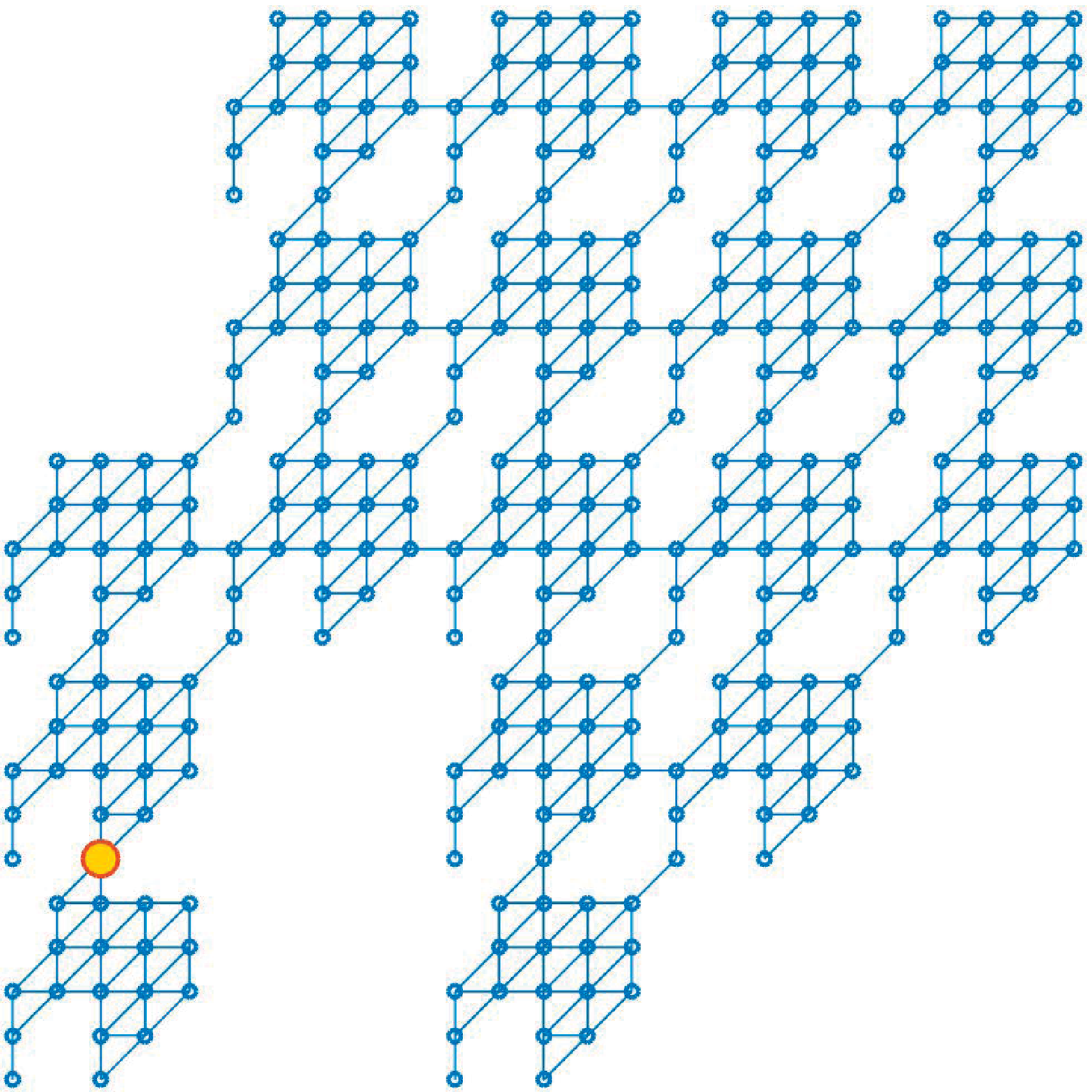}
        \end{minipage}
    }
    \caption{A non-fragile GSC with $\chi_2(N,\D)=|\D|$}
    \label{fig:equalD}
\end{figure}

    \begin{figure}[htbp]
    \centering
    \subfloat[Initial pattern]
    {
        \begin{minipage}[t]{120pt}
            \centering
            \includegraphics[width=4cm]{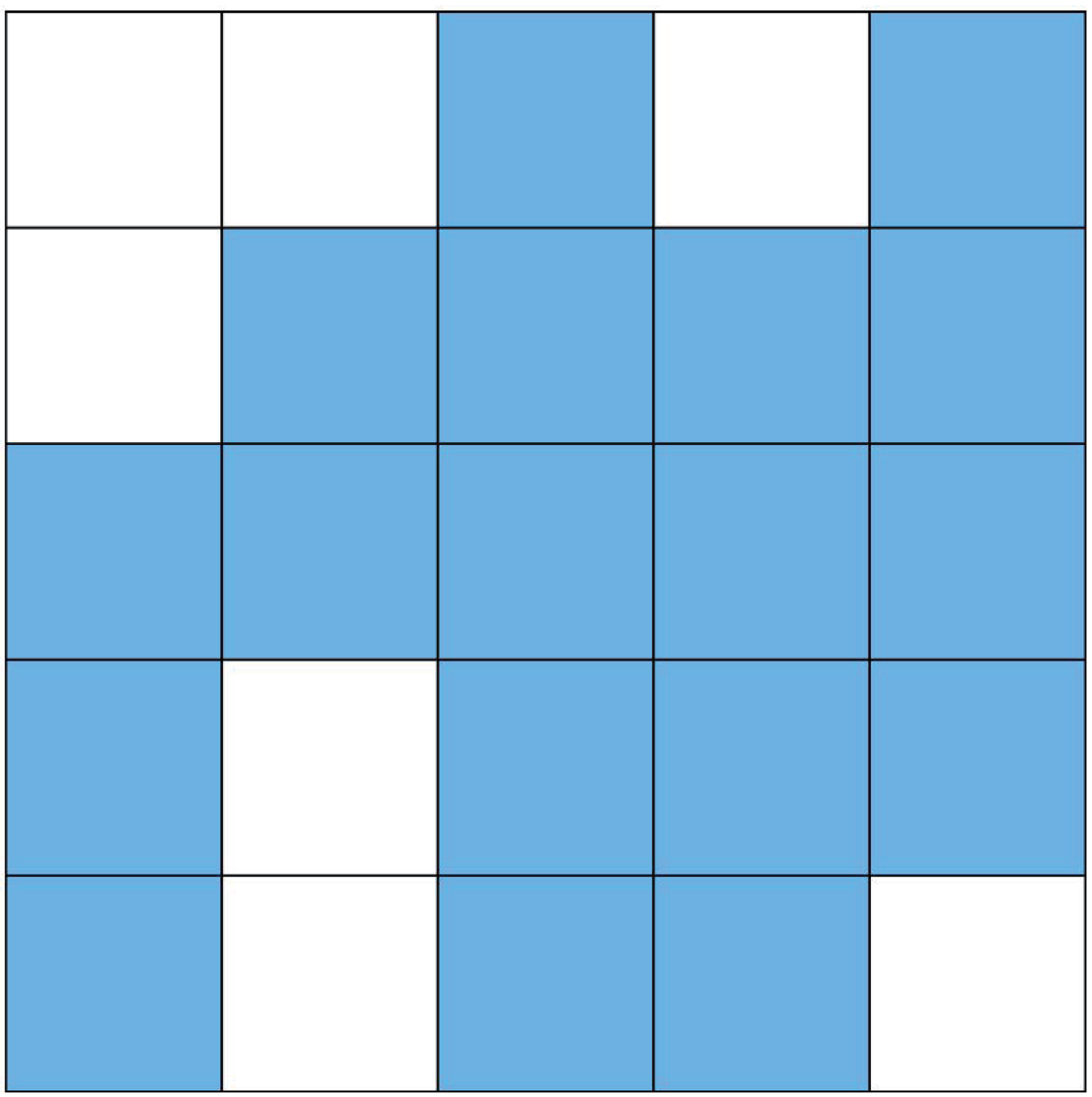}
        \end{minipage}
    }
    \subfloat[The GSC]
    {
        \begin{minipage}[t]{120pt}
            \centering
            \includegraphics[width=4cm]{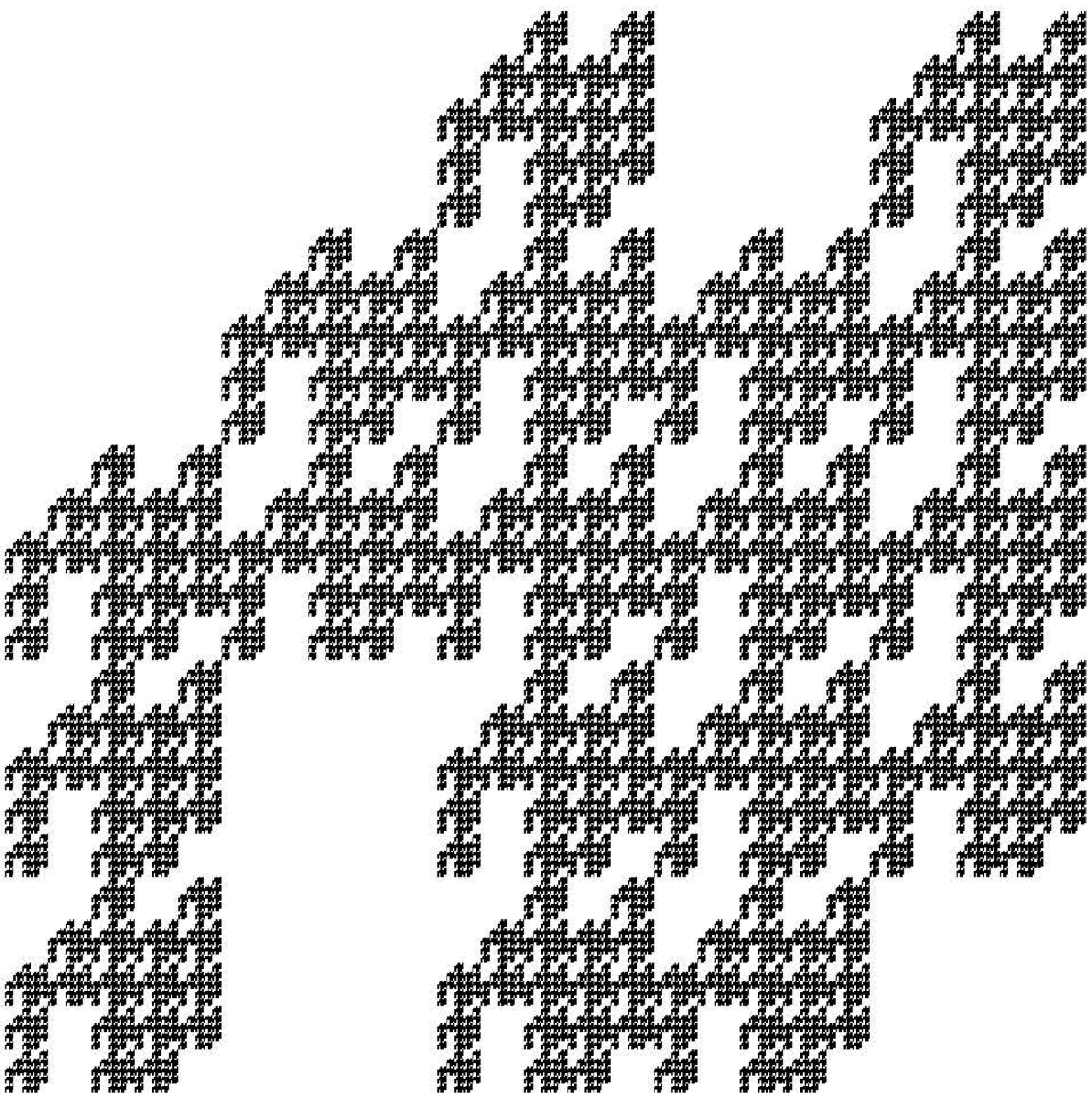}
        \end{minipage}
    }
    \subfloat[$2$-nd Hata graph]
    {
        \begin{minipage}[t]{120pt}
            \centering
            \includegraphics[width=4cm]{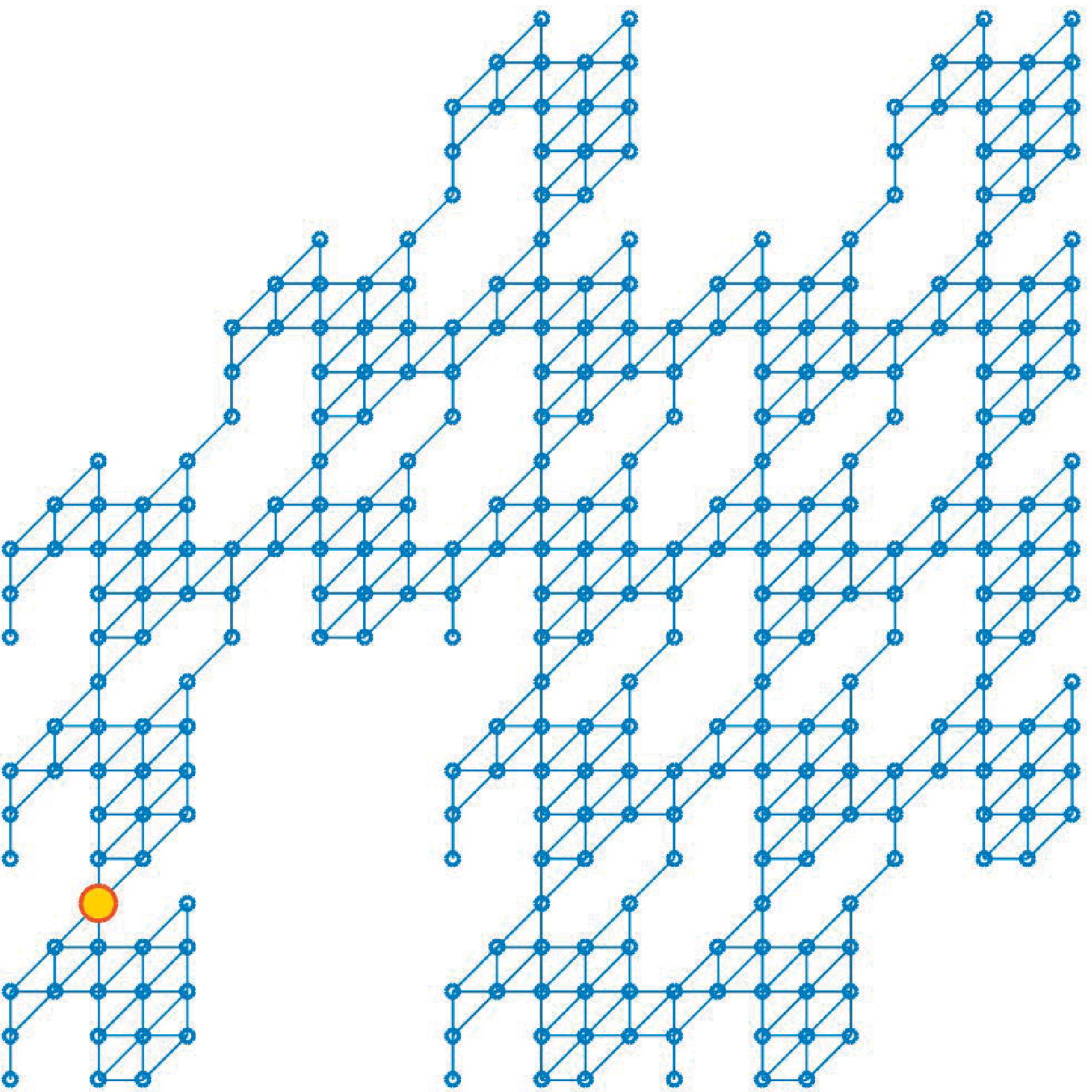}
        \end{minipage}
    }
    \caption{A non-fragile GSC with $\chi_2(N,\D)=|\D|-1$}
    \label{fig:lessD}
\end{figure}

\end{example}

\section{Cut points of non-fragile connected GSCs}

In this section, we will first introduce the notion of well-separated words, which is closely related to the existence of cut points. Then we present a proof of the sufficiency of Theorem~\ref{thm:nonfragile} in Subsection 4.1. In Subsection 4.2, we prove that if a non-fragile connected GSC has cut points, then there exist a cut point $x$ with unique expression and two words which are well separated by $x$. In Subsection 4.3, we prove the necessity of Theorem~\ref{thm:nonfragile}.

\subsection{Well-separated words and proof of the sufficiency of Theorem~\ref{thm:nonfragile}}

From now on, $F=F(N,\D)$ is always presumed to be a connected GSC unless otherwise specified. For every $x\in F$ and $k\geq 1$, write
\begin{equation}\label{eq:omegaande}
    \Omega_k(x) = \{\i\in\D^k: x\in \vp_{\i}(F)\} \quad\text{and}\quad E_k(x) = \bigcup_{\j\in\D^k\setminus\Omega_k(x)} \vp_{\j}(F).
\end{equation}
That is to say, $E_k(x)$ is the union of all level-$k$ cells which do not contain $x$. So $\{E_k(x)\}_{k=1}^\infty$ forms an increasing sequence and
\begin{equation}\label{eq:ekx}
    F\setminus\{x\}=\bigcup_{k=1}^\infty E_k(x).
\end{equation}


\begin{definition}\label{def:5-1}
    Let $x\in F$, $n\in\Z^+$ and let $\omega,\tau$ be two words in $\D^n\setminus\Omega_n(x)$. We say that $\omega$ and $\tau$ are \emph{well separated} by $x$ if $\vp_{\omega}(F)$ and $\vp_{\tau}(F)$ belong to different connected components of $E_{n+p}(x)$ for all $p\geq 1$.
\end{definition}

By definition and noticing that $E_{n+p}(x)$ is increasing with respect to $p$, $\omega$ and $\tau$ are well separated by $x$ if $\vp_{\omega}(F)$ and $\vp_{\tau}(F)$ belong to different connected components of $E_{n+p}(x)$ for all large  $p$.

\begin{theorem}\label{thm:onedire}
    Let $x\in F$, $n\in\Z^+$ and let $\omega,\tau$ be two words in $\D^n\setminus\Omega_n(x)$ which are well separated by $x$. Then $\vp_{\omega}(F)$ and $\vp_{\tau}(F)$ belong to different connected components of $F\setminus\{x\}$. In particular, $x$ is a cut point of $F$.
\end{theorem}
\begin{proof}
    For $p\geq 1$, let
    \begin{align*}
        A_p = \{y\in E_{n+p}(x): y &\text{ and } \vp_{\omega}(F) \text{ belong to } \\ &\text{ the same connected component of $E_k(x)$ for some $k\geq n+p$}\}
    \end{align*}
    and
    \begin{align*}
        B_p = \{y\in E_{n+p}(x): y &\text{ and } \vp_{\omega}(F) \text{ belong to } \\ &\text{ different connected components of $E_k(x)$ for all $k\geq n+p$}\}.
    \end{align*}
    Note that $\vp_\omega(F)\subset A_p$ and $\vp_\tau(F)\subset B_p$ for all $p\geq 1$. It is also clear that $E_{n+p}(x)=A_{p}\cup B_{p}$ is a disjoint union. Since $\{E_k(x)\}_{k=1}^\infty$ is increasing,
    \[
        F\setminus\{x\}=\bigcup_{k=1}^\infty E_k(x)=\bigcup_{k=n+1}^\infty E_k(x)=\bigcup_{p=1}^\infty (A_{p}\cup B_{p})
    \]
    and $\{B_{p}\}_{p=1}^\infty$ is also increasing. Moreover, note that if $y$ and $\vp_\omega(F)$ belong to the same connected component of $E_k(x)$ for some $k\geq n+p$, then they belong to the same connected component of $E_\ell(x)$ for all $\ell\geq k$. Thus $\{A_p\}_{p=1}^\infty$ is also increasing. As a corollary, we have $A_{p}\cap B_q=\varnothing$ for all $p,q$.
    Note that for every $\i\in\D^{n+p}\setminus\Omega_{n+p}(x)$, if $\vp_{\i}(F)\cap A_p\neq\varnothing$ then $\vp_{\i}(F)\subset A_p$. This is because $\vp_{\i}(F)$ is connected. The same statement holds with $A_p$ replaced by $B_p$, so both of $A_{p}$ and $B_{p}$ are finite unions of level-$(n+p)$ cells and hence compact.

    Letting $A=\bigcup_{p=1}^\infty A_{p}$ and $B=\bigcup_{p=1}^\infty B_{p}$, we claim that $\overline{A}=A\cup\{x\}$ and $\overline{B}=B\cup\{x\}$. Since $A\cap B=\varnothing$, $\overline{A}$ and $\overline{B}$ are two closed sets intersecting at exactly one point $x$. By Lemma~\ref{lem:local1}, $F\setminus\{x\}=A\cup B=(\overline{A}\cup\overline{B})\setminus\{x\}$ is disconnected and $\vp_{\omega}(F)\subset A$, $\vp_{\tau}(F)\subset B$ belong to different connected components of $F\setminus\{x\}$.

    It remains to verify the claim. Recall that $\{A_{p}\}_{p=1}^\infty$ is an increasing sequence of compact sets. Fix any $p\geq 1$. Note that for every $q\geq p$,
    \[
        A_q \subset F = \bigcup_{\j\in\D^{n+p}}\vp_{\j}(F) = A_{p} \cup B_{p} \cup \bigcup_{\j\in\Omega_{n+p}(x)} \vp_{\j}(F).
    \]
    Since $A_q\cap B_{p}=\varnothing$, $A_{q}\setminus A_{p}\subset \bigcup_{\j\in \Omega_{n+p}(x)}\vp_{\j}(F).$    Hence
    \begin{equation}\label{eq:aminusap}
        A\setminus A_{p} = \bigcup_{q\geq p} (A_q\setminus A_p) \subset \bigcup_{\j\in \Omega_{n+p}(x)}\vp_{\j}(F).
    \end{equation}
    For every $y\in \overline{A}\setminus A$, there is a sequence of $\{y_m\}_{m=1}^\infty\subset A$ such that $y_m\to y$ as $m\to\infty$. It is easy to see that $y_m\notin A_{p}$ for all large $m$ since otherwise $y\in A_{p}\subset A$, which is a contradiction. Therefore, $y_m\in A\setminus A_{p}$ for all large $m$ and it follows from~\eqref{eq:aminusap} that
    \[
        y=\lim_{m\to\infty} y_m \in \bigcup_{\j\in\Omega_{n+p}(x)}\vp_{\j}(F).
    \]
    Since this holds for every $p\geq 1$, we have
    \[
        y\in \bigcap_{p=1}^\infty \Big( \bigcup_{\j\in\Omega_{n+p}(x)}\vp_{\j}(F) \Big) = \{x\},
    \]
    i.e., $y=x$. This completes the proof that $\overline{A}=A\cup\{x\}$. Similarly, $\overline{B}=B\cup\{x\}$.
\end{proof}

\begin{lemma}\label{lem:wellsep1}
    Let $x\in F$ be a cut point and let $n\in\Z^+$. Then the following statements are equivalent.
    \begin{enumerate}
        \item Every pair of distinct $\omega,\tau\in\D^n\setminus\Omega_n(x)$ is not well separated by $x$;
        \item There is some $p\in\Z^+$ (depending on $n$) such that $E_n(x)$ is contained in exactly one connected component of $E_{n+p}(x)$.
    \end{enumerate}
\end{lemma}
\begin{proof}
    For (1) $\Longrightarrow$ (2), fix any pair of distinct $\omega,\tau\in\D^n\setminus\Omega_n(x)$. Since they are not well separated by $x$, we can find some $p_{\omega,\tau}\in\Z^+$ such that $\vp_\omega(F)$ and $\vp_{\tau}(F)$ belong to the same connected component of $E_{n+p_{\omega,\tau}}(x)$. Letting
    \[
        p=\max\{p_{\omega,\tau}: \omega,\tau\in\D^n\setminus\Omega_n(x), \omega\neq\tau\},
    \]
    we see by the monotonicity of $\{E_{n}(x)\}_{n=1}^\infty$ that $\bigcup_{\omega\in\D^n\setminus\Omega_n(x)}\vp_\omega(F)=E_n(x)$ is contained in exactly one connected component of $E_{n+p}(x)$.

    For (2) $\Longrightarrow$ (1), just note that $\vp_\omega(F)$ and $\vp_\tau(F)$ are both subsets of $E_n(x)$.
\end{proof}

The following result is a converse of Theorem~\ref{thm:onedire}.

\begin{theorem}\label{thm:foreversep}
    Let $x$ be a cut point of $F$. Then there is some large $n_0\in\Z^+$ and two distinct words $\omega,\tau\in\D^{n_0}\setminus\Omega_{n_0}(x)$ which are well separated by $x$.
\end{theorem}
\begin{proof}
    If the statement is not true, then for every fixed $n$, every pair of distinct $\omega,\tau\in\D^n\setminus\Omega_n(x)$ are not well separated by $x$. By Lemma~\ref{lem:wellsep1}, we can find some $p_n\in\Z^+$ such that $E_n(x)$ is contained in exactly one connected component, say $B_n$, of $E_{n+p_n}(x)$. Recall that $\{E_k\}_{k=1}^\infty$ is increasing. Choosing $p_n$ as small as possible, we have $n+p_n\leq (n+1)+p_{n+1}$. So $E_{n+p_n}\subset E_{(n+1)+p_{n+1}}$ and hence $B_n\subset B_{n+1}$ for all $n$. Clearly, $\bigcup_{n=1}^\infty B_n$ is connected. Furthermore,
    \[
        F\setminus\{x\}=\bigcup_{n=1}^\infty E_n(x)\subset\bigcup_{n=1}^\infty B_n\subset \bigcup_{n=1}^\infty E_{n+p_n}(x) = F\setminus\{x\},
    \]
    implying that $F\setminus\{x\}$ is connected. This is a contradiction.
\end{proof}

In the rest of this subsection, we will use Theorem~\ref{thm:onedire} to prove the sufficiency of Theorem~\ref{thm:nonfragile}.

\begin{lemma}\label{lem:suff1}
    Let $k\geq 2$ and let $\j=j_1\cdots j_{k}\in\D^{k}$ be the cut vertex of $\Gamma_{k}$ achieving $\chi(\Gamma_k)$. If $\chi(\Gamma_k)\geq|\D|^{k-1}$, then there are $\omega,\tau\in\D\setminus\{j_1\}$ such that $\omega\D^{k-1}$ and $\tau\D^{k-1}$ belong to different connected components of $\Gamma_k-\{\j\}$.
\end{lemma}
\begin{proof}
    By the connectedness of $\Gamma_{k-1}$, it is clear that for every $i\in\D\setminus\{j_1\}$ and every connected component $C$ of $\Gamma_k-\{\j\}$, either the whole of $i\D^{k-1}$ belongs to $C$ or none of vertices in $i\D^{k-1}$ belongs to $C$.

    If the lemma is false, i.e., there is some connected component $C'$ of $\Gamma_k-\{\j\}$ such that $\bigcup_{i\in\D\setminus\{j_1\}}i\D^{k-1}\subset C'$. Then clearly
    \[
        \chi(\Gamma_k) \leq |\{\eta\in\D^k\setminus\{\j\}: j_1\prec\eta\}| \leq |\D|^{k-1}-1 < |\D|^{k-1},
    \]
    which is a contradiction.
\end{proof}

\begin{lemma}\label{lem:equivalent}
    Let $n\geq 1$ and let $\j=j_1\cdots j_n\in\D^n$. For $\omega,\tau\in\D\setminus\{j_1\}$, the following two statements are equivalent:
    \begin{enumerate}
        \item $\omega\D^{n-1}$ and $\tau\D^{n-1}$ belong to different connected components of $\Gamma_{n}-\{\j\}$;
        \item $\vp_{\omega}(F)$ and $\vp_\tau(F)$ belong to different connected components of $\bigcup_{\eta\in\D^n\setminus\{\j\}} \vp_{\eta}(F)$.
    \end{enumerate}
\end{lemma}
Here we interpret $\omega\D^0=\{\omega\}$ and $\tau\D^0=\{\tau\}$.
\begin{proof}
    Let us write $Y=\bigcup_{\eta\in\D^n\setminus\{\j\}} \vp_{\eta}(F)$ for simplicity. We first prove ``$(1)\Longrightarrow (2)$''. Let $V_1$ be the vertex set of the connected component of $\Gamma_n-\{\j\}$ containing $\omega\D^{n-1}$ and let $V_2=\D^n\setminus(V_1\cup\{\j\})$ (so $\tau\D^{n-1}\subset V_2$). Then there are no edges joinning one vertex in $V_1$ and another vertex in $V_2$, so
    \[
        \Big( \bigcup_{\eta\in V_1}\vp_\eta(F) \Big) \cap \Big( \bigcup_{\eta\in V_2}\vp_\eta(F) \Big) = \varnothing.
    \]
    Note that the above two (non-empty) sets in the brackets are both compact and their union is just $Y$. So they cannot lie in the same connected component of $Y$. In particular, $\vp_{\omega}(F)$ and $\vp_\tau(F)$ belong to different connected components of $Y$.

    For ``$(2)\Longrightarrow (1)$'', suppose on the contrary that $\omega\D^{n-1}$ and $\tau\D^{n-1}$ belong to the same connected component of $\Gamma_{n}-\{\j\}$. So there is a path in $\Gamma_{n}-\{\j\}$ joinning one vertex $\omega\i$ in $\omega\D^{n-1}$ and another vertex $\tau\i'$ in $\tau\D^{n-1}$. By the definition of edges in $\Gamma_{n}-\{\j\}$, we see that $\vp_{\omega\i}(F)$ and $\vp_{\tau\i'}(F)$ lie in one common connected component of $Y$. Since $\vp_\omega(F)$ and $\vp_\tau(F)$ are both connected, they belong to the same connected component of $Y$. This is a contradiction.
\end{proof}

\begin{corollary}\label{cor:nonfragilelast}
    Let $k\geq 1$ and let $\j\in\D^k$ be a cut vertex of $\Gamma_k$. Suppose there are $\omega,\tau\in\D$ such that $\omega\D^{k-1}, \tau\D^{k-1}$ belong to different connected components of $\Gamma_k-\{\j\}$. Then $\omega\D^{q-1}, \tau\D^{q-1}$ belong to different connected components of $\Gamma_{q}-\{\j|_{q}\}$ for all $1\leq q\leq k$.
\end{corollary}
\begin{proof}
    We shall prove this by contradiction. Suppose there is some $1\leq q\leq k$ such that $\omega\D^{q-1}$ and $\tau\D^{q-1}$ belong to one common connected component of $\Gamma_{q}-\{\j|_{q}\}$. Then it follows from the above lemma that $\vp_{\omega}(F)$ and $\vp_\tau(F)$ belong to the same connected component of $\bigcup_{\eta\in\D^{q}\setminus\{\j|_{q}\}}\vp_{\eta}(F)$.
    Since $\bigcup_{\eta\in\D^{q}\setminus\{\j|_{q}\}}\vp_{\eta}(F) \subset \bigcup_{\eta\in\D^{k}\setminus\{\j\}}\vp_{\eta}(F)$, $\vp_{\omega}(F)$ and $\vp_\tau(F)$ must belong to the same connected component of $\bigcup_{\eta\in\D^{k}\setminus\{\j\}}\vp_{\eta}(F)$. Again by the above lemma, we see that $\omega\D^{k-1}$ and $\tau\D^{k-1}$ belong to the same connected component of $\Gamma_k-\{\j\}$, which leads to a contradiction.
\end{proof}

\begin{proof}[Proof of the sufficiency of Theorem~\ref{thm:nonfragile}]
    It suffices to prove the ``moreover" part of the theorem.
    For every $k\geq k_0+1$, let $\i_k$ be a cut vertex of $\Gamma_k$ achieving $\chi(\Gamma_k)$. Since $\chi(\Gamma_k)\geq |\D|^{k-1}$, we can find (by Lemma~\ref{lem:suff1}) $\omega_k,\tau_k\in\D$ such that $\omega_k\D^{k-1}$ and $\tau_k\D^{k-1}$ belong to different connected components of $\Gamma_{k}-\{\i_k\}$. Since there are only a finite number of such pairs $(\omega_k,\tau_k)\in\D\times\D$, we can find $\omega,\tau\in\D$ such that $(\omega,\tau)$ appears infinitely many times in the sequence $\{(\omega_k,\tau_k)\}_{k=k_0+1}^\infty$, say $(\omega_{k_t},\tau_{k_t})\equiv(\omega,\tau)$, $t=1,2,\ldots$. From $\{\i_k\}_{k=k_0+1}^\infty$, we can generate another sequence $\{\j_n\}_{n=k_0+1}^\infty$ by substitution as follows:  Set $\j_n=\i_{k_{t}}|_n$ for $n>k_0$, where $t$ is the unique integer such that $k_{t-1}<n\leq k_t$.
    Since $\omega\D^{k-1},\tau\D^{k-1}$ belong to different connected components of $\Gamma_{k_t}-\{\i_{k_t}\}$ for all $t$, we see by Corollary~\ref{cor:nonfragilelast} that $\omega\D^{n-1},\tau\D^{n-1}$ belong to different components of $\Gamma_n-\{\j_n\}$ for all $n\geq k_0+1$.

    Select a subsequence of $\{\j_n\}$ such that every word in this subsequence shares a common prefix $j_1$. In this subsequence, select another subsequence in which every word shares a common prefix $j_1j_2$.  Continuing in this manner, we can find an infinite word $\bj=j_1j_2j_3\cdots$ such that for each $k\geq 1$, there is a subsequence of $\{\j_n\}$ such that every word in this subsequence shares the common prefix $j_1j_2\cdots j_k$. Using Corollary~\ref{cor:nonfragilelast} again, $\omega\D^{k-1}$ and $\tau\D^{k-1}$ belong to different connected components of $\Gamma_{k}-\{\bj|_{k}\}$ for all $k\geq 1$.  Thus, from Lemma~\ref{lem:equivalent}, $\vp_{\omega}(F)$ and $\vp_{\tau}(F)$ belong to different connected components of $\bigcup_{\eta\in\D^{k}\setminus\{\bj|_{k}\}}\vp_{\eta}(F)$ for all $k\geq 1$.


    Let $x$ be the unique element of the singleton $\vp_{\bj}(F)=\bigcap_{k\geq 1}\vp_{\bj|_k}(F)$. Note that $\bj|_k\in\Omega_k(x)$ and $|\Omega_k(x)|\leq 4$ for all $k$. Fix any large $n\in\Z^+$ such that $|\omega\D^{n-1}|=|\tau\D^{n-1}|=|\D|^{n-1}>4$. So there are $\i_*,\j_*\in\D^{n-1}$ such that $\omega\i_*,\tau\j_*\notin\Omega_n(x)$. Therefore, for all $p\geq 1$ we have
    \begin{equation}\label{eq:5-4}
        \omega\i_*\D^p\cap \Omega_{n+p}(x) = \varnothing, \quad \tau\j_*\D^p\cap \Omega_{n+p}(x) =\varnothing.
    \end{equation}
    Since $\vp_{\omega}(F)$ and $\vp_{\tau}(F)$ belong to different connected components of $\bigcup_{\eta\in\D^{n+p}\setminus\{\bj|_{n+p}\}}\vp_{\eta}(F)$ for all $p$, we know from $\bj|_{n+p}\in \Omega_{n+p}(x)$ and \eqref{eq:5-4} that $\vp_{\omega\i_*}(F)$ and $\vp_{\tau\j_*}(F)$ belong to different connected components of $\bigcup_{\eta\in\D^{n+p}\setminus\Omega_{n+p}(x)}\vp_{\eta}(F)$ for all $p$.  Thus $\omega\i_*$ and $\tau\j_*$ are well separated by $x$. It then follows from Theorem~\ref{thm:onedire} that $x$ is a cut point of $F$.
\end{proof}

\subsection{Cut points with a unique representation}


For $\bi=i_1i_2\cdots\in\D^\infty$ (or a finite word) and $k\in\Z^+$, we simply call $i_k$ the \emph{$k$-th coordinate} of $\bi$. It is easy to see that the limit
\[
    \vp_{\bi}(F) = \lim_{k\to\infty}\vp_{\bi|_k}(F) = \lim_{k\to\infty} \vp_{i_1}\circ\cdots\circ \vp_{i_k}(F)
\]
exists and is a singleton. Denote this singleton by $\{x_{\bi}\}$. This allows us to define a coding map $\pi:\D^\infty\to F$ sending $\bi$ to $x_{\bi}$ and we call $\bi$ a \emph{representation} of $x_{\bi}$. Note that if $\bi$ is a representation of some $x\in F$ then $x\in \vp_{\bi|_n}(F)$ for all $n\geq 1$. Conversely, if $x\in\vp_{\i}(F)$ for some $\i\in\D^*$ then there is a representation of $x$ with $\i$ as its prefix.

\begin{lemma}
    Let $F$ be a GSC. Then every point in $F$ has at most $4$ representations.
\end{lemma}
\begin{proof}
    Note that for every point $x\in F$ and every $n\in\Z^+$, there are at most four level-$n$ squares containing $x$. Suppose there is some $x_0\in F$ with at least five distinct representations, say $\bi,\bj,\omega,\tau,\eta\in\D^\infty$. Let
    \begin{align*}
        m = \max\{n\geq 0: \exists &\text{ a pair of words in } \{\bi,\bj,\omega,\tau,\eta\} \\ &\quad\text{ sharing a common prefix of length $n$}\}.
    \end{align*}
    Then $\bi|_{m+1},\bj|_{m+1},\omega|_{m+1},\tau|_{m+1},\eta|_{m+1}$ are mutually distinct. As a result, there are five level-$(m+1)$ squares containing $x_0$. This is impossible.
\end{proof}

\begin{lemma}\label{lem:connected}
    Let $m\geq 2$ and let $B_1,\ldots,B_m$ be connected compact sets in $\R^2$ such that $\bigcup_{i=1}^m B_i$ is also connected. If $A\subset B_1$ satisfies that $B_1\setminus A$ remains connected and $A\cap B_i=\varnothing$ for all $i\neq 1$, then $X=(B_1\setminus A)\cup B_2\cup\cdots\cup B_m$ is also connected.
\end{lemma}
\begin{proof}
    Write
    \[
       \mathcal{I} = \{2\leq i\leq m: B_i \text{ and } B_1\setminus A \text{ belong to the same connected component of $X$}\}
    \]
    and $\mathcal{I}'=\{2,\ldots,m\}\setminus \mathcal{I}$.  Note that there is some $i\geq 2$ such that $B_i\cap (B_1\setminus A)\neq\varnothing$ (so $\mathcal{I}\neq\varnothing$) since otherwise
    \[
        B_1\cap \Big( \bigcup_{i=2}^m B_i \Big) = \Big( (B_1\setminus A)\cap \Big( \bigcup_{i=2}^m B_i \Big) \Big) \cup \Big( A\cap \Big( \bigcup_{i=2}^m B_i \Big) \Big)=\varnothing,
    \]
    implying that $\bigcup_{i=1}^m B_i$ is disconnected, a contradiction.

    Suppose $X$ is not connected. Then $\mathcal{I}'\neq\varnothing$. So $\mathcal{I},\mathcal{I}'$ are both non-empty. Clearly, $B_i\cap B_j=\varnothing$ for all $i\in \mathcal{I}$ and $j\in \mathcal{I}'$, so
    \[
        \Big( (B_1\setminus A) \cup \bigcup_{i\in \mathcal{I}}B_i \Big) \cap \Big( \bigcup_{i\in \mathcal{I}'}B_i \Big)=\varnothing.
    \]
    It follows that
    \[
        \Big( B_1\cup \bigcup_{i\in \mathcal{I}}B_i \Big) \cap \Big( \bigcup_{i\in \mathcal{I}'}B_i \Big)=A\cap \Big( \bigcup_{i\in \mathcal{I}'}B_i \Big)=\varnothing,
    \]
    which again contradicts the connectedness of $\bigcup_{i=1}^m B_i$. So $X$ is connected.
\end{proof}

\begin{corollary}\label{cor:cp1}
    Let $x$ be a cut point of a connected GSC $F$. If all representations of $x$ share the same first coordinate, say $i_0$, then $x$ is a cut point of $\vp_{i_0}(F)$.
\end{corollary}
\begin{proof}
    It is easy to see that $x\notin \vp_i(F)$ for all $i\in\D\setminus\{i_0\}$. If $\vp_{i_0}(F)\setminus\{x\}$ is connected then it follows from Lemma~\ref{lem:connected} that
    \[
        (\vp_{i_0}(F)\setminus\{x\})\cup\Big( \bigcup_{i\in \D\setminus\{i_0\}}\vp_i(F) \Big) = F\setminus\{x\}
    \]
    is connected, which is a contradiction.
\end{proof}

\begin{corollary}\label{cor:unique}
    Let $x$ be a cut point of a connected GSC $F$ with a unique representation $\bi=i_1i_2\cdots$. Then for every $k\geq 1$, $x_k:=\vp_{\bi|_k}^{-1}(x)$ is a cut point of $F$ with a unique representation $i_{k+1}i_{k+2}\cdots$.
\end{corollary}
\begin{proof}
    Fix any $k\geq 1$. By the above corollary, $x$ is a cut point of $\vp_{i_1}(F)$. So $x_1=\vp_{i_1}^{-1}(x)$ is a cut point of $F$ with $i_2i_3\cdots$ as one of its representations. If $x_1$ has another representation $\bj$, then $i_1\bj$ is clearly a representation of $x$ distinct from $\bi$, which leads to a contradiction. Continuing in this manner, we see that $x_k$ is a cut point of $F$ with a unique representation $i_{k+1}i_{k+2}\cdots$ for all $k\geq 1$.
\end{proof}

The lack of knowledge on the number of representations of cut points often increases the complexity of the problem considerably. Fortunately, the non-fragile requirement will provide us at least one cut point with a unique representation.

\begin{proposition}\label{thm:unirep}
    Let $F$ be a non-fragile connected GSC with cut points. Then there is some cut point of $F$ that has a unique representation.
\end{proposition}
\begin{proof}
    Let $x$ be a cut point of $F$ with $2\leq n\leq 4$ representations $\bi^1,\bi^2,\ldots,\bi^n$, where $\bi^t=i^t_1i^t_2\cdots\in\D^\infty$, $1\leq t\leq n$. Let $\omega=\omega_1\cdots \omega_m$ be the longest common prefix of these infinite words ($\omega$ can be the empty word). In the case that $m\geq 1$, from Corollary~\ref{cor:cp1}, $x$ is a cut point of $\vp_{\omega_1}(F)$ so that $x_1=\vp_{\omega_1}^{-1}(x)$ is a cut point of $F$. Repeating this argument, $x_k:=\vp^{-1}_{\omega_1\cdots \omega_k}(x)$ is a cut point of $F$ for all $1\leq k\leq m$. In particular, $y:=\vp^{-1}_\omega(x)$ is a cut point of $F$  with representations
    \[
        i^1_{|\omega|+1}i^1_{|\omega|+2}\cdots,\, \ldots,\, i^n_{|\omega|+1}i^n_{|\omega|+2}\cdots,
    \]
    and their first coordinates $i^1_{|\omega|+1},\ldots,i^n_{|\omega|+1}$ are not identical. It is clear that this statement also holds when $m=0$. Let $\mathcal{I}=\{i^1_{|\omega|+1},\ldots,i^n_{|\omega|+1}\}$. Then $|\mathcal{I}|>1$. It is noteworthy that $y\notin \vp_i(F)$ for every $i\notin \mathcal{I}$, since otherwise $x\in \vp_\omega\vp_i(F)$ and hence there is a representation of $x$ with $\omega i$ as its prefix (but $x$ has no such representation). Thus $\mathcal{I}=\Omega_1(y)$.

    We claim that there is some $j_0\in \mathcal{I}$ such that $y$ is a cut point of $\vp_{j_0}(F)$. If the claim is not true, then $\vp_j(F)\setminus\{y\}$ is connected for every $j\in \mathcal{I}$. Arbitrarily pick $j_*\in \mathcal{I}$ and write
    \begin{align*}
        \D_1=\{j\in\D: \vp_j(F)\setminus\{y\} & \text{ and } \vp_{j_*}(F)\setminus\{y\} \text{ belong to }\\ &\quad\text{ the same connected component of $F\setminus\{y\}$}\}
    \end{align*}
    and $\D_2=\D\setminus \D_1$. It is easy to see that $\mathcal{I}\cap \D_2\not=\varnothing$ since otherwise
    \[
        F\setminus\{y\}= \Big( \Big( \bigcup_{j\in \D_1}\vp_j(F)\Big)\setminus\{y\} \Big) \cup \Big( \bigcup_{j\in\D_2}\vp_j(F) \Big)
    \]
    is connected (by Lemma~\ref{lem:connected}), which is a contradiction. Note that $j_*\in \D_1$. Thus $\mathcal{I}\cap \D_i\not=\varnothing$ for $i=1,2$. Hence, by the definition of $\D_i$,
    \[
        \Big( \bigcup_{i\in\D_1}\vp_i(F) \Big) \cap \Big( \bigcup_{i\in\D_2}\vp_i(F) \Big) = \{y\}.
    \]
    So $F$ is fragile and we arrive at a contradiction. This completes the proof of our claim.

    Now we find that $y$ is a cut point of $\vp_{j_0}(F)$ for some $j_0\in \mathcal{I}$. So $\vp_{j_0}^{-1}(y)$ is a cut point of $F$ with representations
    \[
        \{i^t_{|\omega|+2}i^t_{|\omega|+3}\cdots: 1\leq t\leq n \text{ with } i^t_{|\omega|+1}=j_0\}.
    \]
    Since $\mathcal{I}$ is not a singleton, there is at least one $1\leq t\leq n$ satisfying that $i^t_{|\omega|+1}\neq j_0$. As a result, $\vp_{j_0}^{-1}(y)$ has $\leq n-1$ distinct representations. In conclusion, given a cut point of $F$ with $n\geq 2$ distinct representations, we can find another cut point of $F$ with $\leq n-1$ distinct representations. So after finitely many steps we should obtain a cut point of $F$ with a unique representation.
\end{proof}


\begin{corollary}\label{cor:foreversep}
    Let $F=F(N,\D)$ be a non-fragile connected GSC with cut points. Then we can find a cut point $x$ of $F$ with a unique representation $\bi$, some $n_0$ large and two distinct words $\omega,\tau\in\D^{n_0}\setminus\{\bi|_{n_0}\}$ which are well separated by $x$.
\end{corollary}
\begin{proof}
    By Proposition~\ref{thm:unirep}, there is a cut point $x$ of $F$ with a unique representation $\bi$. Applying Theorem~\ref{thm:foreversep} to this point, we obtain the desired result.
\end{proof}

Let $\sigma$ be the usual left shift map on $\D^\infty$, i.e., $\sigma(i_1i_2i_3\cdots)=i_2i_3i_4\cdots$ for $i_1i_2i_3\cdots\in \mathcal{D}^\infty$.  For later use, we record the following simple fact.

\begin{lemma}\label{lem:minimal}
    Let $F,x,\bi,n_0,\omega,\tau$ be as in Corollary~\ref{cor:foreversep} and let $\eta$ denote the longest common prefix of $\omega,\tau$ and $\bi$.  If $\eta\neq\vartheta$ then $\sigma^{|\eta|}(\omega)$ and $\sigma^{|\eta|}(\tau)$ are well separated by $x_*:=\vp_{\bi|_\eta}^{-1}(x)$ and $x_*$ is a cut point of $F$  with a unique representation $\bi_*:=\sigma^{|\eta|}(\bi)$.
\end{lemma}
\begin{proof}
    It follows directly from Corollary~\ref{cor:unique} that $x_*$ is a cut point of $F$ and $\bi_*$ is its unique representation.  In particular, $\Omega_{k}(x_*)=\{\bi_*|_{k}\}$ for all $k\geq 1$.

    Since $\omega, \tau$ are well separated by $x$, $\vp_\omega(F)$ and $\vp_{\tau}(F)$ belong to different connected components of $E_{n_0+p}(x)=\bigcup_{\j\in\D^{n_0+p}\setminus\{\bi|_{n_0+p}\}}\vp_{\j}(F)$ for all $p\geq 1$. So for all $p$, these two cells must belong to different connected components of the following subset of $E_{n_0+p}(x)$:
    \[
        \bigcup_{\j\in\D^{n_0+p}\setminus\{\bi|_{n_0+p}\}, \eta\prec\j} \vp_{\j}(F).
    \]
    As a result, $\vp_{\eta}^{-1}\vp_{\omega}(F)=\vp_{\sigma^{|\eta|}(\omega)}(F)$ and $\vp_{\eta}^{-1}\vp_{\tau}(F)=\vp_{\sigma^{|\eta|}(\tau)}(F)$ belong to different connected components of
    \begin{align*}
        \vp_{\eta}^{-1}\Big( \bigcup_{\j\in\D^{n_0+p}\setminus\{\bi|_{n_0+p}\}, \eta\prec\j} \vp_{\j}(F) \Big) &= \bigcup_{\j\in\D^{n_0+p-|\eta|}\setminus\{i_{|\eta|+1}\cdots i_{n_0+p}\}}\vp_{\j}(F) \\
        &= \bigcup_{\j\in\D^{n_0-|\eta|+p}\setminus\{\bi_*|_{n_0-|\eta|+p}\}} \vp_{\j}(F) = E_{(n_0-|\eta|)+p}(x_*)
    \end{align*}
    for all $p\geq 1$, i.e., $\sigma^{|\eta|}(\omega)$ and $\sigma^{|\eta|}(\tau)$ are well separated by $x_*$.
\end{proof}

\subsection{Proof of the necessity of Theorem~\ref{thm:nonfragile}}


In this subsection, we always assume the followings:
\begin{itemize}
    \item[(H1)] $F=F(N,\D)$ is a non-fragile connected GSC;
    \item[(H2)] $x\in F$ is a cut point with a unique representation $\bi=i_1i_2\cdots$;
    \item[(H3)] $\omega=\omega_1\cdots\omega_{n_0}$ and $\tau=\tau_1\cdots\tau_{n_0}$ are well separated by $x$ as in Corollary~\ref{cor:foreversep};
    \item[(H4)] (minimality of $n_0$) $n_0$ is the smallest positive integer allowing well-separated words.  More precisely, for every cut point $y\in F$  with a unique representation and $1\leq k<n_0$, every pair of words in $\D^k$ are not well separated by $y$.
\end{itemize}

We remark that by (H2), $\Omega_n(x)=\{\bi|_n\}$ for all $n\geq 1$.

\begin{proposition}\label{prop:nonfragile1}
    If $n_0=1$, then the necessity of Theorem~\ref{thm:nonfragile} holds.
\end{proposition}
\begin{proof}
    If $n_0=1$, then $\omega$ and $\tau$ are words of length one in $\D\setminus\{i_1\}$. Thus, by (H3), $\vp_\omega(F)$ and $\vp_{\tau}(F)$ belong to different connected components of $E_{p+1}(x)$ for all $p\geq 1$. Notice that $E_{p+1}(x)=\bigcup_{\j\in \D^{p+1}\setminus \{\bi|_{p+1}\}}\vp_{\j}(F)$. Hence, from Lemma~\ref{lem:equivalent}, $\omega\D^p$ and $\tau\D^p$ belong to different connected components of $\Gamma_{p+1}-\{\bi|_{p+1}\}$ for all $p\geq 1$. Thus,
    \begin{equation}\label{eq:non1}
        \chi_{p+1}(N,\D)\geq \min\{|\omega\D^p|,|\tau\D^p|\}=|\D|^{p}, \quad p\geq 1.
    \end{equation}
    As a result, the necessity of Theorem~\ref{thm:nonfragile} holds.
\end{proof}

In the rest of this section, we will show by contradiction that $n_0=1$. From now on, we will suppose on the contrary that $n_0\geq 2$.

\begin{lemma}\label{lem:i1-prefix}
    One of $\omega$ and $\tau$ has $i_1$ as its prefix.
\end{lemma}
\begin{proof}
    We will prove this by contradiction. Suppose $\omega_1,\tau_1\in \D\setminus\{i_1\}$ . Then, by the minimality of $n_0\geq 2$, $\omega_1$ and $\tau_1$ are not well separated by $x$. Thus $\vp_{\omega_1}(F)$ and $\vp_{\tau_1}(F)$ are contained in the same connected component of  $E_{p}(x)$ for some $p$ (and hence for all large $p$). This implies that $\vp_{\omega}(F)$ and $\vp_{\tau}(F)$ are contained in the same connected component of $E_p(x)$ for all large $p$. Thus $\omega$ and $\tau$ are not well separated by $x$, which contradicts (H3).
\end{proof}

We remark that the minimality of $n_0$ in (H4) actually implies that $\omega_1,\tau_1$ and $i_1$ do not share a common value, since otherwise the longest common prefix of $\omega,\tau$ and $\bi$ has positive length and we can obtain by Lemma~\ref{lem:minimal} a contradiction with the minimality of $n_0$. Combining this with the above lemma, we can assume without loss of generality that
\begin{itemize}
    \item[(H5)] $\omega_1=i_1, \tau_1\not=i_1.$
\end{itemize}

\begin{lemma}\label{lem:case1new}
    For every $\j\in\D^{n_0}$ with $\j|_1\neq i_1$, $\j$ and $\omega$ are well separated by $x$.
\end{lemma}
\begin{proof}
    Fix any $\j\in\D^{n_0}$ with $\j|_1\neq i_1$. By the minimality of $n_0\geq 2$, $\j|_1$ and $\tau_1$ are not well separated by $x$. Thus, using the similar argument as in the proof of Lemma~\ref{lem:i1-prefix},  $\vp_{\j}(F)$ and $\vp_{\tau}(F)$ is contained in the same connected component of $E_{p}(x)$ for all large $p$. On the other hand, $\vp_{\omega}(F)$ and $\vp_{\tau}(F)$ belong to different connected components of $E_{n_0+p}(x)$ for all $p\geq 1$. Thus $\vp_{\j}(F)$ and $\vp_{\omega}(F)$ belong to different connected components of $E_p(x)$ for all large $p$. It follows that $\j$ and $\omega$ are well  separated by $x$.
\end{proof}




\begin{lemma}\label{lem:case1}
    The set
    \begin{equation}\label{eq:a}
        K:=\bigcup_{\j\in\D^{n_0}\setminus\{\bi|_{n_0}\},i_1\prec\j} \vp_{\j}(F)
    \end{equation}
    lies in one connected component of $E_{n_0+p}(x)$ for some $p$ (and hence for all large $p$).
\end{lemma}
That is to say, all level-$n_0$ cells contained in $\vp_{i_1}(F)$ except $\vp_{\bi|_{n_0}}(F)$ belong to one common connected component of $E_{n_0+p}(x)$ for  large $p$.
\begin{proof}
    By $\omega_1=i_1$, the minimality of $n_0\geq 2$ and Lemma~\ref{lem:minimal}, for every $\j\in\D^{n_0}\setminus\{\bi|_{n_0}\}$ with $i_1\prec\j$, $\j$ and $\omega$ are not well separated by $x$. This means that there is some $p_{\j}\in\Z^+$ such that $\vp_{\j}(F)$ and $\vp_\omega(F)$ belong to the same connected component of $E_{n_0+p}(x)$. Letting $p=\max\{p_{\j}:\j\in\D^{n_0}\setminus\{\bi|_{n_0}\}, i_1\prec\j\}$ gives us the lemma.
\end{proof}

Let $K$ be as in \eqref{eq:a}. Noting that $\vp_\omega(F)\subset K$, we have
\begin{equation}\label{eq:acapvpj}
    K\cap\vp_{j}(F)=\varnothing, \quad \forall j\in\D\setminus\{i_1\}
\end{equation}
since otherwise $\vp_\omega(F)$ and $\vp_{j}(F)$ belong to the same connected component of $E_{n_0+p}(x)$ for all large $p$ and this contradicts Lemma~\ref{lem:case1new}. In other words, there is only one level-$n_0$ cell (i.e., $\vp_{\bi|_{n_0}}(F)$) in $\vp_{i_1}(F)$ which intersects other level-$1$ cells. Thus, from Proposition~\ref{prop:1and3}, $n_0<3$ since $F$ is non-fragile. Combining this with  the previous hypothesis that $n_0\geq 2$, we have $n_0=2$. As a result,
\begin{equation}\label{eq:K-def-precise}
   K=\bigcup_{i\in \D\setminus \{i_2\}}  \vp_{i_1i}(F).
\end{equation}



\begin{lemma}\label{lem:card-J}
    Let $\mathcal{J}=\{\j\in \D^2:\, i_1\not \prec \j \mbox{ and } \vp_{\bi|_2}(F) \cap \vp_{\j}(F)\not=\varnothing \}$.
    Then $|\mathcal{J}|\geq 2.$
\end{lemma}
\begin{proof}
    Since $F$ is connected, $\vp_{i_1}(F) \cap \Big( \bigcup_{j\in \D\setminus \{i_1\}}\vp_{j}(F) \Big)\not=\varnothing$. Thus, from \eqref{eq:acapvpj},
    \begin{equation}\label{eq:4-9}
        \vp_{\bi|_2}(F) \cap \Big( \bigcup_{j\in \D\setminus \{i_1\}}\vp_{i}(F) \Big)=\vp_{i_1}(F) \cap \Big( \bigcup_{j\in \D\setminus \{i_1\}}\vp_{i}(F) \Big) \not=\varnothing.
    \end{equation}
    This implies that there exists $j_1j_2\in \D^2$ with $j_1\not=i_1$ and $\vp_{j_1j_2}(F)\cap \vp_{\bi|_2}(F)\not=\varnothing$. So $|\mathcal{J}|\geq 1$.

    Suppose $|\mathcal{J}|=1$, i.e., $\mathcal{J}=\{j_1j_2\}$. Then from Proposition~\ref{prop:4-3} and \eqref{eq:acapvpj}, $\vp_{i_1}(F)\cap\vp_{j_1}(F)$ is a singleton. Combining this with~\eqref{eq:4-9},
    \[
        \vp_{i_1}(F)\cap\Big(\bigcup_{j\in \D\setminus \{i_1\}} \vp_j(F) \Big) = \vp_{\bi|_2}(F) \cap \vp_{j_1j_2}(F) = \vp_{i_1}(F)\cap\vp_{j_1}(F)
    \]
    is a singleton and hence $F$ is fragile. This is a contradiction. Thus $|\mathcal{J}|\geq 2$.
\end{proof}


Let $x_{2}=\vp_{\bi|_{2}}^{-1}(x)$. By Corollary~\ref{cor:unique}, $x_{2}$ is a cut point of $F$ with a unique representation $i_{3}i_{4}\cdots$.
Write
\begin{equation}\label{eq:def-Kprime}
    K' = \vp_{\bi|_{2}}(E_1(x_{2})) = \vp_{\bi|_2}\Big(\bigcup_{j\in\D\setminus\{i_3\}}\vp_j(F)\Big).
\end{equation}
The following result plays a key role in the proof of the necessity of Theorem~\ref{thm:nonfragile}.

\begin{lemma}\label{lem:interKKprim}
   Let $K$ and $K'$ be as in \eqref{eq:K-def-precise} and  \eqref{eq:def-Kprime}. Then $K\cap K'\not=\varnothing.$
\end{lemma}
\begin{proof}
    We will prove the lemma by contradiction. Suppose $K'\cap K=\varnothing$. Since $\vp_{i_1}(F)=K\cup \vp_{\bi|_2}(F)$ is connected and $\vp_{\bi|_2}(F)=K'\cup\vp_{\bi|_3}(F)$, $\vp_{\bi|_3}(F)$ is the only level-$3$ cell contained in $\vp_{\bi|_2}(F)$ intersecting $K$.

    \textbf{Case I}. There is some $j\in\D\setminus\{i_1\}$ such that we can find two level-$2$ cells in $\vp_{j}(F)$ which intersect $\vp_{\bi|_2}(F)$ (there might be more than two such cells, but it suffices to look at two of them). Rotating or reflecting if necessary, Figure~\ref{fig:case1} illustrates all possibilities. In both cases, we have the following observations: Writing $i_2=(a,0)$,
    \begin{enumerate}
        \item $(a,1)\notin\D$. Otherwise, since $\vp_{i_1}(\vp_{(a,1)}([0,1]^2)), \vp_{\bi|_2}([0,1]^2)$ are up-down adjacent as are $\vp_{i_1}([0,1]^2)$ and $\vp_j([0,1]^2)$, it follows from the self-similarity of $F$ that there are two level-$3$ cells in $\vp_{\bi|_2}(F)$ intersecting $\vp_{i_1}(\vp_{(a,1)}(F))\subset K$, which is a contradiction;
        \item $(a-1,0),(a+1,0)\notin\D$: Note that $\vp_{i_1}(\vp_{(a-1,0)}(F))$ and $\vp_{i_1}(\vp_{(a+1,0)}(F))$ both intersects $\vp_{j}(F)$. So if one of them belongs to $\D$ then $K\cap\vp_j(F)\neq\varnothing$, which contradicts~\eqref{eq:acapvpj};
        \item At least one of $\vp_{i_1}(\vp_{(a-1,1)}(F))$ and $\vp_{i_1}(\vp_{(a+1,1)}(F))$ does not intersect $\vp_{\bi|_2}(F)=\vp_{i_1i_2}(F)$. Otherwise, we have $\{(0,0),(N-1,0),(0,N-1),(N-1,N-1)\}\subset\D$. Then $\vp_{\bi|_2}(\vp_{(0,N-1)}(F))$ and $\vp_{\bi|_2}(\vp_{(N-1,N-1)}(F))$ are two level-$3$ cells in $\vp_{\bi|_2}(F)$ intersecting $K$, which is again a contradiction.
    \end{enumerate}
    Combining the above observations with the connectedness of $F$, either $\vp_{i_1}(\vp_{(a-1,1)}(F))$ or $\vp_{i_1}(\vp_{(a+1,1)}(F))$ intersects $\vp_{\bi|_2}(F)$, say the former one. But this implies that
    \[
        \vp_{\bi|_2}(F) \cap \Big( \bigcup_{i\in\D\setminus\{i_2\}}\vp_{i_1 i}(F) \Big) = \vp_{\bi|_2}(F) \cap \vp_{i_1}(\vp_{(a-1,1)}(F))
    \]
    which is a singleton (by their positions). Applying $\vp_{i_1}^{-1}$ on both sides, we see that
    \[
        \vp_{i_2}(F) \cap \Big( \bigcup_{i\in\D\setminus\{i_2\}}\vp_{i}(F) \Big)
    \]
    is a singleton and hence $F$ is fragile. This is a contradiction.

    \begin{figure}[htbp]
        \centering
        \begin{tikzpicture}
            \draw[thick] (-2.5,-2) rectangle (-0.5,0);
            \draw[thick] (-2.5,0) rectangle (-0.5,2);
            \draw[thick] (-2,-0.5) rectangle (-1.5,0);
            \draw[thick] (-1.5,-0.5) rectangle (-1,0);
            \draw[thick] (-1.5,0) rectangle (-1,0.5);
            \draw[thick] (0.5,-2) rectangle (2.5,0);
            \draw[thick] (0.5,0) rectangle (2.5,2);
            \draw[thick] (1,-0.5) rectangle (1.5,0);
            \draw[thick] (1.5,0) rectangle (2,0.5);
            \draw[thick] (2,-0.5) rectangle (2.5,0);
            \node[font=\fontsize{18}{1}\selectfont] at(-1.5,1) {$i_1$};
            \node[font=\fontsize{18}{1}\selectfont] at(-1.5,-1) {$j$};
            \node[font=\fontsize{18}{1}\selectfont] at(1.5,1) {$i_1$};
            \node[font=\fontsize{18}{1}\selectfont] at(1.5,-1) {$j$};
            \node[font=\fontsize{10}{1}\selectfont] at(-1.25,0.25) {$\bi|_2$};
            \node[font=\fontsize{10}{1}\selectfont] at(1.75,0.25) {$\bi|_2$};
        \end{tikzpicture}
        \caption{Case I}
        \label{fig:case1}
    \end{figure}

    \textbf{Case II}. For each $j\in\D\setminus\{i_1\}$, there is at most one level-$2$ cell in $\vp_j(F)$ which intersects $\vp_{\bi|_2}(F)$. Recall from Lemma~\ref{lem:card-J} that $|\mathcal{J}|\geq 2$. Thus $\vp_{\bi|_2}(F)$ intersects at least two level-$1$ cells other than $\vp_{i_1}(F)$. In particular, $\vp_{\bi|_2}([0,1]^2)$ must locate at one of the corners of $\vp_{i_1}([0,1]^2)$, i.e., $i_2$ is a corner digit.

    By Lemma~\ref{lem:3-7}, for each $\j\in \mathcal{J}$, $\vp_{\j}([0,1]^2)$ lies at one of the corners of $\vp_{\j|_1}([0,1]^2)$. Thus, from Proposition~\ref{prop:4-3}, $\vp_{\j}(F)\cap \vp_{\bi|_2}(F)$ is a singleton. Combining this with Lemma~\ref{lem:fourvertex}, $\vp_{\j}(F)\cap \vp_{\bi|_2}(F)=\vp_{i_1}(\frac{i_2}{N-1})$. This implies that
    \begin{align*}
        \vp_{i_1}(F) \cap \Big( \bigcup_{j\in\D\setminus\{i_1\}}\vp_j(F) \Big) &= \vp_{\bi|_2}(F) \cap \Big( \bigcup_{j\in\D\setminus\{i_1\}}\vp_j(F) \Big) \\
        &=  \vp_{\bi|_2}(F) \cap \Big( \bigcup_{\j\in \mathcal{J}}\vp_{\j}(F) \Big)= \Big\{ \vp_{i_1}\big(\frac{i_2}{N-1}\big) \Big\}
    \end{align*}
    and hence $F$ is fragile. This is a contradiction.
\end{proof}

\begin{proof}[Proof of the necessity of Theorem~\ref{thm:nonfragile}]

    Let $F=F(N,\D)$ be a non-fragile connected GSC with cut points. By Corollary~\ref{cor:foreversep}  and Lemma~\ref{lem:minimal}, we may assume that (H1)-(H4) hold. Then, by Proposition~\ref{prop:nonfragile1}, it suffices to prove that $n_0=1$. Suppose that $n_0\geq 2$. Then we have already shown that $n_0$ must be $2$, while we can assume that (H1)-(H5) hold.

    Recall that $x_{2}=\vp_{\bi|_{2}}^{-1}(x)$ and $K, K'$ are as in \eqref{eq:K-def-precise} and \eqref{eq:def-Kprime}, respectively. Then $\vp_\omega(F)\subset K$.
    It follows from the minimality of $n_0= 2$ that every pair of distinct $i,i'\in\D\setminus\{i_{3}\}$ is not well separated by $x_{2}$. Combining this with Lemma~\ref{lem:wellsep1} (just apply this to $x_{2}$ and $n=1$), we see that there is some large $p_*\in\Z^+$ such that $E_1(x_{2})$ is contained in exactly one connected component of $E_{p_*}(x_{2})$. Also note that
    \begin{align*}
        \vp_{\bi|_{2}}(E_{p_*}(x_{2})) &= \vp_{\bi|_{2}} \Big( \bigcup_{\j\in\D^{p_*}\setminus\{i_{3}i_4\cdots i_{2+p_*}\}} \vp_{\j}(F) \Big) \\
        &\subset \bigcup_{\j\in\D^{2+p_*}\setminus\{\bi|_{2+p_*}\}} \vp_{\j}(F) = E_{2+p_*}(x).
    \end{align*}
    Therefore, $K'=\vp_{\bi|_{2}}(E_1(x_{2}))$ lies in exactly one connected component of $\vp_{\bi|_{2}}(E_{p_*}(x_{2}))$ and hence of $E_{2+p_*}(x)$. For simplicity, we denote this connected component of $E_{2+p_*}(x)$ by $C$. Then $K'\subset C$.

    By Proposition~\ref{prop:1and3}, there are at least two level-$3$ cells in $\vp_{i_1}(F)$ which intersects other level-$1$ cells. Recalling~\eqref{eq:acapvpj} and~\eqref{eq:K-def-precise}, we can find $i'\in \D\setminus\{i_{3}\}$ and $\j_*\in \mathcal{J}$ such that
    \begin{equation*}
        \vp_{\i|_{2}}(\vp_{i'}(F))\cap \vp_{\j_*}(F)\not=\varnothing.
    \end{equation*}
    Since $\vp_{\i|_{2}}(\vp_{i'}(F))\subset K'\subset C$, we have $\vp_{\j_*}(F)\subset C$.

    Notice that from Lemma~\ref{lem:interKKprim}, $K\cap K'\neq\varnothing$. As a result, $K\subset C$ so that $\vp_{\omega}(F)\subset C$, which implies that $\j_*$ and $\omega$ is not well separated by $x$. This contradicts to Lemma~\ref{lem:case1new} since $\j_*|_1\not=i_1$. Thus the proof is completed.
\end{proof}

\section{Local cut points of connected GSCs}

Now we turn to the existence of local cut points. Let $F=F(N,\D)$ be a connected GSC. The following theorem and the consequent analysis are our main motivation for exploring the existence of cut points first.

\begin{theorem}\label{thm:local1}
    Suppose $x$ is a local cut point of $F$. Recall the notation $\Omega_n(x)$ from~\eqref{eq:omegaande} and let $E'_n(x)=\bigcup_{\i\in\Omega_n(x)}\vp_{\i}(F)$. Then there exists some $m\in\Z^+$ such that $x$ is a cut point of $E'_m(x)$.
\end{theorem}
\begin{proof}
    Since $x$ is a local cut point of $F$, there exists $U\subset F$ such that $U$ is a connected neighborhood of $x$ but $U\setminus\{x\}$ is disconnected. Thus there are a pair of disjoint non-empty open sets $V,W$ such that $U\setminus\{x\}=V\cup W$.
    Notice that $x\in \overline{V}\cap \overline{W}$ since $U$ is connected.

    Suppose on the contrary that $E'_n(x)\setminus\{x\}$ is connected for all $n\geq 1$.
    Write $m=\min\{n\geq 1: E'_n(x)\subset U\}$. Such an $m$ exists since $\diam(E'_n(x))\leq 4\sqrt{2}N^{-n} \to 0$ as $n\to\infty$. Since $E'_m(x)\setminus\{x\}$ is connected, either $E'_m(x)\setminus\{x\}\subset V$ or $E'_m(x)\setminus\{x\}\subset W$. We may assume the former.

    Note that $F=E_m(x)\cup E'_m(x)$, where $E_m(x)$ is as in~\eqref{eq:omegaande}. Since $E'_m(x)\setminus\{x\}\subset V$,
    \begin{equation*}
        W\subset (F\setminus\{x\})\setminus V\subset (F\setminus\{x\})\setminus (E'_m(x)\setminus\{x\})\subset F\setminus E'_m(x) \subset E_m(x).
    \end{equation*}
    Combining this with the fact that $E_m(x)$ is closed, we have $\overline{W}\subset E_m(x)$ so that $x\in \overline{W}\subset E_m(x)$, which contradicts the definition of $E_m(x)$.
\end{proof}

Since a cut point is always a local cut point, it suffices to consider when $F$ has no cut points. Let $x$ be a local cut point of $F$  (if there is any). By Theorem~\ref{thm:local1}, there is some $m\in\Z^+$ such that $x$ is a cut point of $E'_m(x)$. Recall that $|\Omega_m(x)|\leq 4$.

\textbf{Case 1}. $|\Omega_m(x)|=1$. That is to say, there is only one $\i\in\D^m$ such that $x\in \vp_{\i}(F)$. Since $x$ is a cut point of $E'_m(x)=\vp_{\i}(F)$, $\vp_{\i}^{-1}(x)$ is a cut point of $F$, which is a contradiction.

\textbf{Case 2}. $|\Omega_m(x)|=2$, namely $\Omega_m(x)=\{\i,\j\}$. Since there are no cut points of $F$, we see by the same reason as in Case 1 that $\vp_{\i}(F)\setminus\{x\}$ and $\vp_{\j}(F)\setminus\{x\}$ are both connected. If there is some $y\neq x$ such that $y\in \vp_{\i}(F)\cap \vp_{\j}(F)$, then
\[
    E'_m(x)\setminus\{x\} = (\vp_{\i}(F)\setminus\{x\})\cup (\vp_{\i}(F)\setminus\{x\})
\]
is the union of two connected sets with non-empty intersection and hence is also connected. This is a contradiction. In conclusion, $\vp_{\i}(F)\cap\vp_{\j}(F)=\{x\}$.

\textbf{Case 3}. $|\Omega_m(x)|=3$, namely $\Omega_m(x)=\{\i,\j,\k\}$. Letting
\[
    A_1=\vp_{\i}(F)\setminus\{x\}, A_2=\vp_{\j}(F)\setminus\{x\}, A_3=\vp_{\k}(F)\setminus\{x\},
\]
we see again that they are all connected. Similarly as in Case 2, it is not hard to see that there is some $1\leq k\leq 3$ such that $A_k\cap \bigcup_{i\neq k}A_i =\varnothing$. In conclusion, there is some $\omega\in\Omega_m(x)$ such that $\vp_{\omega}(F)\cap \bigcup_{\eta\in\Omega_m(x)\setminus\{\omega\}}\vp_{\eta}(F)=\{x\}$.

\textbf{Case 4}. $|\Omega_m(x)|=4$. We may assume that $\Omega_m(x)=\{\i,\j,\k,\bm{h}\}$. In this case, $x$ is the common vertex of four adjacent level-$m$ squares and we must have
\[
    (0,0),(0,N-1),(N-1,0),(N-1,N-1) \in \D.
\]
Without loss of generality, we may also assume that $x$ is the bottom right, bottom left, top left and top right vertex of $\vp_{\i}([0,1]^2)$, $\vp_{\j}([0,1]^2)$, $\vp_{\k}([0,1]^2)$ and $\vp_{\bm{h}}([0,1]^2)$, respectively.
Then all of $\vp_{\i}(F)\cap \vp_{\j}(F)$, $\vp_{\j}(F)\cap \vp_{\k}(F)$ and $\vp_{\k}(F)\cap \vp_{\bm{h}}(F)$ contain infinitely many points. Similarly as in Case 2, we see that $E'_m(x)\setminus\{x\}$ is connected and obtain a contradiction.

We summarize the above discussion in the following result.

\begin{corollary}\label{cor:local11}
    Let $F$ be a connected GSC with no cut points and let $x$ be a local cut point of $F$. Then there is some $m\in\Z^+$ and a decomposition $\Omega_m(x)=\mathcal{I}\cup \mathcal{J}$ such that
    \begin{equation*}\label{eq:local-decom}
        \Big( \bigcup_{\i\in \mathcal{I}}\vp_{\i}(F) \Big)\cap \Big( \bigcup_{\i\in \mathcal{J}}\vp_{\i}(F) \Big)=\{x\}.
    \end{equation*}
\end{corollary}

The following definition serves just for convenience.

\begin{definition}\label{de:locallyfragile}
    Let $F$ be a connected GSC. For $n\in\Z^+$, we call $\D^n$ \emph{locally fragile} if there are disjoint subsets $\mathcal{I},\mathcal{J}$ of $\D^n$ such that:
        \begin{enumerate}
        \item $ \big( \bigcup_{\i\in \mathcal{I}}\vp_{\i}(F) \big)\cap \big( \bigcup_{\i\in \mathcal{J}}\vp_{\i}(F) \big)$ is a singleton, denoted by $\{x\}$;
        \item $\mathcal{I}\cup \mathcal{J}=\Omega_n(x)$.
    \end{enumerate}
\end{definition}

\begin{remark}\label{rem:omeganxis1or2}
    Corollary~\ref{cor:local11} just states that if $F$ has no cut points but some local cut point, then $\D^m$ is locally fragile for some $m\geq 1$. On the other hand, if $\D^n$ is locally fragile, then we can see from the proof of the following proposition that the point $x$ in Definition~\ref{de:locallyfragile} is actually a local cut point. Combining with our previous discussion (cases 1-4),  we have $|\Omega_n(x)|\in\{2,3\}$.
\end{remark}


\begin{proposition}\label{prop:local1}
     Let $F$ be a connected GSC with no cut points. If $\D^n$ is locally fragile for some $n\in\Z^+$, then $F$ has local cut points.
\end{proposition}
\begin{proof}
    Let $\mathcal{I},\mathcal{J},x$ be as in Definition~\ref{de:locallyfragile}. For simplicity, write $F_{\mathcal{I}}=\bigcup_{\i\in \mathcal{I}}\vp_{\i}(F)$ and $F_{\mathcal{J}}=\bigcup_{\i\in J}\vp_{\i}(F)$. Since $\mathcal{I}\cup\mathcal{J}=\Omega_n(x)$, $\dist(x,\vp_{\j}(F))>0$ for every $\j\in\D^n\setminus( \mathcal{I}\cup \mathcal{J})$. So we can find a small $r>0$ such that
    \begin{equation}\label{eq:local2}
        \{y\in F: |y-x|<r\}\cap \vp_{\j}(F)=\varnothing, \quad \forall \j\in\D^n\setminus( \mathcal{I}\cup \mathcal{J}).
    \end{equation}
    Since $F$ is locally connected, there is a connected neighborhood $U$ of $x$ contained in $\{y\in F: |y-x|<r\}$. By~\eqref{eq:local2}, we have $U\subset F_{\mathcal{I}}\cup F_{\mathcal{J}}$. Therefore
    \[
        U=(U\cap F_{\mathcal{I}})\cup (U\cap F_{\mathcal{J}}).
    \]
    Since $F_{\mathcal{I}}\cap F_{\mathcal{J}}=\{x\}$ and both of them are closed subsets of $F$, $U\cap F_{\mathcal{I}}$ and $U\cap F_{\mathcal{J}}$ are both closed subsets of $U$ which intersects at exactly one point $\{x\}$, and it is easy to see that both of them are non-empty. Then by Lemma~\ref{lem:local1}, $U\setminus\{x\}$ is disconnected. Thus $x$ is a local cut point of $F$.
\end{proof}

\begin{proposition}\label{prop:locallast}
    Let $F$ be a connected GSC with no cut points. Let $n=\inf\{m\geq 1: \D^m \text{ is locally fragile}\}$. If $n<\infty$ then $n\leq 2$.
\end{proposition}
The key observation is that if one can find a pair of level-$n$ squares which are up-down (resp. left-right) adjacent but not contained in one common level-$1$ square, then you can find such a pair of level-$1$ squares.
\begin{proof}
    Suppose $n>2$ and let $\mathcal{I},\mathcal{J}\subset\D^n$, $x\in F$ be as in Definition~\ref{de:locallyfragile}. We first claim that there is no $i\in\D$ such that $i\prec\i$ for all $\i\in \mathcal{I}\cup \mathcal{J}$. Otherwise, letting $\mathcal{I}'=\{\sigma(\i):\i\in \mathcal{I}\}$, $\mathcal{J}'=\{\sigma(\i):\i\in \mathcal{J}\}$ and $y=\vp_i^{-1}(x)$, where $\sigma$ is again the left shift map, we see that $\mathcal{I}'\cap \mathcal{J}'=\varnothing$, $\mathcal{I}'\cup \mathcal{J}'=\Omega_{n-1}(y)$ and
    \[
        \Big( \bigcup_{\j\in \mathcal{I}'}\vp_{\j}(F) \Big)\cap \Big( \bigcup_{\j\in \mathcal{J}'}\vp_{\j}(F) \Big) = \vp_i^{-1}\Big( \Big( \bigcup_{\i\in \mathcal{I}}\vp_{\i}(F) \Big)\cap \Big( \bigcup_{\i\in \mathcal{J}}\vp_{\i}(F) \Big) \Big) = \vp_i^{-1}(\{x\})=\{y\}.
    \]
    So $\D^{n-1}$ is locally fragile and this contradicts the minimality of $n$.

    By Remark~\ref{rem:omeganxis1or2}, it suffices to consider the following two cases.

    \textbf{Case 1}. $|\Omega_n(x)|=2$. Then $\mathcal{I}=\{\i\}$ and $\mathcal{J}=\{\j\}$ for some $\i=i_1\cdots i_n,\j=j_1\cdots j_n\in\D^n$. We have seen that $i_1\neq j_1$. Clearly, if $\vp_{\i}([0,1]^2)$ and $\vp_{\j}([0,1]^2)$ are up-down or left-right adjacent then so are $\vp_{i_1}([0,1]^2)$ and $\vp_{j_1}([0,1]^2)$. Since $x\notin \vp_{\eta}(F)$ for all $\eta\in\D^n\setminus\{\i,\j\}$, $x\notin \vp_{w}(F)$ for all $w\in\D\setminus\{i_1,j_1\}$, i.e., $\Omega_1(x)=\{i_1,j_1\}$. Moreover, $\vp_{i_1}(F)\cap \vp_{j_1}(F)$ is a scaled copy of $\vp_{\i}(F)\cap \vp_{\j}(F)$ so it is also a singleton (this singleton is just $\{x\}$ since it contains $x$). So $\D^1$ is locally fragile and this contradicts the fact that $n>2$.

    If $\vp_{\i}([0,1]^2)$ and $\vp_{\j}([0,1]^2)$ are adjacent but not up-down or left-right adjacent then
    \[
        \vp_{\i}([0,1]^2)\cap\vp_{\j}([0,1]^2) = \vp_{\i}(F)\cap\vp_{\j}(F) = \{x\}.
    \]
    If $\vp_{\i|_{n-1}}([0,1]^2)$ and $\vp_{\j|_{n-1}}([0,1]^2)$ are not up-down or left-right adjacent then they also intersect at exactly one point $x$. Since $x\notin \vp_{\eta}(F)$ for all $\eta\in\D^n\setminus\{\i,\j\}$, $x\notin \vp_{\tau}(F)$ for all $\tau\in\D^{n-1}\setminus\{\i|_{n-1},\j|_{n-1}\}$, i.e., $\Omega_{n-1}(x)=\{\i|_{n-1},\j|_{n-1}\}$. So $\D^{n-1}$ is locally fragile which contradicts the minimality of $n$. So the two squares $\vp_{\i|_{n-1}}([0,1]^2)$ and $\vp_{\j|_{n-1}}([0,1]^2)$ should be either left-right adjacent or up-down adjacent. We may assume the latter. In this case, $\vp_{i_1}([0,1]^2)$ and $\vp_{j_1}([0,1]^2)$ are also up-down adjacent. By the self-similarity, there is a pair of level-$2$ cells, one in $\vp_{i_1}([0,1]^2)$ and another in $\vp_{j_1}([0,1]^2)$, locating and behaving just in the same way as $\vp_{\i}(F)$ and $\vp_{\j}(F)$. Denoting the intersection of these two level-$2$ cells (which is a singleton) by $\{z\}$, it is easy to see that there is a decomposition of $\Omega_2(z)$ making $\D^2$ locally fragile. This contradicts the fact that $n>2$.

    \textbf{Case 2}. $|\Omega_n(x)|=3$. Without loss of generality, we may assume that $\Omega_n(x)=\{\i,\j,\k\}$ and $\mathcal{I}=\{\i,\k\}$, $\mathcal{J}=\{\j\}$, where $\i,\j,\k\in\D^n$. If $\i|_{n-1},\j|_{n-1}$ and $\k_{n-1}$ are pairwise distinct then $x$ is the common vertex of $\vp_{\i|_{n-1}}([0,1]^2), \vp_{\j|_{n-1}}([0,1])^2$ and $\vp_{\k|_{n-1}}([0,1]^2)$. By the self-similarity of $F$ and the position of these squares, the intersection
    \[
        (\vp_{\i|_{n-1}}(F)\cup\vp_{\k|_{n-1}}(F)) \cap \vp_{\j|_{n-1}}(F)
    \]
    coincides with
    \[
        (\vp_{\i}(F) \cup \vp_{\k}(F)) \cap \vp_{\j}(F) = \Big( \bigcup_{\omega\in \mathcal{I}}\vp_{\omega}(F) \Big) \cap \Big( \bigcup_{\omega\in \mathcal{J}}\vp_{\omega}(F) \Big)=\{x\}.
    \]
    Letting $\mathcal{I}'=\{\i|_{n-1},\k|_{n-1}\}$ and $\mathcal{J}'=\{\j|_{n-1}\}$, it is easy to check that $\D^{n-1}$ is locally fragile, which contradicts the minimality of $n$.
    Therefore, at least two of $\i,\j,\k$ share a common prefix of length $n-1$. Combining this with the arguments in the first paragraph of this proof, exactly two of $\i,\j,\k$ share a common prefix of length $n-1$.

    Write $\{\omega,\tau\}=\{\i|_{n-1},\j|_{n-1},\k|_{n-1}\}$. Then $\vp_{\omega}([0,1]^2)$ and $\vp_{\tau}([0,1]^2)$ are either left-right adjacent or up-down adjacent. We may assume that latter.
    Since $\omega|_1\neq \tau|_1$ (otherwise $i_1=j_1=k_1$), $\vp_{\omega|_1}([0,1]^2)$ and $\vp_{\tau|_1}([0,1]^2)$ are also up-down adjacent. By the self-similarity, there exist three level-$2$ cells in $\vp_{\omega|_1}([0,1]^2)$ and $\vp_{\tau|_1}([0,1]^2)$, locating and behaving just in the same way as $\vp_{\i}(F),\vp_{\j}(F)$ and $\vp_{\k}(F)$. So $\D^2$ is locally fragile and we again obtain a contradiction.
\end{proof}

These results establish Theorem~\ref{thm:local} as follows.

\begin{proof}[Proof of Theorem~\ref{thm:local}]
    The ``only if\,'' part follows from Corollary~\ref{cor:local11} and Proposition~\ref{prop:locallast}. The ``if\,'' part follows from Proposition~\ref{prop:local1}.
\end{proof}

\section{Generalization to Bara\'{n}ski carpets and Bara\'{n}ski sponges}

In this section, we will explain why our results on (local) cut points also hold for Bara\'{n}ski carpets. Furthermore, we can extend our results on the connectedness to  Bara\'nski sponges, which can be regarded as the high-dimensional self-affine generalization of GSCs. Thus, by Theorem~\ref{thm:Whyburn}, we characterize when a Bara\'nski carpet is homeomorphic to the standard GSC.

Since the geometrical structure of  Bara\'nski sponges is more complicated than that of GSCs, we will prove our results on connectedness of Bara\'nski sponges in details. On the other hand, we only explain our results on the existence of (local) cut points of Bara\'nski carpets, since the geometrical structure of  Bara\'nski carpets is similar to that of GSCs.



\subsection{Connectedness of Bara\'nski sponges}


First, let us recall the definition of Bara\'nski sponges. Let $d\geq 2$ be an integer. As usual, a vector $\bp=(p_1,\ldots,p_d)\in \R^d$ is called a \emph{probability vector} if $\sum_{j=1}^d p_j=1$ and $p_j>0$ for all $j$.

    Let $N_1,N_2,\ldots,N_d$ be integers such that $N_i\geq 2$ for all $i$.  For each $1\leq i\leq d$, let $\bp_i=(p_{i,1},\ldots,p_{i,N_i})$ be a probability vector, and define another vector $\bq_i=(q_{i,1},\ldots,q_{i,N_i})$ by setting $q_{i,1}=0$ and $q_{i,j}=\sum_{\ell=1}^{j-1} p_{i,\ell}$ for $2\leq j\leq N_i$.

    Let $\D\subset \{(j_1,\ldots,j_d):\, 1\leq j_i\leq N_i \textrm{ for all $1\leq i\leq d$}\}$. To avoid trivial cases, we assume that $1<\# \D< \prod_{i=1}^d N_i$. For each $w=(j_1,\ldots,j_d)\in \D$, define an affine map $\psi_w$ on $\R^d$ by
    \begin{equation}\label{eq:def-spon-psiw}
        \psi_w(x_1,\ldots,x_d)=(\psi_{w,1}(x_1),\ldots,\psi_{w,d}(x_d)),
    \end{equation}
    where $\psi_{w,i}(x_i)=p_{i,j_i}x_i+q_{i,j_i}$, $1\leq i\leq d$.
    Then $\Psi_{\D}:=\{\psi_w:\, w\in \D\}$ is a self-affine IFS on $\R^d$ of which the corresponding attractor $K=K(\Psi_{\D})$ is called a \emph{Bara\'{n}ski sponge}.

In the case when $p_{i,j}=\frac{1}{N_i}$ for all $1\leq j\leq N_i$ and $1\leq i\leq d$, $K$ is usually called a \emph{Bedford-McMullen sponge}. Furthermore, if all $N_i$ are equal, $K$ is called a \emph{Sierpi\'nski sponge}. We remark that in some papers, $K$ is called a sponge only if $d\geq 3$. Recently, Das and Simmons \cite{DS17} proved that for all $d\geq 3$, there exists a Bara\'nski sponge such that its Hausdorff dimension is strictly larger than its dynamical dimension.

In the planar case, the attractor $K=K(\Psi_{\D})$ is usually called a \emph{Bara\'{n}ski carpet} (\cite{Bar07}). Furthermore, in the case when $p_{i,j}=\frac{1}{N_i}$ for $1\leq j\leq N_i$, $K$ is usually called a \emph{Bedford-McMullen carpet} (\cite{Bed84, McMu84}). Recently, there are many discussions on these carpets. Please see \cite{Frs20,LMR20} and references therein.

By a little abuse of notation, we let $Q_0=[0,1]^d$ and recursively define
\[
    Q_n=\bigcup_{w\in \D} \psi_w(Q_{n-1}), \quad n\in \Z^+.
\]
Similarly as in the GSC cases, $\{Q_n\}_{n=0}^\infty$ forms a decreasing sequence of compact sets and $K=\bigcap_{n=0}^\infty Q_n$.
Moreover, we still write $\D^*=\bigcup_{k=1}^\infty \D^k$ and $\psi_{\bw}=\psi_{w_1}\circ\cdots\circ\psi_{w_k}$ for $\bw=w_1\cdots w_k\in\D^*$. Then by~\eqref{eq:def-spon-psiw},
\[
    \psi_{\bw}(x_1,\ldots,x_d)=(\psi_{\bw,1}(x_1),\ldots,\psi_{\bw,d}(x_d)),
\]
where $\psi_{\bw,i}(x_i):=\psi_{w_1,i}\circ \cdots \circ \psi_{w_k,i}(x_i)$, $1\leq i\leq d$. It follows that
\[
    \psi_\bw(Q_0)=\psi_{\bw,1}([0,1]) \times \cdots \times \psi_{\bw,d}([0,1]).
\]
For $1\leq i\leq d$, we define $\pi_i:\, \D^*\to \Z^+\cup\{0\}$ recursively by
\[
    \pi_i(w)=j_i, \quad \textrm{if $w=(j_1,\ldots,j_d)\in\D$,}
\]
and
\begin{equation}\label{eq:piequalsnipii}
    \pi_i(\bw w)=N_i\cdot \pi_i(\bw)+\pi_i(w), \quad \textrm{if $\bw\in \D^*, w\in \D$}.
\end{equation}

Let $\bw,\bw'\in \D^*$ with $|\bw|=|\bw'|$. It is easy to check the following simple facts.
\begin{enumerate}
    \item $\psi_\bw(Q_0)\cap \psi_{\bw'}(Q_0)\not=\varnothing$ if and only if $\psi_{\bw,i}([0,1])\cap \psi_{\bw',i}([0,1])\not=\varnothing$ for all $1\leq i\leq d$. Furthermore, if the above condition holds, then
    \[
        \dim(\psi_{\bw}(Q_0) \cap \psi_{\bw'}(Q_0))=\sum_{i=1}^d\dim\big(\psi_{\bw,i}([0,1]) \cap \psi_{\bw',i}([0,1])\big).
    \]
    \item For $1\leq i\leq d$, $\psi_{\bw,i}=\psi_{\bw',i}$  if and only if $\pi_i(\bw)=\pi_i(\bw')$.
    \item For $1\leq i\leq d$, $\psi_{\bw,i}([0,1]) \cap \psi_{\bw',i}([0,1])\not=\varnothing$ if and only if $|\pi_i(\bw)-\pi_i(\bw')|\leq 1$. Moreover, the intersection is a singleton if $|\pi_i(\bw)-\pi_i(\bw')|= 1$. As a result, if $\psi_{\bw,i}([0,1]) \cap \psi_{\bw',i}([0,1])\not=\varnothing$, then
    \[
        \dim\big(\psi_{\bw,i}([0,1]) \cap \psi_{\bw',i}([0,1])\big) = 1- | \pi_i(\bw)-\pi_i(\bw')|.
    \]
    \item For $1\leq i\leq d$ and $u,u'\in\D$, $|\pi_i(\bw u)-\pi_i(\bw' u')|\geq |\pi_i(\bw)-\pi_i(\bw')|,$
    and the equality holds if and only if $\pi_i(\bw u)-\pi_i(\bw' u')=\pi_i(\bw)-\pi_i(\bw')$, which is equivalent to
    \[
        \pi_i(\bw)-\pi_i(\bw')\in\{0,\pm 1\} \,\,\textrm{and}\,\, \pi_i(u)-\pi_i(u')=-(N_i-1)(\pi_i(\bw)-\pi_i(\bw')).
    \]
\end{enumerate}

Given $\alpha=(\alpha_1,\ldots,\alpha_d)\in\Z^d$, if $|\alpha_i|\leq 1$ for all $1\leq i\leq d$, then we call $\alpha$ a $0$-$\pm 1$ vector and write $|\alpha|=\sum_{i=1}^n|\alpha_i|$. We remark that the notation $|\cdot|$ sometimes refers to the cardinality of sets, but the meaning should be clear from the context.


For $\bw\in\D^*$, define $\pi(\bw)=(\pi_1(\bw),\ldots,\pi_d(\bw))$. From the above Facts (1)-(3), we can easily obtain the following lemma.
\begin{lemma}\label{lem:BarSpon-basic-fact}
  Let $\bw,\bw'\in\D^*$ with $|\bw|=|\bw'|$. Then
  \[
    \psi_{\bw}(Q_0) \cap \psi_{\bw'}(Q_0) \not=\emptyset  \Longleftrightarrow \pi(\bw)-\pi(\bw') \textrm{ is a $0$-$\pm 1$ vector}.
  \]
  Furthermore, if the above condition holds, then
  \[
     \dim(\psi_{\bw}(Q_0) \cap \psi_{\bw'}(Q_0))=d-|\pi(\bw)-\pi(\bw')|.
  \]
\end{lemma}

\begin{lemma}\label{lem:SS2}
    Let $\bw,\bw'\in\D^*$ with $|\bw|=|\bw'|$ and $\psi_{\bw}(Q_0) \cap \psi_{\bw'}(Q_0) \not=\varnothing$. For $u,u'\in \D$,
    \begin{align*}
      &\dim(\psi_{\bw u}(Q_0) \cap \psi_{\bw' u'}(Q_0)) = \dim(\psi_{\bw}(Q_0) \cap \psi_{\bw'}(Q_0)) \\
      \Longleftrightarrow \; & \pi(\bw u)-\pi(\bw' u')=\pi(\bw)-\pi(\bw').
    \end{align*}
\end{lemma}
\begin{proof}
    Notice that $|\pi(\bw)-\pi(\bw')|=\sum_{i=1}^d |\pi_i(\bw)-\pi_i(\bw')|$. Thus, from Fact (4), for any $u,u'\in \D$, we have $|\pi(\bw u)-\pi(\bw' u')|\geq |\pi(\bw)-\pi(\bw')|$, and the equality holds if and only if $\pi_i(\bw u)-\pi_i(\bw' u')= \pi_i(\bw)-\pi_i(\bw')$ for all $1\leq i\leq d$. That is,
    \[
      |\pi(\bw u)-\pi(\bw' u')|=|\pi(\bw)-\pi(\bw')|   \Longleftrightarrow \pi(\bw u)-\pi(\bw' u')=\pi(\bw)-\pi(\bw').
    \]
    Combining this with Lemma~\ref{lem:BarSpon-basic-fact}, it is easy to see that the lemma holds.
\end{proof}

\begin{lemma}\label{lem:BarSpon-conn}
    Let $\bw,\bw'\in \D^*$ with $|\bw|=|\bw'|$ and let $u,u'\in \D$. Assume that
    \[
        \pi(\bw u)-\pi(\bw' u')=\pi(\bw)-\pi(\bw')
    \]
    is a nonzero $0$-$\pm 1$ vector.  Then $\psi_{\bw}(K)\cap \psi_{\bw'}(K)\neq \varnothing$.
\end{lemma}
\begin{proof}
  From the assumption of the lemma and Fact (4),
  \[
    \pi_i(u)-\pi_i(u')=-(N_i-1)(\pi_i(\bw)-\pi_i(\bw')), \quad 1\leq i\leq d.
  \]
  For $k\geq 1$, write $\bu_k=u\cdots u\in\D^k$ and $\bu'_k=u'\cdots u'\in \D^k$. Then, from Fact (4) and by induction, it is easy to see that for all $k\geq 1$ and all $1\leq i\leq d$,
  \[
    \pi_i(\bw \bu_k)-\pi_i(\bw' \bu'_k)=\pi_i(\bw \bu_{k-1})-\pi_i(\bw' \bu'_{k-1})\in \{0,\pm 1\}
  \]
  so that $\psi_{\bw\bu_k}(Q_0) \cap \psi_{\bw'\bu'_k}(Q_0)\not=\varnothing$ for all $k$. Thus $\bigcap_{k=1}^\infty\psi_{\bw\bu_k}(Q_0)=\bigcap_{k=1}^\infty\psi_{\bw'\bu'_k}(Q_0)$, which implies that $\psi_{\bw}(K)\cap \psi_{\bw'}(K)\not=\varnothing$.
\end{proof}

From Fact (3) and Lemma~\ref{lem:SS2}, the assumption in the above lemma is equivalent to
\[
    \psi_{\bw}(Q_0)\cap \psi_{\bw'}(Q_0)\neq \varnothing \,\,\textrm{and}\,\, \dim\big(\psi_{\bw u}(Q_0)\cap \psi_{\bw' u'}(Q_0)\big)=\dim\big(\psi_{\bw}(Q_0)\cap \psi_{\bw'}(Q_0)\big).
\]

\begin{corollary}\label{cor:spon1}
    Let $\alpha=(\alpha_1,\ldots,\alpha_d)$ be a nonzero $0$-$\pm 1$ vector. If $((N_1-1)\alpha_1,\ldots,(N_d-1)\alpha_d)\in\D-\D$, then $\psi_{\bw}(K)\cap\psi_{\bw'}(K)\neq\varnothing$ for all $\bw,\bw'\in \D$ with $|\bw|=|\bw'|$ and $\pi(\bw)-\pi(\bw')=\alpha$.
\end{corollary}
\begin{proof}
    Let $u,u'\in \D$ be such that $u'-u=((N_1-1)\alpha_1,\ldots,(N_d-1)\alpha_d)$. Then by Fact (4),
    \[
      \pi_i(\bw u)-\pi_i(\bw' u')=\pi_i(\bw)-\pi_i(\bw'), \quad 1\leq i\leq d,
    \]
    so $\pi(\bw u)-\pi(\bw' u')=\pi(\bw)-\pi(\bw')=\alpha$. By Lemma~\ref{lem:BarSpon-conn},  $\psi_{\bw}(K)\cap \psi_{\bw'}(K)\neq \varnothing$.
\end{proof}

\begin{lemma}\label{lem:spon2}
    Let $\bw,\bw'\in\D^*$ with $|\bw|=|\bw'|$. If $\psi_{\bw}(Q_d) \cap \psi_{\bw'}(Q_d) \neq \varnothing$, then $\psi_{\bw}(K)\cap \psi_{\bw'}(K)\neq\varnothing$.
\end{lemma}
\begin{proof}
   It is clear that $\psi_{\bw}(K)\cap \psi_{\bw'}(K)\neq\varnothing$ if $\bw=\bw'$. Thus we may assume that $\bw\not=\bw'$. In this case, $\dim(\psi_{\bw}(Q_0)\cap \psi_{\bw'}(Q_0))=d-|\pi(\bw)-\pi(\bw')|\leq d-1$.

   It follows from $Q_d=\bigcup_{\bu\in \D^d} \psi_{\bu}(Q_0)$ and $\psi_{\bw}(Q_d) \cap \psi_{\bw'}(Q_d) \neq \varnothing$ that there exist $\bu=u_1\cdots u_d,\bu'=u_1'\cdots u_d'\in \D^d$ such that $\psi_{\bw\bu}(Q_0)\cap \psi_{\bw'\bu'}(Q_0)\neq \varnothing$. Notice that
   \begin{align*}
       0 &\leq \dim(\psi_{\bw\bu}(Q_0)\cap\psi_{\bw'\bu'}(Q_0)) \\
       &\leq \dim (\psi_{\bw u_1\cdots u_{d-1}}(Q_0)\cap\psi_{\bw' u_1' \cdots u'_{d-1}}(Q_0)) \\
       &\leq \cdots \leq \dim(\psi_{\bw}(Q_0)\cap \psi_{\bw'}(Q_0)) \leq d-1.
   \end{align*}
   Thus there exists $1\leq j\leq d$ such that
   \begin{equation}\label{eq:spon-temp-dimeq}
       \dim(\psi_{\bw u_1\cdots u_j}(Q_0)\cap \psi_{\bw' u_1'\cdots u_j'}(Q_0)) = \dim(\psi_{\bw u_1\cdots u_{j-1}}(Q_0)\cap \psi_{\bw' u_1'\cdots u_{j-1}'}(Q_0)).
   \end{equation}
  Combining this with Lemmas~\ref{lem:SS2} and \ref{lem:BarSpon-conn}, we have $\psi_{\bw u_1\cdots u_{j-1}}(K)\cap \psi_{\bw' u_1'\cdots u_{j-1}'}(K)\neq \varnothing$. So $\psi_{\bw}(K)\cap \psi_{\bw'}(K)\neq \varnothing$.
\end{proof}

\begin{proof}[Proof of Theorem~\ref{thm:main2}]
    Since $K\subset Q_{d+1}$, it suffices to prove the ``$\Longleftarrow$'' part. Assume that $Q_{d+1}$ is connected. Fix any $u,v\in \D$. Since $Q_{d+1}=\bigcup_{w\in \D} \psi_w(Q_d)$ is a finite union of compact sets, we know from the connectedness of $Q_{d+1}$ that there exists $\{w_j\}_{j=1}^m\subset \D$, such that $w_1=u, w_m=v$ and $\psi_{w_j}(Q_d)\cap \psi_{w_{j+1}}(Q_d)\neq \varnothing$ for all $1\leq j\leq m-1$. From Lemma~\ref{lem:spon2}, $\psi_{w_j}(K)\cap \psi_{w_{j+1}}(K)\neq \varnothing$ for all $1\leq j\leq m-1$. Thus the Hata graph of the IFS $\{\psi_w:\, w\in \D\}$ is connected and hence $K$ is connected.
\end{proof}

Similar to GSC cases, it would be nice if we are able to present an effective algorithm to draw the Hata graph sequencee of Bara\'nski sponges by computer. The following result is a preparation for this purpose.

\begin{proposition}\label{prop:spon-intersection}
    Assume that $\bw,\bw'\in\D^*$, where $|\bw|=|\bw'|$ and $\alpha:=\pi(\bw)-\pi(\bw')$ is a nonzero $0$-$\pm 1$ vector. Then $\psi_{\bw}(K)\cap \psi_{\bw'}(K)\neq\varnothing$ if and only if there exist $\bu,\bu'\in \D^*\cup\{\vartheta\}$ with $|\bu|=|\bu'|\leq d-|\alpha|$ and $v,v'\in \D$ such that
    \begin{equation}\label{eq:new6-3}
        \pi(\bw\bu v)-\pi(\bw' \bu' v')=\pi(\bw\bu)-\pi(\bw'\bu').
    \end{equation}
\end{proposition}
\begin{proof}
    First we prove the ``$\Longleftarrow$'' part. From Fact (4) and \eqref{eq:new6-3}, $\pi(\bw\bu)-\pi(\bw'\bu')$ is a $0$-$\pm 1$ vector. On the other hand, since $\pi(\bw)-\pi(\bw')$ is nonzero, we know from Fact (4) that $\pi(\bw\bu)-\pi(\bw'\bu')$ is also nonzero. Thus from Lemma~\ref{lem:BarSpon-conn} and \eqref{eq:new6-3}, $\psi_{\bw\bu}(K)\cap \psi_{\bw'\bu'}(K)\neq \varnothing$ so that $\psi_{\bw}(K)\cap \psi_{\bw'}(K)\neq \varnothing$.


    For the ``$\Longrightarrow$'' part, write $t=d-|\alpha|+1$ (so $1\leq t\leq d$).
    It follows from $K\subset Q_t=\bigcup_{\bu\in \D^t} \psi_{\bu}(Q_0)$ and $\psi_{\bw}(K)\cap \psi_{\bw'}(K)\neq \varnothing$ that there exist $u_1\cdots u_{t}, u'_1\cdots u'_t\in \D^t$, such that $\psi_{\bw u_1\cdots u_t}(Q_0) \cap \psi_{\bw' u'_1\cdots u'_t}(Q_0)\neq \varnothing.$
    Thus, applying similar arguments as in the proof of Lemma~\ref{lem:spon2} and noticing that
    \[
      \dim(\psi_{\bw}(Q_0)\cap \psi_{\bw'}(Q_0))=d-|\alpha|=t-1,
    \]
    there exists $1\leq j\leq t$ such that \eqref{eq:spon-temp-dimeq} holds. Write $\bu=u_1\cdots u_{j-1}, \bu'=u_1'\cdots u_{j-1}'$ and $v=u_j,v'=u_j'$. Then
    \[
      \dim(\psi_{\bw\bu v}(Q_0)\cap \psi_{\bw'\bu' v'}(Q_0))= \dim(\psi_{\bw\bu}(Q_0)\cap \psi_{\bw'\bu'}(Q_0)).
    \]
    Since $\psi_{\bw\bu}(Q_0)\cap \psi_{\bw'\bu' }(Q_0)\supset \psi_{\bw u_1\cdots u_t}(Q_0) \cap \psi_{\bw' u'_1\cdots u'_t}(Q_0)\neq \varnothing$, we know from the remark after Lemma~\ref{lem:BarSpon-conn} that $\pi(\bw\bu v)-\pi(\bw'\bu' v')=\pi(\bw\bu)-\pi(\bw'\bu')$.
\end{proof}

Let $\E_0$ be the set of all nonzero $0$-$\pm 1$ vectors $(\alpha_1,\ldots,\alpha_d)$ such that
\begin{equation}\label{eq:E0-def}
    ((N_1-1)\alpha_1,\ldots,(N_d-1)\alpha_d) \in \D-\D.
\end{equation}
Recursively, for $1\leq t\leq d-1$, we define $\E_t$ to be the set of all nonzero $0$-$\pm 1$ vectors $(\alpha_1,\ldots,\alpha_d)$ that satisfy the following condition: There exists $e\in \D-\D$ such that
\[
    (N_1\alpha_1,\ldots,N_d\alpha_d)+e\in\bigcup_{k=0}^{t-1} \E_k.
\]

By definition, it is clear that $\E_{t}\subset \E_{t+1}$ for $1\leq t\leq d-2$. Furthermore, we have the following simple lemma.
\begin{lemma}\label{lem:Spon-Algo-Recu}
    Let $\bw,\bw'\in \D^*$ be two distinct words with $|\bw|=|\bw'|$. If there exist $u,u'\in\D$ such that $\pi(\bw u)-\pi(\bw' u')\in \E_t$ for some $0\leq t\leq d-2$, then $\pi(\bw)-\pi(\bw')\in \E_{t+1}$.
\end{lemma}
\begin{proof}
    Write $\alpha=(\alpha_1,\ldots,\alpha_d)=\pi(\bw)-\pi(\bw')$. Since $\pi(\bw u)-\pi(\bw' u')\in \E_t$,
    \[
        \pi_i(\bw u)-\pi_i(\bw' u')\in \{0,\pm 1\}, \quad 1\leq i\leq d.
    \]
    Thus from Fact (4), $\pi_i(\bw)-\pi_i(\bw')\in \{0,\pm 1\}$ for all $1\leq i\leq d$. Combining this with $\bw\neq \bw'$, we know that $\pi(\bw)-\pi(\bw')$ is a nonzero $0$-$\pm 1$ vector.

    Note that by~\eqref{eq:piequalsnipii},
    \[
        (N_1\alpha_1,\ldots,N_d\alpha_d)+u-u' = \pi(\bw u)-\pi(\bw'u') \in \E_t.
    \]
    Since $u-u'\in\D-\D$, $\pi(\bw)-\pi(\bw')=\alpha\in \E_{t+1}$.
\end{proof}

The following theorem provides us with an effective algorithm to draw the Hata graph sequence of Bara\'nski sponges.

\begin{theorem}\label{thm:khataofbaranski}
  Let $\bw,\bw'\in \D^*$ be such that $|\bw|=|\bw'|$ and $\pi(\bw)-\pi(\bw')$ is a nonzero $0$-$\pm 1$ vector. Then $\psi_{\bw}(K)\cap \psi_{\bw'}(K)\neq \varnothing$ if and only if $\pi(\bw)-\pi(\bw')\in \E_{d-|\pi(\bw)-\pi(\bw')|}$.
\end{theorem}

\begin{proof}
  Write $\alpha=(\alpha_1,\ldots,\alpha_d)=\pi(\bw)-\pi(\bw')$ and let $t=d-|\alpha|$ (so $0\leq t\leq d-1$). We first prove the ``if\," part by induction on $t$. By definition, if $\alpha=\pi(\bw)-\pi(\bw')\in \E_0$, then \eqref{eq:E0-def} holds. Since $\alpha$ is a nonzero $0$-$\pm 1$ vector, we have by Corollary~\ref{cor:spon1} that $\psi_{\bw}(K)\cap \psi_{\bw'}(K)\neq \varnothing$. Thus the ``if\," part holds for $t=0$.

  Assume that the ``if\," part holds for $0\leq t\leq m$, where $0\leq m\leq d-2$ is an integer. When $t=m+1$ and $\alpha=\pi(\bw)-\pi(\bw')\in \E_{m+1}$, there exist $v,v'\in \D$ such that
  \[
    (N_1 \alpha_1,\ldots,N_d \alpha_d)+(v-v')\in \bigcup_{k=0}^m \E_{k}.
  \]
  So $\pi(\bw v)-\pi(\bw' v')\in \bigcup_{k=0}^m \E_k$ (recall~\eqref{eq:piequalsnipii}). By the inductive hypothesis, $\psi_{\bw v}(K)\cap \psi_{\bw' v'}(K)\neq \varnothing$, which implies that $\psi_{\bw}(K)\cap \psi_{\bw'}(K)\neq \varnothing$. Thus the ``if\," part holds for $t=m+1$.

  Now we prove the ``only if\," part. Suppose $\psi_{\bw}(K)\cap \psi_{\bw'}(K)\neq \varnothing$. From Proposition~\ref{prop:spon-intersection}, there exist $\bu,\bu'\in \D^*\cup\{\vartheta\}$ with $|\bu|=|\bu'|\leq t$ and $v,v'\in \D$ such that \eqref{eq:new6-3} holds. Letting $\beta=\pi(\bw\bu)-\pi(\bw'\bu')$, we can rewrite \eqref{eq:new6-3} as
  \begin{equation}\label{eq:new6-5}
    (N_1\beta_1,\ldots,N_d\beta_d)+v-v'=(\beta_1,\ldots,\beta_d),
  \end{equation}
  which implies that $((N_1-1)\beta_1,\ldots,(N_d-1)\beta_d)=v'-v\in \D-\D$ and $\beta$ is a $0$-$\pm 1$ vector. Since $\pi(\bw')-\pi(\bw)$ is nonzero, from Fact (4), $\beta=\pi(\bw\bu)-\pi(\bw'\bu')$ is also nonzero. Thus $\beta\in \E_0$.

  Let $s=|\bu|=|\bu'|$. If $s=0$, then $\bw\bu=\bw,\bw'\bu'=\bw'$ and hence $\pi(\bw)-\pi(\bw')=\beta\in \E_0$.
  Furthermore, from \eqref{eq:new6-5}, $(N_1\beta_1,\ldots,N_d\beta_d)+v-v'=(\beta_1,\ldots,\beta_d)\in \E_0$  so that $\beta\in \E_1$. Thus $\beta\in \E_0\cap\E_1 \subset \E_t$.

  Now assume that $s\geq 1$. Write $\bu=u_1\cdots u_s$ and $\bu'=u'_1 \cdots u'_s$. From $\pi(\bw\bu)-\pi(\bw'\bu')=\beta\in \E_0$ and Lemma~\ref{lem:Spon-Algo-Recu}, we see that $\pi(\bw u_1\cdots u_{s-1})-\pi(\bw' u'_1\cdots u'_{s-1})\in \E_1$. Applying Lemma~\ref{lem:Spon-Algo-Recu} repeatedly, we obtain that $\pi(\bw)-\pi(\bw')\in \E_s\subset\E_t$ since $t\geq s\geq 1$.
\end{proof}

\subsection{Existence of local cut points of Bara\'nski carpets}
In this subsection, we will explain why our results on the existence of (local) cut points of GSCs are also applicable to Bara\'nski carpets.

Assume that $K=K(\Psi_{\D})$ is a connected Bara\'nski carpet. Similarly as in the GSC cases, we call $K$ \emph{fragile} if we can decompose $\D$ as $\D=\D_1\cup \D_2$ such that
\[
    \Big(\bigcup_{w\in \D_1} \psi_w(K) \Big) \cap \Big(\bigcup_{w\in \D_2} \psi_w(K) \Big)
\]
is a singleton. The carpet $K$ is called \emph{non-fragile} if it is not fragile.

Since it makes no essential difference when we consider the arguments in Sections 2-5 with small squares replaced by small rectangles, one can modify those arguments and show the following results.

\begin{enumerate}
    \item If $K$ is fragile, then $K$ has cut points and $\chi(\Gamma_k)\geq |\D|^{k-1}-1$ for all $k\geq 2$, where $\Gamma_k$ is the $k$-th Hata graph of $K$ (analogous to Lemma~\ref{lem:local1} and Proposition~\ref{prop:fragilelower});
    \item A non-fragile connected Bara\'nski carpet has cut points if and only if $\chi(\Gamma_k)\geq |\D|^{k-1}$ for all $k\geq 2$ (analogous to Theorem~\ref{thm:nonfragile});
    \item A connected Bara\'nski carpet has local cut points but no cut points if and only if there are disjoint subsets $\mathcal{I},\mathcal{J}\subset\D^n$ as in Theorem~\ref{thm:local}.
\end{enumerate}

In conclusion, we can extend Theorems~\ref{thm:main} and~\ref{thm:local} as follows.


\begin{theorem}
    A connected Bara\'nski carpet $K(\Psi_{\D})$ has cut points if and only if $\chi(\Gamma_k)\geq |\D|^{k-1}-1$ for all $k\geq 2$.
\end{theorem}

\begin{theorem}
    Let  $K=K(\Psi_{\D})$ be a connected Bara\'nski carpet with no cut points. Then $K$ contains local cut points if and only if there are disjoint subsets $\mathcal{I},\mathcal{J}$ of $\D^n$ for $n=1$ or $n=2$ such that
    \[
       \Big( \bigcup_{\i\in \mathcal{I}}\psi_{\i}(K) \Big)\cap \Big( \bigcup_{\i\in \mathcal{J}}\psi_{\i}(K) \Big)
    \]
    is a singleton, say $\{x\}$, and $\mathcal{I}\cup \mathcal{J}=\{\i\in\D^n: x\in\psi_{\i}(K)\}$.
\end{theorem}

\section{Further remarks}
\subsection{A gap phenomenon}

Since any planar set with Hausdorff dimension $<1$ is totally disconnected (see~\cite[Proposition 3.5]{Fal14}), a GSC $F=F(N,\D)$ is totally disconnected if $|\D|<N$. When $|\D|=N$, $F$ is connected if and only if $\D$ is of one of the following form:
\begin{enumerate}
    \item $\D=\{(i,j): 0\leq j\leq n-1\}$ for some $0\leq i\leq n-1$;
    \item $\D=\{(i,j): 0\leq i\leq n-1\}$ for some $0\leq j\leq n-1$;
    \item $\D=\{(i,i): 0\leq i\leq n-1\}$;
    \item $\D=\{(i,n-i): 0\leq i\leq n-1\}$.
\end{enumerate}
In these cases, $F$ is a line segment. One can also show that $F$ is not only disconnected but also totally disconnected if $|\D|=N$ but $\D$ is not of one of the above four forms. For details, please refer to~\cite{RXpre}. Then a natural question arises: \emph{For any $k> N$, is there a digit set $\D$ with exactly $k$ elements such that the associated GSC is connected?} The following theorem indicates a gap phenomenon.

\begin{theorem}\label{thm:mainlast}
    Let $F=F(N,\D)$ be a connected GSC with $|\D|>N$. Then $|\D|\geq 2N-1$.
\end{theorem}
\begin{proof}
The proof is divided into four cases.

\textbf{Case 1}. $\{(0,0),(0,N-1),(N-1,0),(N-1,N-1)\}\subset\D$. In this case, a pair of level-$1$ cells has a non-empty intersection if and only if the correpsonding level-$1$ squares are adjacent. More precisely, for $(a_1,b_1),(a_2,b_2)\in\D$,
\begin{equation}\label{eq:gap1}
    \vp_{(a_1,b_1)}(F)\cap \vp_{(a_2,b_2)}(F)\neq\varnothing \Longleftrightarrow \max\{|a_1-a_2|,|b_1-b_2|\}\leq 1.
\end{equation}
Since $F$ is connected, the associated Hata graph is connected. Therefore, we can find two paths in that graph: $(0,0)=i_1\to i_2\to\cdots\to i_m=(N-1,N-1)$ and $(0,N-1)=j_1\to j_2\to\cdots\to j_n=(N-1,0)$. It follows from \eqref{eq:gap1} that $m\geq N$ and $n\geq N$. If these two paths do not share a common vertex then $|\D|\geq 2N$, so we may suppose the contrary. Letting
\begin{equation}\label{eq:pq-def}
    p=\min\{1\leq t\leq m: i_t\in\{j_1,\ldots,j_n\}\}, \quad p'=\max\{1\leq t\leq m: i_t\in\{j_1,\ldots,j_n\}\},
\end{equation}
we see that $i_1,\ldots,i_{p-1},i_{p'+1},\ldots,i_m,j_1,\ldots,j_n$ are distinct elements in $\D$, implying that
\begin{equation}\label{eq:DCard-pq}
    |\D|\geq (p-1)+(m-p')+n.
\end{equation}
For convenience, denote $q,q'$ and $a,b,c,d$ be such that $j_{q}=i_p=(a,b)$ and $j_{q'}=i_{p'}=(c,d)$. We may also assume that $q\leq q'$ (the case when $q>q'$ can be similarly discussed). Consider the following five paths:
\begin{equation}\label{eq:fivepath}
    \begin{gathered}
        (0,0)=i_1\to\cdots\to i_p=(a,b),\quad (c,d)=i_{p'}\to\cdots\to i_m=(N-1,N-1), \\
        (0,N-1)=j_1\to\cdots\to j_q=(a,b),\quad (a,b)=j_q\to\cdots\to j_{q'}=(c,d),\\
        (c,d)=j_{q'}\to\cdots\to j_n=(N-1,0).
    \end{gathered}
\end{equation}
By~\eqref{eq:gap1}, we have the following estimates of their lengths:
\begin{equation*}
    \begin{gathered}
        p\geq \max\{a,b\}+1, \quad m-p'+1\geq \max\{N-1-c,N-1-d\}+1, \\
        q\geq \max\{a,N-1-b\}+1, \quad q'-q+1\geq \max\{|c-a|,|d-b|\}+1,\\
        n-q'+1\geq \max\{N-1-c,d\}+1.
    \end{gathered}
\end{equation*}
As a consequence,
\begin{align*}
    |\D| &\geq (p-1)+(m-p')+n \\
    & = (p-1)+(m-p')+q+(q'-q)+(n-q') \\
    &\geq b+(N-1-d)+(N-1-b+1)+|c-a|+d \geq 2N-1.
\end{align*}



\textbf{Case 2}. $\{(0,0),(N-1,N-1)\}\nsubseteq\D$ and $\{(0,N-1),(N-1,0)\}\nsubseteq\D$. In this case, we have for all $(a_1,b_1),(a_2,b_2)\in\D$ that
\begin{equation}\label{eq:gap2}
    \vp_{(a_1,b_1)}(F)\cap \vp_{(a_2,b_2)}(F)\neq\varnothing \Longrightarrow |a_1-a_2|+|b_1-b_2|=1,
\end{equation}
i.e., the two correpsonding level-$1$ squares are up-down or left-right adjacent. Furthermore,
\begin{align*}
    \vp_{(a,b)}(F)\cap \vp_{(a+1,b)}(F)\neq\varnothing \Longleftrightarrow \exists y_0 \text{ such that } (0,y_0),(N-1,y_0)\in\D, \\
    \vp_{(a,b)}(F)\cap \vp_{(a,b+1)}(F)\neq\varnothing \Longleftrightarrow \exists x_0 \text{ such that } (x_0,0),(x_0,N-1)\in\D.
\end{align*}
Since $F$ is connected and $|\D|>N$, it is not difficult to see that both $x_0$ and $y_0$ as above must exist. Applying an analogous argument as in Case 1 gives us the desired result. But for readers' convenience, we present the detailed proof here. Note that there are two paths $(x_0,0)=i_1\to i_2\to\cdots\to i_m=(x_0,N-1)$ and $(0,y_0)=j_1\to j_2\to\cdots\to j_n=(N-1,y_0)$ in the Hata graph. It follows from~\eqref{eq:gap2} that $m\geq N$ and $n\geq N$. If these two paths do not share a common vertex then $|\D|\geq 2N$, so we may suppose the contrary. Defining $p$ and $q$ as in \eqref{eq:pq-def}, we see that \eqref{eq:DCard-pq} holds.
Choose $q,q'$ and $a,b,c,d$ as before, i.e., $j_{q}=i_p=(a,b)$ and $j_{q'}=i_{p'}=(c,d)$ and again assume that $q\leq q'$. Replacing $(0,0)$, $(N-1,N-1)$, $(0,N-1)$ and $(N-1,0)$ in \eqref{eq:fivepath} by $(x_0,0)$, $(x_0,N-1)$, $(0,y_0)$ and $(N-1,y_0)$, respectively, we obtain five paths.
By~\eqref{eq:gap2}, we have the following estimates of their lengths:
\begin{equation*}
    \begin{gathered}
        p\geq |a-x_0|+b+1, \quad m-p'+1\geq |c-x_0|+(N-1-d)+1, \\
        q\geq a+|b-y_0|+1, \quad q'-q+1\geq |c-a|+|d-b|+1,\\
        n-q'+1\geq (N-1-c)+|d-y_0|+1.
    \end{gathered}
\end{equation*}
As a consequence,
\begin{align*}
    |\D|
    &\geq (p-1)+(m-p')+q+(q'-q)+(n-q') \\
    &\geq (|a-x_0|+b)+(|c-x_0|+N-1-d)+(a+|b-y_0|+1)\\
    &\quad\quad +(|c-a|+|d-b|)+(N-1-c+|d-y_0|) \\
    &\geq (x_0-a+b)+(c-x_0+N-1-d)+(a+y_0-b+1) \\
    &\quad\quad +(|c-a|+|d-b|)+(N-1-c+d-y_0) \geq 2N-1.
\end{align*}


\textbf{Case 3}. $\{(0,0),(N-1,N-1)\}\subset\D$ while $\{(0,N-1),(N-1,0)\}\nsubseteq\D$. Since $|\D|>N$, at least one of the following subcases holds.

\textbf{Subcase 3.1}. If $i,j\in\D$ are such that $\vp_i([0,1]^2)$ and $\vp_j([0,1]^2)$ are left-right adjacent then $\vp_i(F)\cap \vp_j(F)\neq\varnothing$. More precisely, if $i-j=(\pm 1,0)$ then $i,j$ are adjacent in the Hata graph. This requires that there is some $y_0$ such that $(0,y_0)\in\D$ and either $(N-1,y_0)$ or $(N-1,y_0-1)\in\D$. Here we only consider the former case since the latter one can be similarly discussed. Then for every pair of $(a_1,b_1),(a_2,b_2)\in\D$, every path in the Hata graph connecting them has a length at least
\begin{equation}\label{eq:gap3}
    \left\{
        \begin{array}{ll}
            \max\{|a_1-a_2|,|b_1-b_2|\}+1, & (a_1-a_2)(b_1-b_2)\geq 0, \\
            |a_1-a_2|+|b_1-b_2|+1, & (a_1-a_2)(b_1-b_2)<0.
        \end{array}
    \right.
\end{equation}

Note that there are two paths $(0,0)=i_1\to i_2\to\cdots\to i_m=(N-1,N-1)$ and $(0,y_0)=j_1\to j_2\to\cdots\to j_n=(N-1,y_0)$ in the Hata graph. It follows from~\eqref{eq:gap3} that $m\geq N$ and $n\geq N$ and hence we may assume similarly as before that these two paths share at least one common vertex. Defining $p$ and $q$ as in \eqref{eq:pq-def}, we see that \eqref{eq:DCard-pq} holds.
Again denote $q,q'$ and $a,b,c,d$ be such that $j_{q}=i_p=(a,b)$ and $j_{q'}=i_{p'}=(c,d)$ and again assume that $q\leq q'$. Replacing $(0,N-1)$ and $(N-1,0)$ in \eqref{eq:fivepath} by $(0,y_0)$ and $(N-1,y_0)$, respectively, we obtain five paths.
First note that \eqref{eq:gap3} implies that
\[
    p\geq b+1, \quad m-p'+1\geq (N-1-d)+1, \quad q'-q+1\geq c-a+1,
\]
\[
          q\geq \left\{\begin{array}{ll} \max\{a,b-y_0\}+1, & b\geq y_0, \\ a+(y_0-b)+1, & b<y_0, \end{array}\right. 
\]
and
\[
        n-q'+1\geq \left\{\begin{array}{ll} \max\{N-1-c,y_0-d\}+1, & d\leq y_0, \\ (N-1-c)+(d-y_0)+1, & d>y_0. \end{array}\right.
\]
Thus
\[
   (p-1)+(m-p')+(q'-q)\geq b+ (N-1-d)+(c-a)= N-1+b-d+c-a.
\]
When $b\geq y_0$ and $d\leq y_0$,
\begin{align*}
    |\D| &\geq (p-1)+(m-p')+(q'-q)+q+(n-q') \\
    &\geq (N-1+b-d+c-a)+(a+1)+(N-1-c) \\
    &= 2N-1+(b-d) \geq 2N-1.
\end{align*}
When $b\geq y_0$ and $d>y_0$,
\begin{align*}
    |\D| &\geq (N-1+b-d+c-a)+ (a+1)+(N-1-c+d-y_0) \\
    &= 2N-1+(b-y_0) \geq 2N-1.
\end{align*}
When $b< y_0$ and $d\leq y_0$,
\begin{align*}
    |\D| &\geq (N-1+b-d+c-a)+(a+y_0-b+1)+(N-1-c) \\
    &= 2N-1+(y_0-d) \geq 2N-1.
\end{align*}
When $b< y_0$ and $d > y_0$,
\begin{align*}
    |\D| &\geq (N-1+b-d+c-a)+(a+y_0-b+1)+(N-1-c+d-y_0) \\
    &= 2N-1.
\end{align*}

\textbf{Subcase 3.2}. If $i,j\in\D$ are such that $\vp_i([0,1]^2)$ and $\vp_j([0,1]^2)$ are up-down adjacent then $\vp_i(F)\cap \vp_j(F)\neq\varnothing$. We can apply an analogous argument as in Subcase 3.1 to obtain the desired estimate.

\textbf{Case 4}. $\{(0,0),(N-1,N-1)\}\nsubseteq\D$ while $\{(0,N-1),(N-1,0)\}\subset\D$. This can be similarly discussed as Case 3.
%
\end{proof}

With the above theorem in hand, the next example establishes Theorem~\ref{thm:gap}.

\begin{example}
   Let $N\geq 2$. For every $2N-1\leq k\leq N^2-1$, there is a digit set $\D$ with $|\D|=k$ such that the associated GSC $F=F(N,\D)$ is connected. For example, write $k=N+m(N-1)+p$ where $m\in\{1,\ldots,N\}$, $k=p+1\pmod{N-1}$ and $0\leq p< N-1$ and let
    \begin{align*}
        \D=&\{(0,j): 0\leq j\leq N-1\} \\
        & \cup\{(i,j): 1\leq i\leq N-1, 0\leq j\leq m\} \cup \{(i,m+1): 1\leq i\leq p\}.
    \end{align*}
    Figure~\ref{fig:exalast} depicts a GSC where $N=4$ and $|\D|=11$.
    \begin{figure}[htbp]
        \centering
        \includegraphics[width=3.8cm]{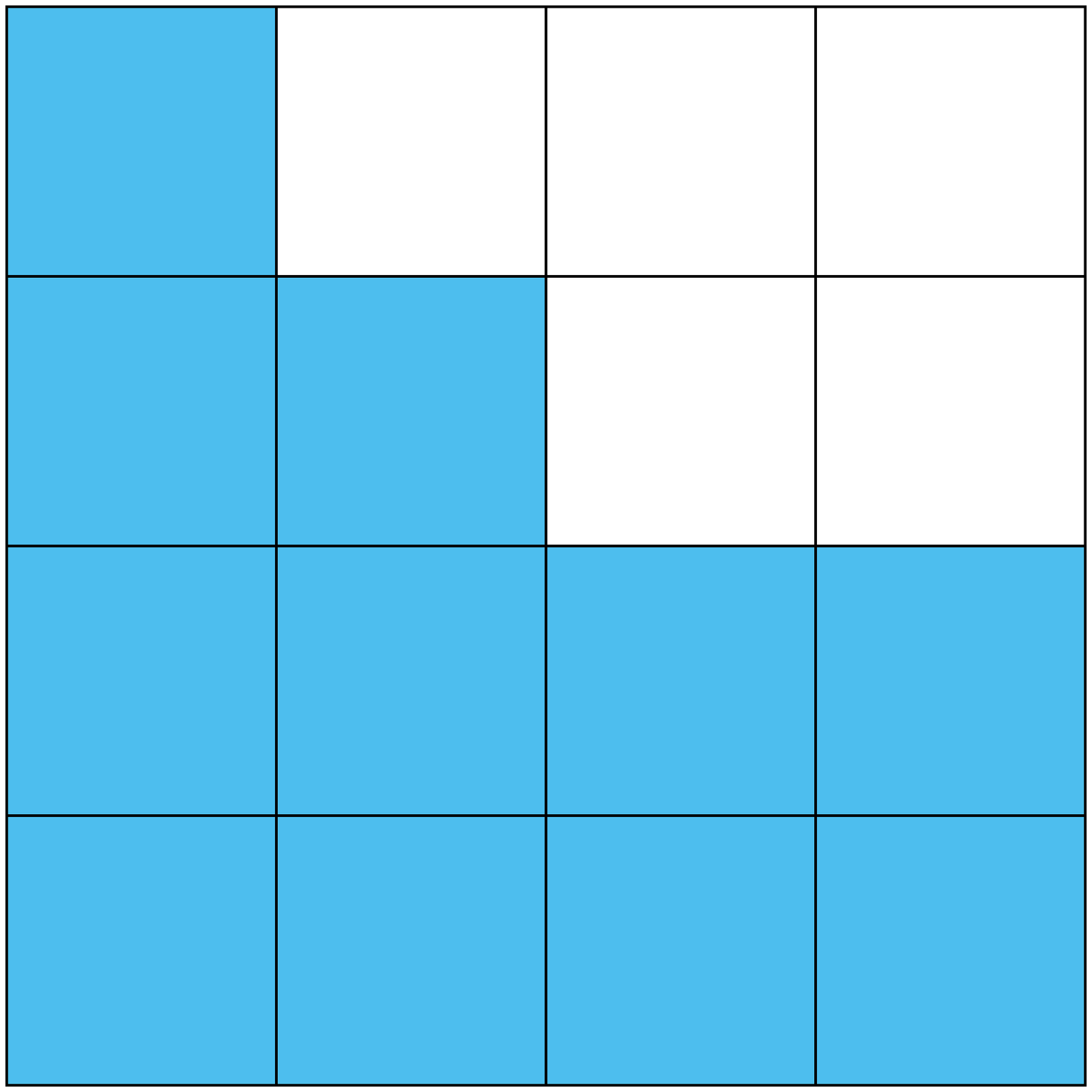} \qquad
        \includegraphics[width=3.8cm]{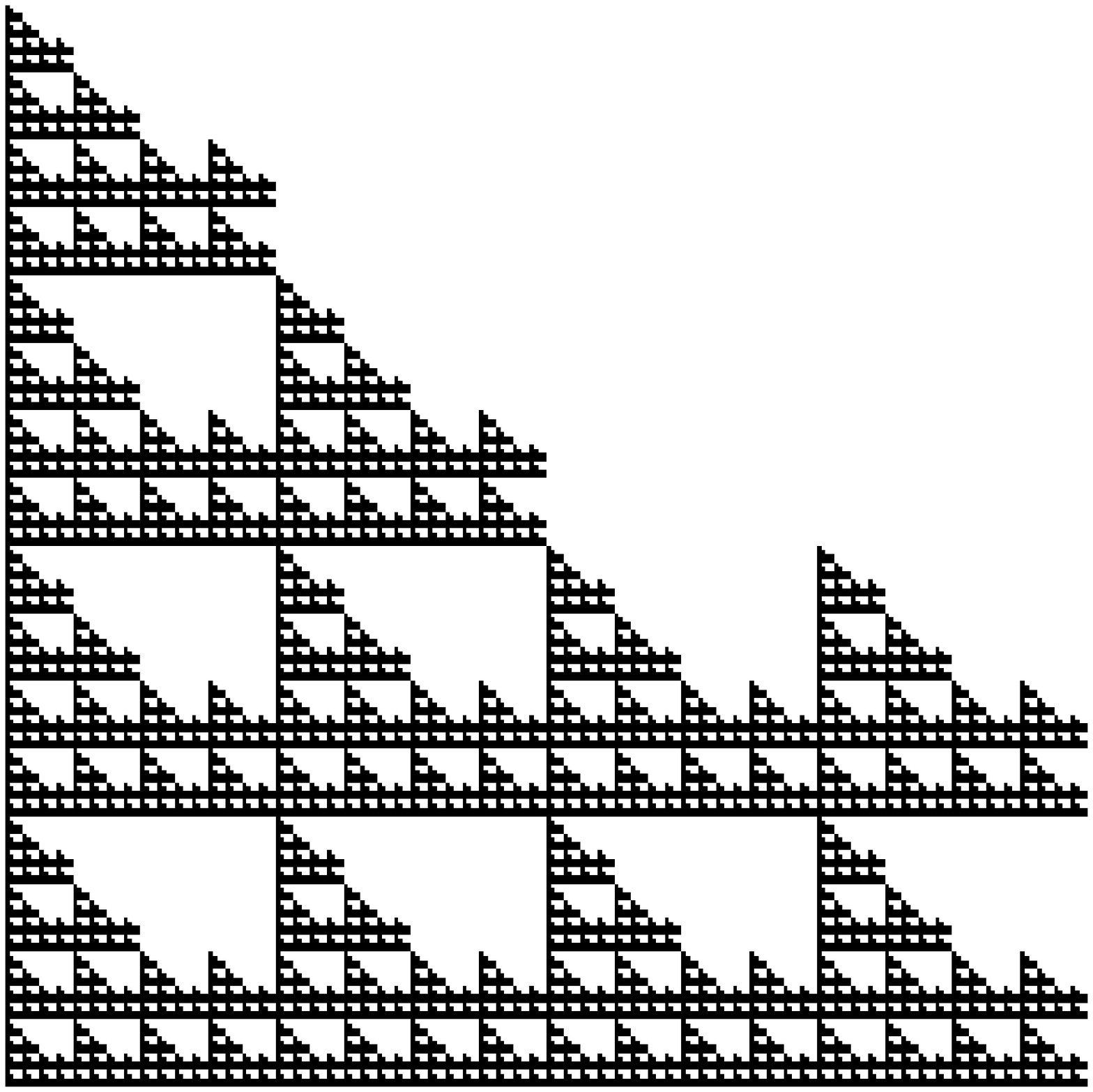}
        \caption{A connected GSC $F=F(N,\D)$ with $N=4$ and $|\D|=11$}
        \label{fig:exalast}
    \end{figure}
\end{example}

\subsection{An improvement on the lower bound of Theorem~\ref{thm:nonfragile}}

With some effort, we are able to make a small improvement on $\chi(\Gamma_k)$ for non-fragile connected GSCs with cut points as in Proposition~\ref{prop:chi-betterBd}.
In order to prove this proposition, we need the following lemma.




\begin{lemma}\label{lem:simplefact}
    Let $u\in\D$ and let $i_1i_2i_3\in\D^3$ be such that $i_1\neq u$. If $\vp_{i_1i_2i_3}(F)$ is the only one level-$3$ cell in $\bigcup_{j\in\D\setminus\{u\}}\vp_j(F)$ which intersects $\vp_u(F)$, then $F$ is fragile.
\end{lemma}
\begin{proof}
    Clearly, $\vp_{i_1}(F)$ is the only one level-$1$ cell in $\bigcup_{j\in\D\setminus\{u\}}\vp_j(F)$ which intersects $\vp_u(F)$. From the first part of the proof of Proposition~\ref{prop:1and3}, we see that $\vp_{i_1}(F)\cap\vp_{u}(F)$ is a singleton. Thus
    \[
        \vp_u(F) \cap \Big( \bigcup_{j\in\D\setminus\{u\}}\vp_j(F) \Big)
    \]
    is a singleton. So $F$ is fragile.
\end{proof}

\begin{proof}[Proof of Proposition~\ref{prop:chi-betterBd}]
Assume on the contrary that $\chi(\Gamma_k)\leq |\D|^{k-1}+|\D|^{k-3}-1$ for some $k\geq 3$ and let $i_1\cdots i_k$ be the cut vertex achieving $\chi(\Gamma_k)$. Then there is some $u\in\D$ such that the following two statements hold:
\begin{enumerate}
    \item $u\D^{k-1}$ and $\{j\D^{k-1}:j\neq i_1,u\}$ belong to different connected components of $\Gamma_k-\{i_1\cdots i_k\}$;
    \item $u\D^{k-1}$ and $\j\D^{k-3}$ belong to different connected components of $\Gamma_k-\{i_1\cdots i_k\}$ for all $\j\in\D^3\setminus\{i_1i_2i_3\}$.
\end{enumerate}
As a consequence, $u\D^{2}$ and $\{j\D^{2}:j\neq i_1,u\}$ belong to different connected components of $\Gamma_3-\{i_1i_2i_3\}$, and $u\D^2$ and $\j$ belong to different components for all $\j\in\D^3\setminus\{i_1i_2i_3\}$. Equivalently, $\vp_{i_1i_2i_3}(F)$ is only one level-$3$ cell in $\bigcup_{j\in\D\setminus\{u\}}\vp_j(F)$ which intersects $\vp_{u}(F)$. It then follows from Lemma~\ref{lem:simplefact} that $F$ is fragile, which leads to a contradiction.
\end{proof}

\bigskip

\noindent{\bf Acknowledgements.}
The work of Dai is supported in part by NSFC grants 11771457 and 11971500.
The work of Luo is supported in part by NSFC grant 11871483.
The work of Ruan is supported in part by NSFC grant 11771391, ZJNSF grant LY22A010023 and the Fundamental Research Funds
for the Central Universities of China grant 2021FZZX001-01. The work of Wang
is supported in part by the Hong Kong Research Grant Council grants 16308518 and 16317416 and
HK Innovation Technology Fund ITS/044/18FX, as well as Guangdong-Hong Kong-Macao Joint
Laboratory for Data-Driven Fluid Mechanics and Engineering Applications. We are grateful to Professor Christoph Bandt for helpful discussions.

\small
\bibliographystyle{amsplain}

\end{document}